\documentclass[12pt,a4paper]{amsart}
\usepackage{tikz}
\usepackage{pdflscape}
\usepackage{amsmath}
\usepackage{amsfonts}
\usepackage{amssymb}
\usepackage{mathtools}
\usepackage[all]{xy}
\usepackage{color}
\usetikzlibrary{arrows}
\usepackage{float}
\usepackage[shortlabels]{enumitem}
\usepackage{ifpdf}
\ifpdf
  \usepackage[colorlinks,
            linkcolor=magenta,       
            anchorcolor=magenta,     
            citecolor=magenta,       
            final,
            hyperindex]{hyperref}
\else
  \usepackage[colorlinks,
            linkcolor=magenta,       
            anchorcolor=magenta,     
            citecolor=magenta,       
            final,
            hyperindex,driverfallback=hypertex]{hyperref}
\fi
\newtheorem{theorem}{Theorem}[section]
\newtheorem{lem}[theorem]{Lemma}
\newtheorem{coro}[theorem]{Corollary}
\newtheorem{problem}[theorem]{Problem}

\newtheorem{prop}[theorem]{Proposition}
\theoremstyle{definition}
\newtheorem{defn}[theorem]{Definition}
\newtheorem{remark}[theorem]{Remark}
\newtheorem{ex}[theorem]{Example}

\newcommand{\delete}[1]{}

\begin{document}
\title[Set-theoretic solutions of the Yang-Baxter equation  and regular $\star$-semibraces]{Set-theoretic solutions of the Yang-Baxter equation  and regular $\star$-semibraces}

\author{Qianxue Liu}
\address{School of Mathematics, Yunnan Normal University, Kunming, Yunnan 650500, China}
\email{2521782539@qq.com}

\author{Shoufeng Wang$^{\ast}$}\thanks{*Corresponding author}
\address{School of Mathematics, Yunnan Normal University, Kunming, Yunnan 650500, China}
\email{wsf1004@163.com}


\begin{abstract}
As generalizations of inverse semibraces introduced by Catino, Mazzotta and Stefanelli, Miccoli has introduced regular $\star$-semibraces under the name   of involution semibraces and given a sufficient condition under which the associated map  to a regular $\star$-semibrace  is a  set-theoretic solution  of the Yang-Baxter equation. From the viewpoint of universal algebra, regular $\star$-semibraces are (2,2,1)-type algebras.
In this paper we   continue to study  set-theoretic solutions  of the Yang-Baxter equation  and regular $\star$-semibraces.
We first consider several kinds of (2,2,1)-type algebras that induced by regular $\star$-semigroups and give some equivalent characterizations of the statement that  they form  regular $\star$-semibraces. Then we give sufficient and necessary conditions under which the associated maps to these (2,2,1)-type algebras are set-theoretic solutions of the Yang-Baxter equation. Finally, as analogues of weak braces defined by Catino, Mazzotta, Miccoli and Stefanelli, we introduce weak $\star$-braces in the class of regular $\star$-semibraces,  describe their algebraic structures and prove that the associated maps to weak $\star$-braces are  always set-theoretic solutions of the Yang-Baxter equation. The result of the present paper shows that the class of completely regular, orthodox and locally inverse regular $\star$-semigroups is a source of possibly new set-theoretic solutions of the Yang-Baxter equation. Our results establish the close connection between the Yang-Baxter equation and the classical structural theory of regular $\star$-semigroups.
\end{abstract}


\maketitle



\vspace{-0.8cm}
\section{Introduction}
The quantum Yang-Baxter equation first appeared in theoretical physics in Yang \cite{Yang} and  in statistical mechanics  in Baxter \cite{Baxter} independently. Now,  the Yang-Baxter equation has many applications in the areas of mathematics and mathematical physics such as knot theory, quantum computation, quantum group theory and so on. Let $V$ be a vector space and $R: V\otimes V \rightarrow V\otimes V$ be a linear transformation. Then $R$ is called  {\em a solution on a vector space} of the Yang-Baxter equation if $R_{12}R_{23}R_{12} =R_{23}R_{12}R_{23} \mbox{ in End}(V\otimes V\otimes V),$
where $R_{ij}$ acts as $R$ on the $i$-th and $j$-th components, and as the identity on the remaining component.

Finding all solutions of the Yang-Baxter equation is presently impossible, so in 1992, Drinfeld \cite{Drinfeld} posed the question of finding all the so called set-theoretic  solutions of the Yang-Baxter equation as they form solutions on vector spaces by linearly extending them. Let $S$ be a non-empty set and $r_S: S\times S\longrightarrow S\times S$ be a map. Then for any $a\in S$, $r$ induces the following maps:
$$\lambda_{a}: S\rightarrow S, \,b\mapsto \mbox{the first component of }r(a,b),$$$$
\rho_{b}: S\rightarrow S,\, a\mapsto \mbox{the second component of }r(a,b).$$
Thus, for all $a,b\in S$, we have $r(a,b)=(\lambda_{a}(b),\rho_{b}(a))$.
If $r$ satisfies the following equality $$(r\times {\rm id}_{S})({\rm id}_{S}\times r)(r\times {\rm id}_{S})=({\rm id}_{S}\times r)(r\times
{\rm id}_{S})({\rm id}_{S}\times r)$$ in the set of maps from $S\times S\times S$ to $S\times S\times S$, where  ${\rm id}_{S}$ is the identity map on $S$, then  $r_S$ is called a {\em set-theoretic  solution} of the Yang-Baxter equation, or briefly a {\em solution}.
It is a routine matter to prove  that $r_S$ is  a   solution  of the Yang-Baxter equation if and only if
\begin{equation}\label{jie1}
\lambda_{x}\lambda_{y}(z)=\lambda_{\lambda_{x}(y)}\lambda_{\rho_{y}(x)}(z),
\end{equation}
\begin{equation}\label{jie2}
\rho_{z}\rho_{y}(x)=\rho_{\rho_{z}(y)}\rho_{\lambda_{y}(z)}(x),
\end{equation}
\begin{equation}\label{jie3}
\lambda_{\rho_{\lambda_{y}(z)}(x)}\rho_{z}(y)=\rho_{\lambda_{\rho_{y}(x)}(z)}\lambda_{x}(y)
\end{equation}
for all $x,y,z\in S$. Moreover, a solution $r_S$ is called {\em left non-degenerate} (respectively, right non-degenerate) if $\lambda_a$ (respectively, $\rho_a$) is bijective for each $a\in S$. A {\em non-degenerate solution} is a solution which is both left and right non-degenerate.  A solution
that is neither left nor right non-degenerate is called {\em degenerate}. Furthermore,  a solution $r_S$ is called
{\em involutive} (respectively, {\em idempotent}) if $r^2 ={\rm id}_{S\times S}$ (respectively, $r^2=r$).

Set-theoretic solutions of the Yang-Baxter equation are investigated extensively in recent years, see for example
\cite{Etingof-Schedler-Soloviev,Gateva-Ivanova-Van den Bergh,Lu-Yan-Zhu,Soloviev}. Subsequently, non-degenerate involutive solutions have been studied by many authors, see the papers \cite{Catino-Mazzotta-Stefanelli,Cedo,Cedo-Jespers-Okninski,Rump1,Rump2} and the  references therein. In particular,   Rump \cite{Rump1,Rump2} introduced left cycle sets and left braces to investigate non-degenerate involutive solutions, respectively.  In 2007, Rump introduced left braces in \cite{Rump2} as a generalization of Jacobson radical rings, and a few years later, Ced$\acute{\rm o}$, Jespers, and Okni$\acute{\rm n}$ski \cite{Cedo-Jespers-Okninski} reformulated
Rump's definition of left braces.  Since then, left braces have become the most used tool in the investigations of  non-degenerate involutive solutions of the Yang-Baxter equation. In order to study non-degenerate solutions that are not necessarily involutive, Guarnieri and Vendramin \cite{Guarnieri-Vendramin} introduced a generalization of left braces, namely, skew left braces. To state the notions of skew left braces and left braces, we need recall some necessary terminologies and notations.

A non-empty set $S$ equipped with an associative binary operation $``\cdot"$ is usually  called a {\em semigroup} and is denoted by $(S, \cdot)$.
Let $(S, \cdot)$ be a semigroup.  As usual, for all $x, y\in S$, we denote $x\cdot y$ by $xy$.
Recall  from Howie \cite{Howie} that a semigroup $(S, \cdot)$ is called an {\em inverse semigroup} if, for every $a\in S$, there exists a unique element $b$ in  $S$ such that $aba=a$ and $bab=b$. We denote such an element $b$ by $a^{-1}$ and call it the {\em inverse } of $a$ in $(S, \cdot)$. In this case, we have
\begin{equation}\label{inverse semigroup}
(a^{-1})^{-1}=a^{-1} \mbox{ and } (ab)^{-1}=b^{-1}a^{-1}\mbox{ for all } a,b\in S.
\end{equation}
  Obviously, groups are inverse semigroups, and if $(S, \cdot)$ is a group and  $a\in S$, then $a^{-1}$ above is exactly the usual inverse of $a$ in the group $(S, \cdot)$.
Let $(S, +)$ and $(S, \cdot)$  be two inverse semigroups.  Then we usually denote the inverse  of $x$ in $(S, +)$ by $-x$ and denote $x+(-y)$ by $x-y$ for all $x, y\in S$. {\bf To avoid parentheses,  throughout this paper we will assume that the multiplication has higher precedence than the addition}. For  example, we use $xy^{-1}-z+y+zw$ to denote $(x\cdot (y^{-1}))+(-z)+y+(z\cdot w)$  for all $x, y, z, w\in S$.

Let $(S, +)$ and $(S, \cdot)$ be two groups. The triple $(S, +, \cdot)$ is called a {\em skew left brace}  if the following axiom holds:
\begin{equation}\label{skew brace}x(y+z)=xy-x+xz.
\end{equation}
If this is the case, the identities in $(S, +)$ and $(S, \cdot)$ coincide.  A skew left brace $(S, +, \cdot)$ is called a {\em left brace} if $(S, +)$  is an abelian group.
In \cite[Theorem 3.1]{Guarnieri-Vendramin}, Guarnieri and Vendramin have proved that every skew left brace $(S, +, \cdot)$ can give rise to some bijective non-degenerate solution $r_S$. One can prove that $r_S$ can be rewritten as
\begin{equation}\label{associated map}
\begin{array}{cc}
r_S: S\times S \rightarrow  S\times S,\,\,\,\,  (x, y)\mapsto (x(x^{-1}+y), (x^{-1}+y)^{-1}y),\\[2mm]
\mbox{i.e. } r_S(x, y)=(x(x^{-1}+y), (x^{-1}+y)^{-1}y) \mbox{ for all } x, y\in S,
\end{array}
\end{equation}
where $x^{-1}$ is the inverse of $x$ in the group $(S, \cdot)$ for all $x\in S$. Moreover, $r_S$ is involutive if and only if  $(S, +, \cdot)$ is a left brace. The above $r_S$ is called {\em the map associated to $(S, +, \cdot)$} in literature.
More details on left braces and skew left braces, the readers can  consult
the survey articles \cite{Cedo,Vendramin}.

To investigate left  non-degenerate solutions of the Yang-Baxter equation, Catino,  Mazzotta and Stefanelli \cite{Catino-Colazzo-Stefanelli} introduced left cancellative semibraces. Recall that a semigroup $(S, \cdot)$ is called {\em left cancellative} if for all $a,b,c\in S$, $ab=ac$ implies that $b=c$.
A triple $(S, +, \cdot)$ is called a {\em left cancellative semibrace} if $(S, +)$ is a  left cancellative semigroup, $(S, \cdot)$ is a group and the following axiom holds:
\begin{equation}\label{semibraces}
x(y+z)=xy+x(x^{-1}+z).
\end{equation}
One can show that (\ref{skew brace}) and (\ref{semibraces}) coincide in a skew brace. Since groups are  left cancellative semigroups, it follows that skew left braces are left cancellative semibraces.   It is obvious that $r_S$ in (\ref{associated map}) is well defined in a left cancellative semibrace $(S, +, \cdot)$, and     \cite[Theorem 9]{Catino-Colazzo-Stefanelli} says that $r_S$ is a left non-degenerate solution of the Yang-Baxter equation in the case. More results about left cancellative semibraces and left non-degenerate solutions of the Yang-Baxter equation have been provided in \cite{Castelli,Catino-Cedo-Stefanelli,Colazzo-Jespers-Van Antwerpen-Verwimp,Miccoli}. In 2019, Jespers and Van Antwerpen \cite{Jespers-Van Antwerpen} introduced and investigated general left semibraces.  A triple $(S, +, \cdot)$ is called a {\em left  semibrace} if $(S, +)$ is a semigroup, $(S, \cdot)$ is a group and the axiom (\ref{semibraces}) holds. Again, $r_S$ in (\ref{associated map}) is also well defined in a left semibrace. Among other things, Jespers and Van Antwerpen have shown that under a natural assumption,  the map $r_S$ associated to a left semibrace $(S, +, \cdot)$ is a solution (see \cite[Theorem 5.1]{Jespers-Van Antwerpen}), and there exists  a left semibrace $(S, +, \cdot)$ such that $r_S$ in (\ref{associated map}) is not a solution (see \cite[Example 2.11]{Jespers-Van Antwerpen}). Moreover,  Catino, Colazzo and Stefanelli \cite{Catino-Colazzo-Stefanelli2} have given a necessary and sufficient condition under which the map associated to a left semibrace  is a solution (see \cite[Theorem 3]{Catino-Colazzo-Stefanelli2}). More  results on general left semibraces can be found in \cite{Colazzo-Van Antwerpen,Stefanelli}.

To obtain new solutions, left inverse semibraces were introduced by  Catino, Mazzotta and Stefanelli in \cite{Catino-Mazzotta-Stefanelli}. A triple $(S, +, \cdot)$ is called a {\em left inverse semibrace} if $(S, +)$ is a  semigroup, $(S, \cdot)$ is an inverse semigroup and the axiom (\ref{semibraces}) holds. In this case, $r_S$ in
(\ref{associated map}) is well defined, but $x^{-1}$ is the inverse of $x$ in the inverse semigroup $(S, \cdot)$ for all $x\in S$.
Catino, Mazzotta and Stefanelli \cite{Catino-Mazzotta-Stefanelli} have provided a sufficient condition under which the map associated to a left inverse semibrace given in (\ref{associated map}) is a solution, and given some construction methods of  left inverse semibraces. Since  left semibraces are left inverse semibraces, the map associated to a left inverse semibrace may not be a solution. For this reason, Catino, Mazzotta, Miccoli and Stefanelli \cite{Catino-Mazzotta-Miccoli-Stefanelli} have considered a class of special left inverse semibraces, namely {\em weak left  braces}.
From \cite{Catino-Mazzotta-Miccoli-Stefanelli},  a triple $(S, +, \cdot)$ is called a {\em weak left brace} if $(S, +)$ and $(S, \cdot)$ are  inverse  semigroups and the following axioms hold:
\begin{equation}\label{weak semibrace}
x(y+z)=xy-x+xz,\,\,\,\,\,\, xx^{-1}=-x+x.
\end{equation}
According to \cite[Theorem 11]{Catino-Mazzotta-Miccoli-Stefanelli},  the map  given in (\ref{associated map})  associated to a weak left  brace is always  a solution. More recent results on left inverse semibraces and weak left braces can be found in \cite{Catino-Colazzo-Stefanelli3,Catino-Mazzotta-Stefanelli2,Catino-Mazzotta-Stefanelli3,Mazzotta-Rybolowicz-Stefanelli}.

The class of inverse semigroups is a kind of semigroups which has been studied most extensively in algebraic theory of semigroups, and so various generalizations of inverse semigroups have also received much attention from researchers. As a generalization of inverse semigroups, regular $\star$-semigroups (see Definition \ref{zhengze*} in our present paper) were introduced by   Nordahl  and Scheiblich in \cite{Nordahl-Scheiblich} in 1978.  Since then, regular $\star$-semigroups have been investigated by several authors and many meaningful results have been obtained, see for example \cite{Auinger,East-Azeef Muhammed,Imaoka-Inata-Yokoyama,Jones,Yamada2} and the references therein.  In view of the fact that regular $\star$-semigroups are generalizations of inverse semigroups, Miccoli \cite{Miccoli2} has introduced  left regular $\star$-semibraces under the name  of {\em left involution semibraces} (see Definition \ref{left regular*} and Remark \ref{left regular**} in our present paper) which are generalizations of left inverse semibraces,  and obtained several results parallel to left inverse semibraces. In particular, he has pointed out that left  regular $\star$-semibraces can provide some solutions of the Yang-Baxter equation under some necessary assumptions.

From the above statements, the following  problems seem  natural:
\begin{problem}\label{wenti1}
Find more natural examples of left  regular $\star$-semibraces which are induced by a given regular $\star$-semigroup.
\end{problem}
\begin{problem}\label{wenti2}Find and study an analogue of weak left braces in the class of left regular $\star$-semibraces.
\end{problem}
In this paper, around the above two questions, we continue to study  solutions of the Yang-Baxter equation and left regular $\star$-semibraces on the basis of the results in Miccoli \cite{Miccoli2}.  Section 2 contains some preliminaries on regular $\star$-semigroups. We include basic knowledge of regular $\star$-semigroups which is needed for the rest of the paper. In Section 3, inspired by examples given in \cite{Catino-Mazzotta-Stefanelli} and \cite{Miccoli2}, we consider Problem \ref{wenti1} and give more examples of left regular $\star$-semibraces which are induced by regular $\star$-semigroups. Section 4 gives   sufficient and necessary  conditions under which the maps associated  to (2,2,1)-type algebras provided in Section 3 are solutions of the Yang-Baxter equation. The final section is devoted to Problem \ref{wenti2}. We introduce weak left $\star$-braces in the class of left regular $\star$-semibraces as   analogues of weak left braces. After obtaining some structural properties, we show that the maps associated  to weak left $\star$-braces are always solutions of the Yang-Baxter equation.

\section{Preliminaries}
In this section, among other things we recall and give some notions and basic facts on regular $\star$-semigroups. Let $(S, \cdot)$ be a semigroup and $e\in S$. Then $e$ is called idempotent if $e^2=e$. As usual, we denote the set of idempotent elements in $(S, \cdot)$ by $E(S, \cdot)$.
We begin with the following definition of regular $\star$-semigroups.

\begin{defn}[\cite{Nordahl-Scheiblich}]\label{zhengze*} Let $(S, \cdot)$ be a semigroup  and $\ast: S\rightarrow S,\, x\mapsto x^\ast$ be a map.  Then the triple  $(S, \cdot,\, \ast)$ is called {\em a regular ${\star}$-semigroup} if the following axioms are satisfied:
\begin{equation}\label{shishi1.2} x x^\ast x=x,\,\, x^{\ast\ast}=x,\,\,
(xy)^\ast=y^\ast x^\ast,
\end{equation} where $x^{\ast\ast}=(x^\ast)^\ast$. We denote $$P(S, \cdot)=\{e\in E(S, \cdot)\mid e^\ast=e\}$$ and called it {\em the set of projections of $(S,\cdot, \ast)$}.
\begin{remark}\label{inverse-regular*}
Let $(S, \cdot,\, \ast)$  be a regular ${\star}$-semigroup. Then by the first two axioms in (\ref{shishi1.2}), we have $x^\ast=x^\ast x^{\ast\ast} x^\ast=x^\ast xx^\ast$  for all $x\in S$. Obviously, if $(S, \cdot)$ is an inverse semigroup, then by the definition of inverse semigroups and (\ref{inverse semigroup}), $(S, \cdot, \ast)$ forms a regular $\star$-semigroup by setting $x^\ast=x^{-1}$ for all $x\in S$. In fact, this is the only way by which an inverse semigroup $(S, \cdot)$  becomes a regular $\star$-semigroup.
\end{remark}
\end{defn}
\begin{lem}\label{danweiyuan}Let $(S,\cdot,\, \ast)$ be a regular $\star$-semigroup with  identity $1$. Then $1^\ast=1\in P(S, \cdot)$.
\end{lem}
\begin{proof}Let $x\in S$. Then $x^\ast=x^\ast\cdot 1=1\cdot x^\ast$. This implies that $$x=x^\ast=(x^\ast\cdot 1)^\ast=(1\cdot x^\ast)^\ast=1^\ast \cdot x=x\cdot 1^\ast,$$  and so $1^\ast$ is also the identity of $S$. Thus $1\cdot 1=1= 1^\ast$, and hence $1 \in P(S, \cdot)$.
\end{proof}
Recall from Howie \cite{Howie} that the {\em Green's relations} ${\mathcal L}$ and ${\mathcal R}$ on a semigroup $(S, \cdot)$ are defined as follows: For all $a,b\in S$, $$a{\mathcal L} b \mbox{ if and only if } a=b \mbox{ or  there exist } u, v\in S \mbox{ such that }a=ub,\, \, b=va,$$
$$a{\mathcal R} b \mbox{ if and only if } a=b \mbox{ or  there exist  } u, v\in S \mbox{ such that }a=bu,\, \, b=av.$$
Moreover, we denote ${\mathcal H}=\mathcal L\cap R$. Observe that $\mathcal L, R$ and $\mathcal H$ are all equivalences on $S$.
The following lemma  collects some  basic properties of regular $\star$-semigroups.
\begin{lem}[\cite{Nordahl-Scheiblich,Yamada2}] \label{1.2} Let $(S,\cdot,\, \ast)$ be a regular $\star$-semigroup, $e,f\in P(S,\cdot)$ and $a, b\in S$.
\begin{itemize}
\item[(1)] $P(S,\cdot)=\{ xx^\ast\mid x\in S\} =\{x^\ast x\mid x\in S\}$, $e^\ast=e\in E(S, \cdot)$ and  $aea^\ast, fef\in P(S,\cdot)$.
\item[(2)] $ef \in P(S,\cdot)$ if and only if $ef=fe$.
\item[(3)] $a\in E(S,\cdot)$ if and only if $a^\ast\in E(S,\cdot)$.
\item[(4)] $E(S,\cdot)=(P(S,\cdot))^2=\{ef\mid e,f \in P(S, \cdot)\}$.
\item[(5)] $a{\mathcal R} b$ (respectively, $a{\mathcal L} b$) if and only if  $ aa^\ast =bb^\ast $ (respectively, $a^\ast a =b^\ast b$).
\item[(6)] $a{\mathcal H} b$   if and only if  $aa^\ast=bb^\ast$ and $a^\ast a=b^\ast b$.
\end{itemize}
\end{lem}

Let $(S, \cdot)$ be a semigroup. Then $(S, \cdot)$ is called {\em regular} if, for every $x\in S$, there exists $y\in S$ such that $xyx=x$.
Obviously, inverse semigroups and regular $\star$-semigroups are regular semigroups. On the other hand, \cite[Theorem 5.1.1]{Howie} says that
\begin{equation}\label{inverse}
\mbox{ a regular semigroup } (S, \cdot) \mbox{ is inverse if and only if }   ef=fe \mbox{ for all } e,f\in E(S,\cdot).
\end{equation}
 Moreover, a regular semigroup  $(S, \cdot)$ is called
\begin{itemize}
\item  {\em orthodox} if $(E(S,\cdot), \cdot)$ is a  semigroup (i.e. $ef\in E(S,\cdot)$ for all $e,f\in E(S,\cdot)$),\,\, (\cite[Section 6.2]{Howie})
\item  {\em completely regular}  if $a\, {\mathcal H}\, a^2$ for all $a\in S$,\,\, (\cite[Theorem 2.2.5 and Proposition 4.1.1]{Howie})
\item {\em locally inverse} if $(eSe, \cdot)$ is an inverse  semigroup  for each $e\in E(S,\cdot)$,   (\cite[Section 6.1]{Howie})
\item{\em a Clifford semigroup} if $ex=xe$ for all $x\in S$ and $e\in E(S,\cdot)$.\,\, (\cite[Theorem 4.2.1]{Howie})
\end{itemize}
The above five classes of regular semigroups have been studied extensively (see Howie \cite{Howie}, Lawson \cite{Lawson}, Petrich \cite{Petrich}, Petrich and Reilly \cite{Petrich-Reilly} for details).
In the following statements, we shall give the characterizations of  regular $\star$-semigroups which are also inverse (respectively, orthodox, completely regular, locally inverse, Clifford) semigroups.
\begin{lem}\label{zhengzexingbanqun5}Let $(S,\cdot,\, \ast)$ be a regular $\star$-semigroup. Then
$(S, \cdot, \ast)$ is an inverse semigroup if and only if $ef=fe$ for all $e,f\in P(S,\cdot)$.
\end{lem}
\begin{proof}Since $P(S, \cdot)\subseteq E(S, \cdot)$, the necessity follows from (\ref{inverse}). To see the sufficiency, let $x,y\in E(S, \cdot)$. Then $x=ef$ and $y=gh$ for some $e,f,g,h\in P(S, \cdot)$ by Lemma \ref{1.2} (4). In view of the given condition, we have $xy=efgh=egfh=gefh=gehf=ghef=yx.$ This together with (\ref{inverse}) gives that $(S, \cdot, \ast)$ is an inverse semigroup.
\end{proof}

\begin{lem}[Theorem 3.2 in \cite{Nordahl-Scheiblich}]\label{zicu6}Let $(S,\cdot,\,  \ast)$ be a regular $\star$-semigroup.
Then $(S, \cdot, \ast)$ is orthodox if and only if $efg=(efg)^2$  $(\mbox{\rm i.e. } efg\in E(S, \cdot))$ for all $e,f,g\in P(S,\cdot)$.

\end{lem}
\begin{lem}\label{zicu7}For a regular $\star$-semigroup $(S,\cdot,\,  \ast)$,   the following statements are equivalent:
\begin{itemize}
\item[(1)] $(S, \cdot, \ast)$ is completely regular.
\item[(2)] $(S, \cdot, \ast)$ satisfies the axiom $xx^{\ast}=xx^{\ast}x^{\ast}xxx^{\ast}$.
\item[(3)] $(S, \cdot, \ast)$ satisfies the axiom $x^{\ast}x=x^{\ast}xxx^{\ast}x^{\ast}x$.
\end{itemize}
\end{lem}
\begin{proof} Since $``\ast"$ is an involution on $S$,  it follows that (2) is equivalent to (3). Let $x\in S$. If (1) holds, then $x{\mathcal H}x^2$ and so by Lemma \ref{1.2} (6), we have $x^\ast x=(x^2)^\ast x^2=x^\ast x^\ast xx$, and so $xx^\ast=x(x^\ast x)x^\ast =x(x^\ast x^\ast xx) x^\ast=xx^\ast x^\ast x xx^\ast.$ Thus (2) holds. Conversely, assume that (2) holds.  Then   have $xx^{\ast}=xx^{\ast}x^{\ast}xxx^{\ast}$  and   $$x^\ast x=x^\ast\cdot xx^{\ast}\cdot x=x^\ast\cdot xx^{\ast}x^{\ast}xxx^{\ast} \cdot x=x^\ast x^\ast xx=(x^2)^\ast x^2,$$  which gives that $x\,{\mathcal L}\, x^2$ by Lemma \ref{1.2} (5).  Moreover, in this case, (3) is also true. By the dual argument, we have $x {\mathcal R} x^2$. This implies that $x{\mathcal H} x^2$. Thus (1) holds.
\end{proof}
\begin{lem}\label{zicu4}Let $(S,\cdot,\, \ast)$ be a regular $\star$-semigroup. Then
$(S,\cdot,\, \ast)$ is locally inverse if and only if $efege=egefe$ for all $e,f,g\in P(S,\cdot)$.
\end{lem}
\begin{proof}
Assume that $(S,\cdot,\, \ast)$ is locally inverse and $e,f,g\in P(S,\cdot)\subseteq E(S,\cdot)$. Then  $(eSe, \cdot)$ is an inverse semigroup. Moreover, by Lemma \ref{1.2} (1) we have   $efe, ege\in P(S, \cdot)\cap eSe\subseteq E(eSe, \cdot)$.  By (\ref{inverse}),  $efege=efe\cdot ege=ege\cdot efe=egefe$.

Conversely, assume that the given condition holds. \cite[Lemma 1.2]{Imaoka-Inata-Yokoyama} says that $(S,\cdot,\, \ast)$ is locally inverse if and only if $(eSe,\cdot)$ is inverse for each $e\in P(S, \cdot)$. Now let $e \in P(S, \cdot)$. Then $e^\ast=e$,  and so $(exe)^\ast=e^\ast x^\ast e^\ast=ex^\ast e\in eSe$ for all $x\in S$. This implies that $(eSe, \cdot, \ast)$ is a regular $\star$-semigroup as $(eSe, \cdot)$ is a semigroup obviously.  Take arbitrary $f,g\in P(eSe, \cdot)$. Then $f, g\in P(S, \cdot)$ and $f=efe, g=ege$. By the given condition, we have $fg=efeege=efege=egefe=gf$. By Lemma \ref{zhengzexingbanqun5}, $(eSe, \cdot)$  is inverse. Thus $(S,\cdot,\, \ast)$ is locally inverse.
\end{proof}
\begin{lem}\label{zicu5}
Let $(S,\cdot,\,  \ast)$ be a regular $\star$-semigroup. Then $(S,\cdot,\,  \ast)$  is completely regular and locally inverse if and only if $xx^\ast y^\ast yxy=xy$ for all $x,y\in S$.
\end{lem}
\begin{proof}
Let $(S,\cdot,\ast)$ be completely regular and locally inverse, and  $a,b\in S$. Then
$$(aba)(bb^{\ast}a^{\ast}b^{\ast}a^{\ast}ab)=ab ab (ab)^\ast(ab)^\ast  ab=ab\cdot (ab)^\ast   ab ab (ab)^\ast(ab)^\ast  ab =ab (ab)^\ast ab=ab$$ by Lemma \ref{zicu7}.
As $(ab)a=aba$, we can obtain $ab{\mathcal R}aba$. By Lemma \ref{1.2} (5),
\begin{equation}\label{2.55}
abb^{\ast}a^{\ast}=ab(ab)^\ast=(aba)(aba)^\ast=abaa^{\ast}b^{\ast}a^{\ast}.
\end{equation}
Substituting  $a$ by $b^\ast$ and replacing  $b$   by $a^\ast$   in (\ref{2.55}), respectively, we have
\begin{equation}\label{2.56}
b^\ast a^\ast ab=b^\ast a^\ast a^{\ast\ast}b^{\ast\ast}=b^\ast a^\ast b^\ast b^{\ast\ast}a^{\ast\ast}b^{\ast\ast}=b^\ast a^\ast b^\ast bab.
\end{equation}
Denote $e=a^\ast a, f=a^\ast b^\ast ba=(ba)^\ast ba, g=a^\ast a bb^\ast a^\ast a$. Then $e,f,g\in P(S, \cdot)$ by Lemma \ref{1.2} (1). It is easy to see that $efe=f$ and $ege=g$.
Observe that $(S,\cdot,\ast)$  is locally inverse, it follows that $fg=efeege=efege=egefe=egeefe=gf$ by Lemma \ref{zicu4}. Thus
\begin{eqnarray*}\label{wl}
\nonumber &&aa^{\ast}b^{\ast}bab=aa^{\ast}b^{\ast}bab(ab)^\ast ab=(aa^{\ast}b^{\ast}b)abb^{\ast}a^{\ast}ab=(aa^{\ast}b^{\ast}b)aa^\ast abb^{\ast}a^{\ast}ab\\
\nonumber&=&a(a^{\ast}b^{\ast}ba)(a^{\ast}abb^{\ast}a^{\ast}a)b=a(a^{\ast}abb^{\ast}a^{\ast}a)(a^{\ast}b^{\ast}ba)b\,\,\,\,\,\,\,\,\,\,\,\,\,\,(\mbox{by }fg=gf)\\
\nonumber&=&abb^{\ast}a^{\ast}b^{\ast}bab=ab(b^{\ast}a^{\ast}b^{\ast}bab)=ab(b^{\ast}a^{\ast}ab)=ab(ab)^\ast ab=ab. \,\,\,\,\,\,\,\,\,\,\,\,\,\, (\mbox{by (\ref{2.56})})
\end{eqnarray*}
Conversely, assume that $(S,\cdot,\,  \ast)$ satisfies the axiom
\begin{equation}\label{2.57} xx^\ast y^\ast yxy=xy.
\end{equation}
Let $y=x^\ast x$ in (\ref{2.57}). Then $xx^\ast x^\ast x x=xx^\ast (x^\ast x)^\ast x^\ast x xx^\ast x=xx^\ast x=x$, and so $xx^\ast x^\ast xxx^\ast=xx^\ast$. By Lemma \ref{zicu7}, $(S,\cdot,\,  \ast)$  is completely regular.
Let $e,f,g\in P(S,\cdot)$ and take $x=e$ and $y=feg$ in (\ref{2.57}). Then $ee^\ast (feg)^\ast feg e feg=efeg$ and so $ege\cdot efe\cdot ege\cdot efe\cdot g=eegeffegefeg=efeg$.
By Lemma \ref{1.2} (1), (4), we have $ege, efe\in P(S, \cdot)$ and $ege efe\in E(S, \cdot)$. This implies that $efeg=egeefeg$ and hence $efeege=efege=egeefege=egefege\in P(S, \cdot)$ by Lemma \ref{1.2} (1).  In view of  Lemma \ref{1.2} (2) and the fact $ege, efe\in P(S, \cdot)$, we obtain $efege=efeege=egeefe=egefe$. By Lemma \ref{zicu4}, $(S,\cdot,\,  \ast)$  is locally inverse.
\end{proof}
\begin{lem}\label{zicu8}Let $(S,\cdot,\,  \ast)$ be a regular $\star$-semigroup. Then
$(S,\cdot,\,  \ast)$ is orthodox and locally inverse if and only if $afgb=agfb$ for all  $a,b\in S$ and $f,g\in P(S, \cdot)$.
\end{lem}
\begin{proof}Assume that $(S,\cdot,\,  \ast)$ is orthodox and locally inverse. By \cite[Lemma 1 and Theorem 1]{Yamada}, we have $axyb=ayxb$ for all $a,b\in S$ and $x,y\in E(S, \cdot)$, and so $afgb=agfb$ for all $a,b\in S$ and $f,g\in P(S, \cdot)$ as $P(S, \cdot)\subseteq E(S, \cdot)$. Conversely, assume that the given condition holds and $e, f, g\in P(S, \cdot)$. Then $(efg)^2=efgefg=eeffgg=efg$ and $efege=egefe$ by using the given condition several times. By Lemmas \ref{zicu6} and \ref{zicu4}, $(S,\cdot,\,  \ast)$ is orthodox and locally inverse.
\end{proof}

\begin{lem}\label{zicu9}For a regular $\star$-semigroup $(S,\cdot,\,  \ast)$,   the following statements are equivalent:
\begin{itemize}
\item[(1)]$(S,\cdot,\,  \ast)$ is completely regular, orthodox and locally inverse.
\item[(2)]$(S,\cdot,\,  \ast)$ satisfies the axiom $xy=xxx^{\ast}y$.
\item[(3)]$(S,\cdot,\,  \ast)$ satisfies the axiom $yx=yx^{\ast}xx$.
\item[(4)]$(S,\cdot,\,  \ast)$ satisfies the axiom $xyz=xyxx^{\ast}z$.
\item[(5)]$(S,\cdot,\,  \ast)$ satisfies the axiom $zyx=zx^{\ast}xyx$.
\item[(6)]$(S,\cdot,\,  \ast)$ satisfies the axiom $xyz=xyx^{\ast} xz$.
\item[(7)]$(S,\cdot,\,  \ast)$ satisfies the axiom $zyx= zxx^{\ast}yx$.
\end{itemize}
\end{lem}
\begin{proof}Since $``\ast"$ is an involution on $S$,  it follows that (2) is equivalent to (3), (4)  is equivalent to (5), and  (6) is equivalent to (7), respectively.
In the following statements, we shall show  that (1) implies   (2), (2) implies   (4), (4) implies (1),  and (4) is equivalent (6), respectively.

(1) $\Longrightarrow$ (2). Assume  that (1) holds and $a,b,c\in S$. By Lemmas \ref{zicu8} and \ref{zicu7},
$$a^{\ast}aaa^{\ast}b=a^{\ast}a\underline{a^{\ast}aaa^{\ast}} b=a^{\ast}a\underline{aa^{\ast}a^{\ast}a}b=\underbrace{a^\ast aa a^\ast a^\ast a}b=\underbrace{a^{\ast}a}b,$$
which implies that $aaa^{\ast}b=aa^{\ast}aaa^{\ast}b=aa^{\ast}ab=ab$.

(2) $\Longrightarrow$ (4). Assume that  the axiom $xy=xxx^{\ast}y$ is satisfied and $a,b,c\in S$. Then $abab(ab)^{\ast}c=abc$. In this equation,  replacing $a$ by $ab$, $b$ by $b^{\ast}b$ and $c$ by $aa^{\ast}c$, respectively, we obtain that
$abb^{\ast}babb^{\ast}b(abb^{\ast}b)^{\ast}aa^{\ast}c=abb^{\ast}baa^{\ast}c=abaa^{\ast}c.$
Since $abab(ab)^{\ast}c=abc$, we have
$$abb^{\ast}babb^{\ast}b(abb^{\ast}b)^{\ast}aa^{\ast}c=abab(ab)^{\ast}aa^{\ast}c
=ababb^{\ast}a^{\ast}aa^{\ast}c=ababb^{\ast}a^{\ast}c=abab(ab)^{\ast}c=abc.$$
Thus $abc=abaa^{\ast}c$, and so (4) holds.

(4) $\Longrightarrow$ (1). Assume that (4) holds. Then (5) also holds.
Let $e,f,g\in P(S, \cdot)$. By items (4) and (5) in the present proposition and Lemma \ref{1.2} (1), (4), we have
$$efg=efee^{\ast}g=(efe)g\in (P(S, \cdot))^{2}=E(S, \cdot),\,\,  efg=eg^{\ast}gfg=e(gfg)\in (P(S, \cdot))^{2}=E(S, \cdot).$$
By Lemma \ref{zicu6}, $(S,\cdot,\, \ast)$ is orthodox.  Moreover, we have $efeg=egfg$ for all $e,f,g\in P(S, \cdot)$. On the other hand,  we have $efe,ege\in P(S, \cdot)$ by Lemma \ref{1.2} (1). Replacing $g$ by $ege$ in the equation $efeg=egfg$,   we obtain   $efe
ege=eegefege=egefege\in P(S, \cdot)$  from Lemma \ref{1.2} (1). In view of Lemma \ref{1.2} (2), we have $efe ege=ege efe$ and so $efege=egefe$. Now Lemma \ref{zicu4} gives that $(S,\cdot,\, \ast)$ is locally inverse.
Finally, let $x\in S$. By item (4), we have $x(x^{\ast}x)(x^{\ast}x)=x(x^{\ast}x)xx^{\ast}(x^{\ast}x),$ and so $x=xxx^{\ast}x^{\ast}x$. This gives that $x^{\ast}x=x^\ast xxx^{\ast}x^{\ast}x$. By Lemma \ref{zicu7},   $S$ is completely regular. Thus (1) holds.

(4) $\Longleftrightarrow$ (6). Assume that $S$ satisfies the axiom $xyz=xyxx^{\ast}z$. Let $a ,b, c\in S$  and take $x=a^\ast, y=abaa^\ast, z=c$. Then
$a^\ast(abaa^\ast)c=a^\ast(abaa^\ast)a^\ast a^{\ast\ast}c=a^\ast(abaa^\ast)a^\ast ac.$ Multiply both sides of this equation by $a$ to the left, it follows that   $aa^\ast(abaa^\ast)c=aa^\ast(abaa^\ast)a^\ast ac$, and so $ab\cdot aa^\ast \cdot c=ab\cdot aa^\ast \cdot a^\ast ac$. By the given axiom, we have
$abc=ab\cdot aa^\ast \cdot c=ab\cdot aa^\ast \cdot a^\ast ac=aba^\ast ac,$ which gives (6). Dually, we can prove that (6) implies that (4).
\end{proof}
Lemma \ref{zicu9} gives the following corollary which will be used frequently in the sequel.
\begin{coro}\label{zhang2}Let $(S,\cdot,\,  \ast)$ be a regular $\star$-semigroup which is completely regular, orthodox and locally inverse. Then for all $x,y,z\in S$, we have
$$xyx^{\ast} xz=xyz=xyxx^{\ast}z,\,\, zxx^{\ast}yx=zyx=zx^{\ast}xyx,$$$$ x^\ast yx^{\ast} xz=x^\ast yz=x^\ast yxx^{\ast}z,\,\, zxx^{\ast}yx^\ast =zyx^\ast =zx^{\ast}xyx^\ast,$$  and so $xyy^\ast z=x(yy^\ast z)=x(yy^\ast y^\ast y^{\ast\ast} z)=(x yy^\ast y^\ast y) z=(xy^\ast y) z=xy^\ast yz$.
\end{coro}

\begin{lem}\label{zicu10}For a regular $\star$-semigroup $(S,\cdot,\,  \ast)$,   the following statements are equivalent:
\begin{itemize}
\item[(1)]$(S,\cdot,\,  \ast)$ is a Clifford semigroup.
\item[(2)]$(S,\cdot,\,  \ast)$ satisfies the axiom $xx^\ast=x^\ast x$.
\item[(3)]$(S,\cdot,\,  \ast)$ is completely regular and inverse.
\end{itemize}
\end{lem}
\begin{proof}
(1) $\Longrightarrow$ (2).  Assume that (1) holds. Then $ae=ea$ for all $a\in S$ and $e\in E(S,\cdot)$. Let $x\in S$. Since $x^\ast x\in E(S,\cdot)$, it follows that  $x=xx^\ast x=x^\ast xx$, and so $xx^\ast=x^\ast xxx^\ast$. Replacing $x$ by $x^\ast$, we have
$x^\ast x=xx^\ast x^\ast x$. Thus $xx^\ast=x^\ast xxx^\ast= xx^\ast x^\ast x=x^\ast x$.

(2) $\Longrightarrow$ (1).  Assume that (2) holds. Let $g,h\in P(S, \cdot)$.  Then by Lemma \ref{1.2} (1), $$ghg=ghhg=ghh^\ast g^\ast=(gh)(gh)^\ast=(gh)^\ast gh=h^\ast g^\ast gh=hggh=hgh.$$ This together with Lemma \ref{1.2} (4) implies that $gh=ghgh=g ghg=ghg$. Dually, we have $hg=hgh$. Thus $gh=ghg=hgh=hg$.
By Lemma \ref{zhengzexingbanqun5}, $(S,\cdot,\,  \ast)$  is inverse. On the other hand, let $x\in S$ and $ u\in E(S, \cdot)$. By Lemma \ref{1.2} (4), there exist $e,f\in P(S, \cdot)$ such that $u=ef$. As $(S,\cdot,\,  \ast)$  is inverse and $xx^\ast\in P(S, \cdot)$, by Lemma \ref{zhengzexingbanqun5} and item (2) this  implies that
$$ex=exx^\ast x=xx^\ast ex=x(ex)^\ast ex=xex(ex)^\ast=xe xx^\ast e=xxx^\ast ee=xx^\ast x e=xe.$$ Similarly, we can show that $xf=fx$. So we have $xu=xef=exf=efx=ux$. Thus $(S,\cdot,\,  \ast)$ is a Clifford semigroup.

(2) $\Longrightarrow$ (3).  Assume that (2) holds. Then $xx^\ast x^\ast x xx^\ast=xx^\ast xx^\ast xx^\ast=xx^\ast$ for all $x\in S$. By Lemma \ref{zicu7}, $(S,\cdot,\,  \ast)$ is completely regular. On the other hand, by the proof of ``(2) $\Longrightarrow$ (1)'',   $(S,\cdot,\,  \ast)$ is also inverse. Thus (3) hold.

(3) $\Longrightarrow$ (2).  Assume that (3) holds. Then by Lemmas \ref{zicu7} and \ref{zhengzexingbanqun5}, we have
 $$xx^\ast=xx^\ast (x^\ast x xx^\ast)=xx^\ast (xx^\ast x^\ast x)=xx^\ast x^\ast x=(xx^\ast x^\ast x)x^\ast x= (x^\ast x xx^\ast) x^\ast x=x^\ast x,$$
which implies that (2) is true.
\end{proof}
\begin{coro}\label{zhang1}
Let $(S,\cdot,\,  \ast)$ be a regular $\star$-semigroup which is commutative. Then $(S,\cdot,\,  \ast)$ is  completely regular, orthodox and locally inverse. In particular, if  $(S,\cdot,\,  \ast)$ satisfies the axiom $x=x^\ast$, then  for all  $a,b\in S$, we have
 $a^\ast=a=a^3$ and $ab=ba$, and so $(S,\cdot,\,  \ast)$ is  completely regular, orthodox and locally inverse in this case.
\end{coro}
\begin{proof} Let $x,y\in S$. Then $xxx^\ast y=x(xx^\ast)y=x(x^\ast x)y=xy$ by the commutativity. By Lemma \ref{zicu9}, $(S,\cdot,\,  \ast)$ is  completely regular, orthodox and locally inverse. If  $(S,\cdot,\,  \ast)$ satisfies the axiom $x=x^\ast$, then  we have
$a=a^{\ast}, ab=(ab)^{\ast}=b^{\ast}a^{\ast}=ba$ and $ a=aa^{\ast}a=aaa=a^{3}$ for all $a,b\in S$,
and so  $(S,\cdot,\,  \ast)$ is  commutative. Thus the remaining result also holds.
\end{proof}

\section{Regular $\star$-semibraces induced by regular $\star$-semigroups}\label{brace}
In this section, we study left (respectively, right, two-sided)  regular $\star$-semibraces which are  induced by regular $\star$-semigroups.
We first recall the notion of left (respectively, right, two-sided)  regular $\star$-semibraces.  Recall that a (2,2,1)-type algebra $(S,+,\cdot,\,\ast)$ is a non-empty set $S$ equipped with two binary operations $``+"$, $``\cdot"$  and  a map $\ast: S\rightarrow S,\,\, x\mapsto x^\ast$.
\begin{defn}\label{left regular*}Let $(S,+,\cdot,\, \ast)$ be a (2,2,1)-type algebra  such that  $(S, +)$ is a semigroup and $(S,\cdot,\, \ast)$ is a regular $\star$-semigroup. Then $(S,+,\cdot,\, \ast)$ is called a {\em left regular $\star$-semibrace} if the following axiom is satisfied:
\begin{equation}\label{lw1}x(y+z)=xy+x(x^{\ast}+z).
\end{equation}
Dually, $(S,+,\cdot,\,\ast)$ is called a {\em right regular $\star$-semibrace} if the following axiom is satisfied:
\begin{equation}\label{lw2}(z+y)x=(z+x^{\ast})x+yx.
\end{equation}
Moreover, $(S,+,\cdot,\, \ast)$ is called a {\em two-sided regular $\star$-semibrace} if both (\ref{lw1}) and  (\ref{lw2}) are satisfied.
\end{defn}
\begin{remark}\label{left regular**}
Left (respectively, right, two-sided)  regular $\star$-semibraces were introduced actually by Miccoli in \cite{Miccoli2} under the name of ``left (respectively, right, two-sided) involution semibraces". Since involution semigroups in literature are semigroups equipped with a general involution  and regular $\star$-semigroups are just special  involution semigroups, we tend to use the term  left (respectively, right, two-sided)  regular $\star$-semibraces.
\end{remark}
By Remark \ref{inverse-regular*}, left (respectively, right, two-sided) inverse semibraces are left (respectively,  right, two-sided)
regular $\star$-semibraces, and so left (respectively, right, two-sided) braces, skew left braces and  left semibraces are all left (respectively,
right, two-sided) regular $\star$-semibraces. Motivated by \cite[Examples 1--4]{Catino-Mazzotta-Stefanelli},  we shall investigate some left (respectively, right, two-sided)  regular $\star$-semibraces induced regular $\star$-semigroups.
\begin{prop}\label{lw3}Let $(S,\cdot,\, \ast)$ be a regular $\star$-semigroup. Define the binary operations ``$+$" and $``\oplus"$ on $S$ as follows: For all $a,b\in S$,  $a+b=ab,\, a\oplus b=ba.$ Then $(S,+, \cdot,\, \ast)$ (respectively, $(S,\oplus, \cdot,\, \ast)$) is a left (or equivalently, right, two-sided) regular ${\star}$-semibrace if and only if $(S,\cdot,\,  \ast)$ is completely regular, orthodox and locally inverse.
\end{prop}
\begin{proof} In this case, $(S, +)$ (respectively, $(S,\oplus)$) is a semigroup obviously.   Let $a,b,c\in S$. Then
$$a(b+c)=abc,\,ab+a(a^{\ast}+c)=abaa^{\ast}c,\, (c+b)a=cba,\,(c+a^{\ast})a+ba=ca^{\ast}aba,$$
$$a(b\oplus c)=acb,\,ab\oplus a(a^{\ast}\oplus c)= aca^{\ast}ab,\, (c\oplus b)a=bca,\,(c\oplus a^{\ast})a\oplus ba=ba a^{\ast} ca.$$
By Lemma \ref{zicu9}, the result follows.
\end{proof}

\begin{prop}\label{2.9}
Let $(S,\cdot,\, \ast)$ be a regular $\star$-semigroup. Define the binary operations ``$+$" and  ``$\oplus$" on $S$ as follows: For all $a,b\in S$,
$a+b=a^{\ast}b^{\ast},\,\,\,  a\oplus b=b^{\ast}a^{\ast}.$  Then the following statements are equivalent:
\begin{itemize}
\item[(1)] $(S, +)$ (respectively, $(S,\oplus)$) is a  semigroup.
\item[(2)]$(S,+, \cdot,\, \ast)$  (respectively, $(S,\oplus, \cdot,\, \ast)$) satisfies the axiom (\ref{lw1}): $x(y+z)=xy+x(x^{\ast}+z)$ (respectively, $x(y\oplus z)=xy\oplus x(x^{\ast}\oplus z)$).
\item[(3)]$(S,+, \cdot,\, \ast)$ (respectively, $(S,\oplus, \cdot,\, \ast)$) satisfies the axiom (\ref{lw2}): $(z+y)x=(z+x^{\ast})x+yx$ (respectively, $(z\oplus y)x=(z\oplus x^{\ast})x\oplus yx$).
\item[(4)]$(S,\cdot,\, \ast)$ satisfies the axiom $x=x^{\ast}$.
\end{itemize}
Thus, $(S,+, \cdot,\, \ast)$ (respectively, $(S,\oplus, \cdot,\,\ast)$) is a left (or equivalently, right, two-sided) regular ${\star}$-semibrace if and only if $(S,\cdot,\ast)$ satisfies the axiom $x=x^{\ast}$.
\end{prop}
\begin{proof}We only prove the case for  $(S, +, \cdot,\,  \ast)$, the case for  $(S,\oplus, \cdot,\,  \ast)$ can be proved similarly.

(1) $\Longrightarrow$ (4). Assume that (1) holds. Then $(S,\cdot,\, \ast)$ satisfies the axiom
\begin{equation}\label{2.9.1}
yxz^{\ast}=(x+y)+z=x+(y+z)=x^{\ast}zy.
\end{equation}
Let $a\in S$ and take $x=z=aa^{\ast},y=a^{\ast}$ in (\ref{2.9.1}). Then
$a^{\ast}=a^{\ast}(aa^{\ast})(aa^{\ast})^{\ast}=(aa^{\ast})^{\ast}aa^{\ast}a^{\ast} =aa^{\ast}a^{\ast}.$
Let $y=z=a,x=a^{\ast}$ in (\ref{2.9.1}). Then $aa^{\ast}a^{\ast}=(a^{\ast})^{\ast}aa=a^{3}$. This implies that $a^{\ast}=a^{3}$, and so $a=aa^{\ast}a=aa^{3}a=a^{5}.$ Let $y=z=a,x=a^{\ast}a$  in (\ref{2.9.1}). Then $a^{\ast}aaa=(aa^{\ast})^{\ast}aa=aa^{\ast}aa^{\ast}=aa^{\ast}$. So
$a^{2}=aa=aa^{5}=a^{3}a^{3}=a^{\ast}a^{3}=aa^{\ast}=aa^{3}=a^{4}.$ Thus $a=a^{5}=a^{4}a=a^{2}a=a^{3}=a^{\ast}.$

(2) $\Longrightarrow$ (4). Assume that (2) holds.  Then $(S,\cdot,\, \ast)$ satisfies the axiom
\begin{equation}\label{2.9.2}
xy^{\ast}z^{\ast}=x(y+z)=xy+x(x^{\ast}+z)=xy+x(x^{\ast})^{\ast}z^{\ast}=(xy)^{\ast}(x(x^{\ast})^{\ast}z^{\ast})^{\ast}=y^{\ast}x^{\ast}zx^{\ast}x^{\ast}.
\end{equation}
Let $a\in S$ and take $x=a^{\ast},y=a^{\ast},z=a$ in (\ref{2.9.2}). Then
\begin{eqnarray*}
a^{\ast}=a^{\ast}aa^{\ast}=a^{\ast}(a^{\ast})^{\ast}a^{\ast}=(a^{\ast})^{\ast}(a^{\ast})^{\ast}a(a^{\ast})^{\ast}(a^{\ast})^{\ast}=aaaaa=a^{5},
\end{eqnarray*}
and so $a=aa^{\ast}a=aa^{5}a=a^{7}$. Take $x=a^{\ast}a,y=z=a^{\ast}$ in (\ref{2.9.2}). Then $a^{\ast}a(a^{\ast})^{\ast}(a^{\ast})^{\ast}=(a^{\ast})^{\ast}(a^{\ast}a)^{\ast}a^{\ast}(a^{\ast}a)^\ast (a^{\ast}a)^\ast.$
That is, $a^{\ast}aaa=aa^{\ast}aa^{\ast}a^{\ast}a=aa^{\ast}a^{\ast}a$. Since $a^{\ast}=a^{5}$ and  $a=a^{7}$,
$$a^{2}=aa^{7}=a^{8}=a^{\ast}aaa=aa^{\ast}a^{\ast}a=aa^{5}a^{5}a=a^{7}a^{5}=aa^{5}=a^{6},$$
and so $a=a^{7}=a^{5}a^{2}=a^{5}a^{6}=a^{5}aa^{5}=a^{\ast}aa^{\ast}=a^{\ast}$.

(3) $\Longrightarrow$ (4). Assume that (3) is true. Then $(S,\cdot, \ast)$ satisfies the axiom
\begin{equation}\label{2.9.3}
z^{\ast}y^{\ast}x=(z+y)x=(z+x^{\ast})x+yx=z^{\ast}(x^{\ast})^{\ast}x+yx=(z^{\ast}xx)^{\ast}(yx)^{\ast}=x^{\ast}x^{\ast}zx^{\ast}y^{\ast}.
\end{equation}
Since $\ast$ is an involution, (\ref{2.9.3}) is equivalent to its dual axiom: $xy^{\ast}z^{\ast}=y^{\ast}x^{\ast}zx^{\ast}x^{\ast}.$ This axiom is exactly (\ref{2.9.2}). By the proof of ``(2) $\Longrightarrow$ (4)", $(S,\cdot,\, \ast)$ satisfies the axiom $x=x^{\ast}$.

(4) $\Longrightarrow$ (1), (2), (3). If (4) holds, then by Corollary \ref{zhang1}, it is easy to see that (\ref{2.9.1}), (\ref{2.9.2}) and  (\ref{2.9.3}) hold. That is,  (1), (2) and (3) are true.
\end{proof}

\begin{prop}\label{2.11}
Let $(S,\cdot,\,  \ast)$ be a regular $\star$-semigroup. Define the binary operations ``$+$"  and $``\oplus"$on $S$ as follows: For all $a,b\in S$,
$a+b=a^{\ast}b,\,\,\, a\oplus b=ab^\ast.$  Then the following statements hold:
\begin{itemize}
\item[(A)] $(S, +, \cdot,\, \ast)$ (respectively, $(S, \oplus, \cdot,\, \ast)$) satisfies the axiom (\ref{lw1}): $x(y+z)=xy+x(x^{\ast}+z)$ (respectively, the axiom (\ref{lw2}): $(z\oplus y)x=(z\oplus x^{\ast})x\oplus yx$) if and only if $(S,\cdot, \ast)$ is commutative.
\item[(B)]  The following statements are equivalent:
\begin{itemize}
\item[(1)]$(S, +)$ (respectively, $(S, \oplus)$) is a  semigroup.
\item[(2)]$(S,+, \cdot,\, \ast)$ (respectively, $(S, \oplus, \cdot, \ast)$) satisfies the axiom (\ref{lw2}): $(z+y)x=(z+x^{\ast})x+yx$ (respectively, (\ref{lw1}): $x(y\oplus z)=xy\oplus x(x^{\ast}\oplus z)$).
\item[(3)]$(S,\cdot,\, \ast)$ satisfies the axiom $x=x^{\ast}$.
\end{itemize}
\item[(C)] If item (3) in Part (B) holds, then  $(S,\cdot, \ast)$ is commutative. Thus, $(S,+, \cdot,\, \ast)$ (respectively, $(S, \oplus, \cdot,\, \ast)$)  is a left (or equivalently, right, two-sided) regular ${\star}$-semibrace if and only if $(S,\cdot,\ast)$ satisfies the axiom $x=x^{\ast}$.
\end{itemize}
\end{prop}
\begin{proof}We only prove the case for  $(S, +, \cdot,\,  \ast)$, the case for  $(S,\oplus, \cdot,\,  \ast)$ can be proved similarly.
We first prove part (A).  Assume that $(S, +, \cdot,\, \ast)$ satisfies the axiom (\ref{lw1}). Then $(S,\cdot,\, \ast)$ satisfies the axiom
\begin{equation}\label{2.12.1}
xy^{\ast}z=x(y+z)=xy+x(x^{\ast}+z)=xy+x(x^{\ast})^{\ast}z=(xy)^{\ast}x(x^{\ast})^{\ast}z=y^{\ast}x^{\ast}xxz.
\end{equation}
Let $e,f\in P(S, \cdot)$ and take $x=e,y=z=f$ in (\ref{2.12.1}). Then $ef=ef^\ast f=f^{\ast\ast}eeef=fef\in P(S, \cdot)$,  and  so  $ef=fe$ by Lemma \ref{1.2} (1), (2). This implies that $(S,\cdot,\, \ast)$ is inverse by Lemma \ref{zhengzexingbanqun5}. Let $a\in S$ and take $x=y=z=a$ in (\ref{2.12.1}). Then
$a=aa^{\ast}a=a^{\ast}a^{\ast}aaa$, and so $a^{\ast}=(a^{\ast}a^{\ast}aaa)^{\ast}=a^{\ast}a^{\ast}a^{\ast}aa$. Furthermore, we have $a^{\ast}(a^{\ast}a)=a^{\ast}a^{\ast}a^{\ast}aa(a^{\ast}a)=a^{\ast}a^{\ast}a^{\ast}aa=a^{\ast}$ and so $a(a^\ast a^\ast a)aa^\ast=aa^\ast aa^\ast=aa^\ast$.
This implies that $(S,\cdot,\, \ast)$ is completely regular by Lemma \ref{zicu7}. So Lemma \ref{zicu10} gives that $(S,\cdot,\, \ast)$ is a Clifford semigroup, and so
\begin{equation}\label{zhongxin}
\mbox{\bf Clifford condition}\hspace{1cm} ex=xe  \mbox{ for all } x\in S  \mbox{ and } e\in E(S, \cdot).
\end{equation}
Let $a,b\in S$ and take $x=b, y=a^{\ast}, z=b^{\ast}b$ in (\ref{2.12.1}). Then $$ba=bb^{\ast}ba=bab^{\ast}b=b(a^{\ast})^{\ast}b^{\ast}b=(a^{\ast})^{\ast}b^{\ast}bbb^{\ast}b=ab^{\ast}bb=abb^{\ast}b=ab.$$
This gives that $(S,\cdot, \ast)$ is commutative.

Conversely, assume that $(S,\cdot, \ast)$ is commutative and $x,y,z\in S$. Then $y^{\ast}x^{\ast}xxz=xx^{\ast}xy^{\ast}z=xy^{\ast}z$ and so (\ref{2.12.1}) holds. That is, $(S, +, \cdot,\, \ast)$ satisfies (\ref{lw1}).

Now we prove Part (B).

(1) $\Longrightarrow$ (3). Assume that (1) holds. Then $(S,\cdot,\, \ast)$ satisfies the axiom
\begin{equation}\label{2.11.1}
y^{\ast}xz=(x^{\ast}y)^{\ast}z=x^{\ast}y+z=(x+y)+z=x+(y+z)=x+y^{\ast}z=x^{\ast}y^{\ast}z.
\end{equation}
Let $a\in S$ and take $x=a,y=z=a^{\ast}$ in (\ref{2.11.1}). Then
$aaa^{\ast}=(a^{\ast})^{\ast}aa^{\ast}=a^{\ast}(a^{\ast})^{\ast}a^{\ast}=a^{\ast}aa^{\ast}=a^{\ast}.$
This shows  $a=(aaa^{\ast})^{\ast}=(a^{\ast})^{\ast}a^{\ast}a^{\ast}=aa^{\ast}a^{\ast}=aa^{\ast}(aaa^{\ast})=aaa^{\ast}=a^{\ast},$ and so (3) holds.

(2) $\Longrightarrow$ (3). Assume that (2) holds. Then $(S,\cdot,\, \ast)$ satisfies the axiom
\begin{equation}\label{2.11.2}
z^{\ast}yx=(z+y)x=(z+x^{\ast})x+yx=z^{\ast}x^{\ast}x+yx=(z^{\ast}x^{\ast}x)^{\ast}yx=x^{\ast}xzyx.
\end{equation}
Let $a\in S$ and take $x=a, z=y=a^{\ast}$ in (\ref{2.11.2}). Then
$a=aa^{\ast}a=(a^{\ast})^{\ast}a^{\ast}a=a^{\ast}aa^{\ast}a^{\ast}a=a^{\ast}a^{\ast}a.$
This yields that $a^{\ast}=(a^{\ast}a^{\ast}a)^{\ast}=(a^{\ast}a)a=a^{\ast}a(a^{\ast}a^{\ast}a)=a^{\ast}a^{\ast}a=a,$ and so (3) holds.

(3)$\Longrightarrow$ (1), (2).  If (3) holds, then by Corollary \ref{zhang1}, it is easy to see that (\ref{2.11.1}) and (\ref{2.11.2}) are true, and so items (1) and (2) hold.

Finally, we consider Part (C). If item (3) in Part (B) holds,  then by Corollary \ref{zhang1}  $(S, \cdot, \ast)$ is commutative.  The final result follows from Parts (A) and (B).
\end{proof}

\begin{prop}\label{wliu2}
Let $(S,\cdot,\, \ast)$ be a regular $\star$-semigroup. Define the binary operations ``$+$"  and $``\oplus"$ on $S$ as follows: For all $a,b\in S$,
$a+b=b^{\ast}a,\,\,\, a\oplus b=ba^{\ast}.$ Then the following statements hold:
\begin{itemize}
\item[(1)] $(S, +)$ (respectively, $(S, \oplus)$) is a  semigroup if and only if $(S,\cdot,\, \ast)$ satisfies the axiom $x=x^{\ast}$.
\item[(2)] $(S, +, \cdot,\, \ast)$ (respectively, $(S, \oplus, \cdot,\, \ast)$) satisfies the axiom (\ref{lw2}): $(z+y)x=(z+x^{\ast})x+yx$ (respectively, the axiom (\ref{lw1}): $x(y\oplus z)=xy\oplus x(x^{\ast}\oplus z)$) if and only if $(S,\cdot, \ast)$ is commutative.
\item[(3)]If $(S, +, \cdot,\, \ast)$ satisfies  the axiom (\ref{lw1}): $x(y+z)=xy+x(x^{\ast}+z)$ (respectively, the axiom (\ref{lw2}): $(z\oplus y)x=(z\oplus x^{\ast})x \oplus yx$), then
$(S,\cdot,\, \ast)$ is completely regular, orthodox and locally inverse.
\item[(4)] If $(S,\cdot,\, \ast)$ satisfies the axiom $x=x^{\ast}$, then  $(S,\cdot, \ast)$ is commutative and (\ref{lw1}) (respectively, (\ref{lw2})) is satisfied. Thus, $(S,+, \cdot,\, \ast)$ (respectively, $(S, \oplus, \cdot,\, \ast)$) is a left (or equivalently, right, two-sided) regular ${\star}$-semibrace if and only if $(S,\cdot,\, \ast)$ satisfies the axiom $x=x^{\ast}$.
\end{itemize}
\end{prop}
\begin{proof}We only prove the case for  $(S, +, \cdot,\,  \ast)$, the case for  $(S,\oplus, \cdot,\,  \ast)$ can be proved similarly.

(1) The associativity of $``+"$ is equivalent to the axiom
\begin{equation}\label{2.13.1}
y^{\ast}zx=(z^{\ast}y)^{\ast}x=x+(y+z)=(x+y)+z=y^{\ast}x+z=z^{\ast}y^{\ast}x.
\end{equation}
Clearly, (\ref{2.13.1}) is exactly (\ref{2.11.1}). By the proof of Proposition \ref{2.11}, the result follows.

(2) Assume that $(S,+, \cdot,\, \ast)$ satisfies (\ref{lw2}). Then $(S,\cdot,\ast)$ satisfies the axiom
\begin{equation}\label{2.14.1}
y^{\ast}zx=(z+y)x=(z+x^{\ast})x+yx=(x^{\ast})^{\ast}zx+yx=(yx)^{\ast}xzx=x^{\ast}y^{\ast}xzx.
\end{equation}
Let $e,f\in P(S, \cdot)$ and take $x=f, y=z=e$ in (\ref{2.14.1}). Then $ef=e^\ast ef=f^\ast e^\ast fef=fefef\in P(S, \cdot)$, and so $ef=fe$ by Lemma \ref{1.2} (1) and (2). This implies that $(S,\cdot,\ast)$ is inverse by Lemma \ref{zhengzexingbanqun5}. Let $a\in S$ and take $x=a, y=z=a^{\ast}$ in (\ref{2.14.1}).  Then $a=aa^{\ast}a=(a^{\ast})^{\ast}a^{\ast}a=a^{\ast}(a^{\ast})^{\ast}aa^{\ast}a=a^{\ast}aa,$ and so $aa^\ast=(aa^\ast)\cdot a \cdot a^\ast=aa^\ast(a^\ast aa)a^\ast$. This shows that   $(S,\cdot,\ast)$  is completely regular by Lemma \ref{zicu7}. Thus $(S,\cdot,\, \ast)$ is a Clifford semigroup by Lemma \ref{zicu10}.  Let $a, b\in S$ and take $y=a^{\ast}, z=b, x=b^{\ast}$ in (\ref{2.14.1}). Then $abb^{\ast}=(a^{\ast})^{\ast}bb^{\ast}=(b^{\ast})^{\ast}(a^{\ast})^{\ast}b^{\ast}bb^{\ast}=bab^{\ast}.$ So $ab=(abb^{\ast})b=(bab^{\ast})b=bb^{\ast}ba\overset{(\ref{zhongxin})}{=}ba.$ This gives that $(S,\cdot, \ast)$ is commutative.

Conversely, if $(S,\cdot, \ast)$ is commutative, then for all $x, y, z\in S$, we have $x^{\ast}y^{\ast}xzx=y^{\ast}zxx^{\ast}x=y^{\ast}zx$. This implies that (\ref{2.14.1}) holds. That is, $(S,+,\cdot,\, \ast)$ satisfies (\ref{lw2}).

(3) Assume that $(S,+, \cdot,\, \ast)$ satisfies (\ref{lw1}).  Then $(S,\cdot,\,\ast)$ satisfies the axiom:
\begin{equation}\label{2.15.1}
xz^{\ast}y=x(y+z)=xy+x(x^{\ast}+z)=xy+xz^{\ast}x^{\ast}=(xz^{\ast}x^{\ast})^{\ast}xy=xzx^{\ast}xy.
\end{equation}
Let $a,b\in S$ and take $x=y=z=a$ in (\ref{2.15.1}). Then $a=aa^{\ast}a=aaa^{\ast}aa=a^{3}$. On the other hand, let $x=a, y=b, z=aa^{\ast}$ in (\ref{2.15.1}). Then we have $a(aa^{\ast})^{\ast}b=aaa^{\ast}a^{\ast}ab$. That is, $aaa^{\ast}b=aaa^{\ast}a^{\ast}ab. $  Since $a=a^{3}$, we have $aaa^{\ast}a^{\ast}ab=aaa^{\ast}a^{\ast}a^{3}b=(aaa^{\ast}a^{\ast}aa)ab=aaab=a^{3}b=ab,$  which gives that $aaa^{\ast}b=ab$. By Lemma \ref{zicu9}, $(S,\cdot,\ast)$ is completely regular, orthodox and locally inverse.

(4) If $(S,\cdot,\, \ast)$ satisfies the axiom $x=x^{\ast}$, then    $(S, \cdot, \ast)$ is commutative by Corollary \ref{zhang1}. This implies that
$xzx^\ast xy=xx^\ast xzy=xzy=xz^\ast y$ for all $x,y,z\in S$. By (\ref{2.15.1}),  $(S,\cdot,\, \ast)$ satisfies the axiom (\ref{lw1}).
Now the final statement follows from (1), (2) and (3).
\end{proof}

\begin{remark}\label{zhuji1} In the environment of Proposition  \ref{wliu2} (3), we can not infer that $(S,\cdot,\,\ast)$ is commutative. For example,
let $X=\{1,2\}$. Define a binary operation $``\cdot"$ and a unary operation $``\ast"$ on $S=X\times X$ as follows: $$(\forall (a, b),(c, d)\in S)\,\,\, (a, b)(c,d)=(a,d),\,\, (a,b)^\ast=(b,a).$$ Then  $(S, \cdot,\, \ast)$ forms a regular $\star$-semigroup obviously.
Define the binary operations ``$+$" and  ``$\oplus$" on $S$ as follows: For all $x, y\in S$,  $x+y=y^{\ast}x,\,\, x\oplus y=yx^\ast.$ Then by (\ref{2.15.1}), we have
$$x(y+z)=xz^{\ast}y=xy=xzx^{\ast}xy=xy+x(x^{\ast}+z)$$
$$(\mbox{respectively},\, (z\oplus y)x=yz^\ast x=yx=yxx^\ast zx=(z\oplus x^{\ast})x \oplus yx)$$
for all $x, y, z\in S$. This implies that (\ref{lw1}) (respectively,  (\ref{lw2})) holds. Obviously, $(S,\cdot, \ast)$ is not commutative.

On the other hand, the following example shows that the converse of Proposition   \ref{wliu2} (3) is not true. Let $G=\{e,x,y\}$ be a cyclic group, where $e$ is the identity. Define $e^{\ast}=e, x^{\ast}=x^{-1}=y, y^{\ast}=y^{-1}=x.$ Obviously, $(G,\cdot,\, \ast)$ is a completely regular, orthodox, locally inverse and commutative regular ${\star}$-semigroup. Define ``$+$" and  ``$\oplus$" on $G$ as follows: For all $a,b\in G$, $a+b=b^{\ast}a,\,\, a\oplus b=ba^\ast.$
As $x(x+x)=xx^\ast x=xe=x$ and  $$xx+x(x^{\ast}+x)=y+xx^\ast y=y+y=y^\ast y=e,$$ it follows that $(G, +, \cdot,\, \ast)$ does not satisfy the axiom (\ref{lw1}).
Similarly, we have $(x\oplus x)x=x\not=e= (x\oplus x^\ast)x\oplus xx$, and so $(G, \oplus, \cdot,\, \ast)$ does not satisfy the axiom (\ref{lw2}).
\end{remark}
In the following statements, we consider another types of left (respectively, right, two sided) regular $\star$-semibraces.
\begin{prop}\label{lww3}Let $(S,\cdot,\, \ast)$ be a regular $\star$-semigroup. Define the binary operations ``$+$" and ``$\oplus$" on $S$ as follows: For all $a,b\in S$,
$a+b=aa^{\ast}b,\,\,  a\oplus b=ab^\ast b.$ Then the following statements hold:
\begin{itemize}
\item[(A)]The following statements are equivalent:
\begin{itemize}
\item[(1)] $(S,+)$ (respectively,  $(S,\oplus)$) is a semigroup.
\item[(2)] $(S, \cdot,\, \ast)$  is orthodox and locally inverse.
\item[(3)]$(S,+, \cdot,\, \ast)$ (respectively, $(S,\oplus, \cdot,\, \ast)$) satisfies the axiom (\ref{lw1}): $x(y+z)=xy+x(x^{\ast}+z)$ (respectively, the axiom (\ref{lw2}): $(z\oplus y)x=(z\oplus x^{\ast})x\oplus yx$).
\end{itemize}
\item[(B)]$(S,+, \cdot,\, \ast)$ (respectively, $(S,\oplus, \cdot,\,  \ast)$) satisfies the axiom (\ref{lw2}): $(z+y)x=(z+x^{\ast})x+yx$ (respectively, the axiom (\ref{lw1}): $x(y\oplus z)=xy\oplus x(x^{\ast}\oplus z)$) if and only if $(S,\cdot,\ \ast)$ is completely regular and locally inverse.
\item[(C)] $(S,+, \cdot,\, \ast)$ (respectively, $(S,\oplus, \cdot,\, \ast)$) is a left (respectively, right) regular $\star$-semibrace if and only if $(S,\cdot, \, \ast)$ is orthodox and locally inverse, and $(S,+, \cdot,\, \ast)$ (respectively, $(S,\oplus, \cdot,\, \ast)$) is   right (respectively, left) (or equivalently, two-sided) regular $\star$-semibrace if and only if $(S,\cdot,\, \ast)$ is completely regular, orthodox  and locally inverse.
\end{itemize}
\end{prop}
\begin{proof}We only prove the case for  $(S, +, \cdot,\,  \ast)$, the case for  $(S,\oplus, \cdot,\,  \ast)$ can be proved similarly.

{\em Part (A)}:
(1) $\Longrightarrow$ (2). Let $(S,+)$ be a semigroup. Then $(S,\cdot,\ast)$ satisfies the axiom:
\begin{equation}\label{2.3}
  xx^{\ast}yy^{\ast}xx^{\ast}z=xx^{\ast}y(xx^{\ast}y)^{\ast}z=(x+y)+z=x+(y+z)=xx^{\ast}yy^{\ast}z.
\end{equation}
Let $e,f,g\in P(S, \cdot)$. Let $x=f, y=g, z=e$ in (\ref{2.3}).  Then we have $fgfe=ff^\ast gg^\ast ff^\ast e=ff^\ast gg^\ast e=fge$ by Lemma \ref{1.2} (1). Since $fgf\in P(S, \cdot)$ by  Lemma \ref{1.2} (1), we have
$fge=fgf e\in E(S, \cdot)$ by Lemma \ref{1.2} (4), and so $(S,\cdot,\, \ast)$ is orthodox by Lemma $\ref{zicu6}$. On the other hand, take $x=e, y=f$ and $z=ge$ in (\ref{2.3}). Then $efege=ee^\ast ff^\ast ee^\ast ge=ee^\ast ff^\ast ge=efge$ as $e$ is idempotent and $e^\ast=e$.  This together with Lemma \ref{1.2} (1) and the fact that $fgfe=fge$   implies that $efe, ege\in  P(S, \cdot)$ and
$efeege=efege=e(fge)=efgfe\in P(S,\cdot)$. By Lemma \ref{1.2} (2),  $efege=efe ege=ege efe=egefe$. Thus $(S,\cdot,\, \ast)$ is locally inverse by Lemma \ref{zicu4}.

(2) $\Longrightarrow$ (3). Let $(S,\cdot,\ast)$  be orthodox and locally inverse and $x,y,z\in S$. Then
$$xy+x(x^{\ast}+z)=xy(xy)^{\ast}x(x^{\ast}x^{\ast\ast}z)$$$$=xyy^{\ast}x^{\ast}xx^{\ast}xz=x(yy^{\ast}x^{\ast}x)z=x(x^{\ast}x yy^{\ast})z=xyy^{\ast}z=x(y+z)$$ by Lemma \ref{zicu8} and the fact that $x^\ast x, yy^\ast \in P(S)$. Thus (\ref{lw1}) holds.

(3) $\Longrightarrow$ (1). Assume that (\ref{lw1}) holds. Then by the last paragraph we  obtain that $xyy^{\ast}z=xyy^{\ast}x^{\ast}xz$ for all $x,y,z\in S$. Thus for all  $a,b,c\in S$,  we have
$$a+(b+c)=(aa^{\ast})bb^{\ast}c=aa^{\ast}bb^{\ast}(aa^{\ast})^{\ast}aa^{\ast}c =aa^{\ast}bb^{\ast}aa^{\ast}c=(aa^\ast b)(aa^\ast b)^\ast c=(a+b)+c.$$ This gives that $(S,+)$ is a semigroup.

{\em Part (B)}:  Assume that the axiom (\ref{lw2}) holds. Then the following axiom is satisfied:
\begin{equation}\label{2.4}
\begin{array}{cc}
zz^\ast yx=(z+y)x=(z+x^{\ast})x+yx\\[3mm]
=zz^\ast x^\ast x+yx =zz^\ast x^\ast x (zz^\ast x^\ast x)^\ast yx=zz^\ast x^\ast xzz^\ast yx.
\end{array}
\end{equation}
Let $z=y$ in (\ref{2.4}). Then $yx=yy^\ast yx=yy^\ast x^\ast x yy^\ast yx=yy^\ast x^\ast xyx$. By Lemma \ref{zicu5}, $(S,\cdot,\ast)$ is completely regular and locally inverse.  Conversely, if  $(S,\cdot,\ast)$ is completely regular and locally inverse, then the axiom $yy^\ast x^\ast xyx=yx$ holds by Lemma \ref{zicu5}. Let $a, b, c\in S$ and take $y=cc^\ast b$ and $x=a$. Then $cc^\ast b (cc^\ast b)^\ast a^\ast a cc^\ast b a=cc^\ast b a$. Since $(S,\cdot,\ast)$ is locally inverse and $cc^\ast, bb^\ast, a^\ast a\in P(S, \cdot)$,  we have $cc^\ast bb^\ast cc^\ast a^\ast a cc^\ast=cc^\ast a^\ast a cc^\ast bb^\ast cc^\ast$ by Lemma \ref{zicu4}. Moreover, $cc^\ast bb^\ast$ is idempotent by Lemma \ref{1.2} (4). Thus
$$cc^\ast b a=cc^\ast b (cc^\ast b)^\ast a^\ast a cc^\ast b a=cc^\ast b b^\ast cc^\ast a^\ast a cc^\ast b a$$$$= cc^\ast a^\ast a cc^\ast bb^\ast cc^\ast ba=cc^\ast a^\ast a cc^\ast bb^\ast cc^\ast bb^\ast ba=cc^\ast a^\ast a   cc^\ast bb^\ast ba=cc^\ast a^\ast a   cc^\ast ba.$$
By (\ref{2.4}), the axiom (\ref{lw2}) is satisfied.

{\em Part (C)}: It follows from Parts (A) and  (B) immediately.
\end{proof}

\begin{prop}\label{lww2}
Let $(S,\cdot,\, \ast)$ be a regular $\star$-semigroup. Define the binary operations``$+$" and  ``$\oplus$" on $S$ as follows: For all $a,b\in S$,
$a+b=a^{\ast}ba,\,\, a\oplus b=bab^\ast.$ Then the following statements are equivalent:
\begin{itemize}
\item[(1)]$(S, +, \cdot,\, \ast)$ (respectively, $(S, \oplus, \cdot,\, \ast)$) satisfies the axiom (\ref{lw1}): $x(y+z)=xy+x(x^{\ast}+z)$ (respectively, the axiom (\ref{lw1}): $x(y\oplus z)=xy\oplus x(x^{\ast} \oplus z)$).
\item[(2)]$(S, +, \cdot,\, \ast)$ (respectively, $(S, \oplus, \cdot,\, \ast)$) satisfies the axiom (\ref{lw2}): $(z+y)x=(z+x^{\ast})x+yx$ (respectively, the axiom (\ref{lw2}): $(z\oplus y)x=(z\oplus x^{\ast})x\oplus yx$).
\item[(3)]$(S,\cdot,\, \ast)$  is commutative.
\end{itemize}
If the above condition (3) holds,  then $(S, +)$ (respectively, $(S, \oplus)$) is a semigroup.  Thus $(S,+, \cdot,\, \ast)$ (respectively, $(S, \oplus, \cdot,\, \ast)$) is a left (or equivalently, right, two-sided) regular ${\star}$-semibrace if and only if $(S,\cdot,\, \ast)$  is commutative.
\end{prop}
\begin{proof} We only prove the case for  $(S, +, \cdot,\,  \ast)$, the case for  $(S,\oplus, \cdot,\,  \ast)$ can be proved similarly.

(1) $\Longrightarrow$ (3). Assume that (1) holds. Then $(S,\cdot,\ast)$ satisfies the axiom:
\begin{equation}\label{2.5.1}
xy^{\ast}zy=x(y+z)=xy+x(x^{\ast}+z)=xy+xx^{\ast\ast} zx^\ast=y^{\ast}x^{\ast}xxzx^{\ast}xy.
\end{equation}
Let $a\in S$ and take  $x=a^\ast$, $y=z=aa^\ast$ in (\ref{2.5.1}). Then
$$a^{\ast}=a^{\ast}aa^{\ast}=a^{\ast}(aa^{\ast})^{\ast}(aa^{\ast})(aa^{\ast})
\overset{(\ref{2.5.1})}{=}(aa^{\ast})^{\ast}a^{\ast\ast}a^{\ast}a^{\ast}(aa^{\ast})a^{\ast\ast}a^{\ast}(aa^{\ast})= aa^{\ast}a^{\ast},$$
and so $aa^\ast a^\ast a=a^\ast a\in P(S, \cdot)$. Replacing $a$ by $a^\ast$, we have $a^\ast aaa^\ast=a^\ast a^{\ast\ast} a^{\ast\ast} a^\ast=a^{\ast\ast}a^\ast=aa^\ast$.
By Lemma \ref{1.2} (2), we have $a^\ast a=aa^\ast a^\ast a=a^\ast a aa^\ast=aa^\ast$. Thus $(S,\cdot,\, \ast)$ is a Clifford semigroup by Lemma \ref{zicu10}.
Let $a,b\in S$ and take $x=b$, $y=a^\ast$ and  $z=a$ in (\ref{2.5.1}). Then
$$ba=baa^\ast a\overset{aa^\ast=a^\ast a}{=}baaa^\ast=ba^{\ast\ast}aa^\ast\overset{(\ref{2.5.1})}{=}a^{\ast\ast} b^\ast bb ab^\ast ba^\ast=ab^{\ast}bbab^{\ast}ba^{\ast}$$$$\overset{(\ref{zhongxin})}{=}ab^{\ast}bbb^{\ast}baa^{\ast}\overset{bb^\ast=b^\ast b}{=}abb^{\ast}bb^{\ast}baa^{\ast}=abaa^\ast\overset{(\ref{zhongxin})}{=}aaa^\ast b\overset{aa^\ast=a^\ast a}{=}aa^\ast ab=ab.$$

(2) $\Longrightarrow$ (3). Assume that (2) holds. Then the following  axiom is satisfied:
\begin{equation}\label{2.5.2}
z^{\ast}yzx=(z+y)x=(z+x^{\ast})x+yx=z^\ast x^\ast zx+yx=x^{\ast}z^{\ast}xzyxz^{\ast}x^{\ast}zx.
\end{equation}
Let $a\in S$ and take  $x=a, y=z=aa^\ast$   in (\ref{2.5.2}). Then
$$
a=aa^{\ast}a=(aa^{\ast})^{\ast}(aa^{\ast})(aa^{\ast})a
\overset{(\ref{2.5.2})}{=}a^{\ast}(aa^{\ast})^{\ast}a(aa^{\ast})(aa^{\ast})a(aa^{\ast})^{\ast}a^{\ast}(aa^{\ast})a
=a^{\ast}aaaa^{\ast}a^{\ast}a,$$ and so  $a^{\ast}aa=(a^{\ast}a)(a^{\ast}aaaa^{\ast}a^{\ast}a)=a^{\ast}aaaa^{\ast}a^{\ast}a=a$.
This gives  that $a^\ast a a a^\ast =aa^\ast \in P(S, \cdot)$.  Replacing $a$ by $a^\ast$, we have $aa^\ast a^\ast a=a^\ast a$.
By Lemma \ref{1.2} (2), we have $aa^\ast=a^\ast aaa^\ast=aa^\ast a^\ast a=a^\ast a$. Thus $(S,\cdot,\, \ast)$ is a Clifford semigroup by Lemma \ref{zicu10}.
Let $a,b\in S$ and take $x=a$, $y=b$ and $z=a^\ast$ in (\ref{2.5.2}). Then
$$ab=aa^\ast a b\overset{(\ref{zhongxin})}{=}aba^{\ast}a=a^{\ast\ast}ba^{\ast}a\overset{(\ref{2.5.2})}{=}a^{\ast}a^{\ast\ast}aa^{\ast}baa^{\ast\ast}a^{\ast}a^{\ast}a
=a^{\ast}aaa^{\ast}baaa^{\ast}a^{\ast}a$$$$\overset{(\ref{zhongxin})}{=}ba^{\ast}aaa^{\ast}aaa^{\ast}a^{\ast}a\overset{aa^\ast=a^\ast a}{=}baa^\ast aa^\ast aa^\ast aa^\ast a=ba.$$

(3) $\Longrightarrow$ (1), (2).
Assume that $(S,\cdot, \ast)$ is commutative and $x,y,z\in S$. Then $y^\ast x^\ast xxzx^\ast xy=xx^\ast xx^\ast xy^\ast zy=xy^\ast zy.$ This implies that (\ref{2.5.1}) holds. Similarly, we can show that (\ref{2.5.2}) also holds. So (1) and (2) are true.

Assume that $(S,\cdot, \ast)$ is commutative and  $a,b,c\in S$. Then
$$(a+b)+c=a^\ast b a+c=(a^\ast ba)^\ast c a^\ast ba=a^\ast b^\ast ac a^\ast b a$$$$=a^\ast a a^\ast ab^\ast bc=aa^\ast bb^\ast c=a^\ast b^\ast cba=a+(b^\ast cb)=a+(b+c).$$ This implies that $(S, +)$ is a semigroup.  Thus the final conclusion is true.
\end{proof}

\begin{prop}\label{li2.6}
Let $(S,\cdot,\, \ast)$ be a regular $\star$-semigroup. Define the binary operations ``$+$" and  ``$\oplus$" on $S$ as follows: For all $a,b\in S$,
$a+b=aba^{\ast},\,\,\, a\oplus b=b^\ast a b.$ Then the following statements are equivalent:
\begin{itemize}
\item[(1)]$(S,+, \cdot,\, \ast)$ (respectively, $(S,\oplus, \cdot,\, \ast)$) satisfies the axiom (\ref{lw1}): $x(y+z)=xy+x(x^{\ast}+z)$ (respectively, the axiom (\ref{lw1}): $x(y\oplus z)=xy\oplus x(x^{\ast}\oplus z)$).
\item[(2)]$(S,+, \cdot,\, \ast)$ (respectively, $(S,\oplus, \cdot,\, \ast)$) satisfies the axiom (\ref{lw2}): $(z+y)x=(z+x^{\ast})x+yx$ (respectively, the axiom (\ref{lw2}): $(z\oplus y)x=(z\oplus x^{\ast})x\oplus yx$).
\item[(3)]$(S,\cdot, \ast)$ is commutative.
\end{itemize}
If the above condition (3) holds,  then $(S, +)$ (respectively, $(S, \oplus)$) is a semigroup.  Thus $(S,+, \cdot,\, \ast)$ (respectively, $(S,\oplus, \cdot,\,  \ast)$) is a left (or equivalently, right, two-sided) regular ${\star}$-semibrace if and only if $(S,\cdot, \ast)$ is commutative.
\end{prop}
\begin{proof} We only prove the case for  $(S, +, \cdot,\,  \ast)$, the case for  $(S,\oplus, \cdot,\,  \ast)$ can be proved similarly.

(1) $\Longrightarrow$ (3).  Assume that (1) holds. Then $(S,\cdot,+,\ast)$ satisfies the axiom
\begin{equation}\label{2.6.1}
xyzy^{\ast}=x(y+z)=xy+x(x^{\ast}+z)=xyxx^{\ast}zxy^{\ast}x^{\ast}.
\end{equation}
Let $a,b\in S$ and take $x=a$, $y=z=a^\ast a$ in (\ref{2.6.1}). Then
$$a=aa^{\ast}a=a(a^{\ast}a)(a^{\ast}a)(a^{\ast}a)^{\ast}
\overset{(\ref{2.6.1})}{=}a(a^{\ast}a)aa^{\ast}(a^{\ast}a)a(a^{\ast}a)^{\ast}a^{\ast}
=aaa^\ast a^\ast aa a^\ast aa^\ast=aa a^{\ast}.
$$
This gives that  $a^\ast a a a^\ast =a^\ast a \in P(S, \cdot)$.  Replacing $a$ by $a^\ast$, we have $aa^\ast a^\ast a=aa^\ast$. By Lemma \ref{1.2} (2), we have $aa^\ast=a^\ast a$. Thus $(S,\cdot,\, \ast)$ is a Clifford semigroup by Lemma \ref{zicu10}.
Let $x=a^\ast$, $y=b$ and $z=a$ in (\ref{2.6.1}). Then
$$a^{\ast}bab^{\ast}\overset{(\ref{2.6.1})}{=}a^{\ast}ba^{\ast}a^{\ast\ast}aa^{\ast}b^{\ast}a^{\ast\ast}
=a^{\ast}ba^{\ast}a aa^{\ast}b^{\ast}a $$$$\overset{aa^\ast=a^\ast a}{=}
a^{\ast}baa^{\ast}aa^{\ast}b^{\ast}a\overset{(\ref{zhongxin})}{=}a^\ast bb^\ast a a^{\ast}aa^{\ast}a=a^\ast bb^\ast a\overset{(\ref{zhongxin})}{=}a^\ast abb^\ast,$$ and so
$ab=a(a^{\ast}abb^{\ast})b=a(a^{\ast}bab^{\ast})b\overset{(\ref{zhongxin})}{=}ba a^\ast a b^\ast b=ba b^\ast b\overset{(\ref{zhongxin})}{=}bb^\ast ba=ba.$

(2) $\Longrightarrow$ (3). Assume that (2) holds. Then $(S,\cdot,+,\, \ast)$ satisfies the axiom
\begin{equation}\label{2.6.2}
zyz^{\ast}x=(z+y)x=(z+x^{\ast})x+yx=zx^{\ast}z^{\ast}xyxx^{\ast}zxz^{\ast}.
\end{equation}
Let $a,b\in S$ and take $x=a$ and $y=z=aa^\ast$ in (\ref{2.6.2}). Then
$$a=(aa^{\ast})(aa^\ast)(aa^{\ast})^{\ast}a\overset{(\ref{2.6.2})}{=}(aa^{\ast})a^{\ast}(aa^{\ast})^{\ast}a(aa^\ast)(aa^{\ast})(aa^{\ast})a(aa^{\ast})^{\ast}
=aa^{\ast}a^{\ast}a aa^{\ast}aaa^{\ast},$$ and so $a\cdot aa^\ast=aa^{\ast}a^{\ast}a aa^{\ast}aaa^{\ast}\cdot aa^\ast=aa^{\ast}a^{\ast}a aa^{\ast}aaa^{\ast}=a$. This gives that $a^\ast a a a^\ast =a^\ast a \in P(S, \cdot)$.   Replacing $a$ by $a^\ast$, we have $aa^\ast a^\ast a=aa^\ast$. By Lemma \ref{1.2} (2), we have $aa^\ast=a^\ast a$. Thus $(S,\cdot,\, \ast)$ is a Clifford semigroup by Lemma \ref{zicu10}.
Let $x=b^\ast , y=b,  z=a^\ast$ in (\ref{2.6.2}). Then
$$a^\ast bab^\ast=a^\ast ba^{\ast\ast}b^\ast\overset{(\ref{2.6.2})}=a^{\ast}b^{\ast\ast}a^{\ast\ast}b^{\ast}bb^{\ast}b^{\ast\ast}a^{\ast}b^{\ast}a^{\ast\ast}=
a^{\ast}bab^{\ast}bb^{\ast}ba^{\ast}b^{\ast}a=
a^{\ast}bab^{\ast}ba^{\ast}b^{\ast}a$$$$\overset{(\ref{zhongxin})}=a^{\ast}bb^{\ast}b aa^{\ast}b^{\ast}a=a^{\ast}b
aa^{\ast}b^{\ast}a\overset{(\ref{zhongxin})}=a^{\ast}aa^{\ast} b b^{\ast}a=a^{\ast} b b^{\ast}a\overset{(\ref{zhongxin})}=a^{\ast}a b b^{\ast},$$ and so
$ab=a(a^{\ast}abb^{\ast})b=a(a^\ast bab^\ast)b\overset{(\ref{zhongxin})}=b(aa^\ast a)b^\ast b=bab^\ast b\overset{(\ref{zhongxin})}=bb^\ast ba=ba.$

(3) $\Longrightarrow$ (1), (2).  If $(S,\cdot,\,\ast)$ is commutative, then it is easy to check that (\ref{2.6.1}) and (\ref{2.6.2}) hold and so (1) and (2) are true.

Assume that $(S,\cdot, \ast)$ is commutative and  $a,b,c\in S$. Then
$$(a+b)+c=a b a^\ast+c=(a ba^\ast) c (a ba^\ast)^\ast=aba^\ast c a b^\ast a^\ast$$$$=aa^\ast aa^\ast bb^\ast c=aa^\ast bb^\ast c=abcb^\ast a^\ast=a+bcb^\ast =a+(b+c).$$ This implies that $(S, +)$ is a semigroup.  Thus the final conclusion is true.
\end{proof}
\begin{remark}\label{zhuji}In Propositions \ref{lww2}   and \ref{li2.6}, if $(S,+, \cdot,\,\ast)$ is commutative, then  $``+"$ and $``\oplus"$ are associative. However, the converse is not true in general. Consider the regular $\star$-semigroup $S$ appeared in Remark \ref{zhuji1}. In that regular $\star$-semigroup,  one can prove easily that the operations  $``+"$ and  $``\oplus"$ defined in  Propositions \ref{lww2} and \ref{li2.6}  are associative. Obviously, $(S,\cdot,\, \ast)$ is not commutative.
\end{remark}

\section{Solutions of the Yang-Baxter equation and the maps associated  to (2,2,1)-type algebras}\label{solution}
In this section, we  give some set-theoretic  solutions of the Yang-Baxter equation  by using the maps associated  to (2,2,1)-type algebras which are induced by regular $\star$-semigroups and obtained in the previous section. Let $(S,+,\cdot,\, \ast)$ be a (2,2,1)-type algebra. For all $a,b\in S$, define
\begin{equation}\label{lw4}\lambda_{a}: S\rightarrow S,\,\, b\mapsto a(a^{\ast}+b),\,\,\,\, \rho_{b}: S\rightarrow S,\,\, a\mapsto (a^{\ast}+b)^{\ast}b.
\end{equation}
That is, $\lambda_{a}(b)=a(a^{\ast}+b)$ and $\rho_{b}(a)=(a^{\ast}+b)^{\ast}b$ for $a,b\in S$. Then we call the map
\begin{equation}\label{lw41} r_{S}:S\times S\longrightarrow S\times S,\,\, (a,b)\longmapsto (\lambda_{a}(b),\rho_{b}(a))
\end{equation} {\em the map associated to the algebra $(S,+,\cdot,\,\ast)$}. Recall that $r_S$ is a solution of the Yang-Baxter equation if and only if (\ref{jie1})--(\ref{jie3}) in Section 1 hold.

\begin{prop}\label{wwliu3}Let $(S,\cdot,\, \ast)$ be a regular $\star$-semigroup. Define the binary operations ``$+$" and ``$\oplus$" on $S$ as follows: For all $a,b\in S$,  $a+b=ab$ and $a\oplus b=ba.$   Then the map $r_S$ associated to $(S,+,\cdot,\,\ast)$  $(\mbox{respectively, }(S,\oplus, \cdot,\, \ast))$ is a solution if and only if $(S,\cdot,\,\ast)$ is completely regular, orthodox and locally inverse.
\end{prop}
\begin{proof} We only prove the case for $(S,+,\cdot,\,\ast)$, the case for $(S,\oplus,\cdot,\,\ast)$ can be proved dually. Let $x,y,z\in S$. Then
$$\lambda_{x}(y)=x(x^{\ast}+y)=xx^{\ast}y,\,\, \rho_{y}(x)=(x^{\ast}+y)^{\ast}y=(x^{\ast}y)^{\ast}y=y^{\ast}xy.$$
Moreover, we have
\begin{equation}\label{zhang3}
\lambda_{\lambda_{x}(y)}\lambda_{\rho_{y}(x)}(z)=\lambda_{xx^\ast y}\lambda_{y^{\ast}xy}(z)
=xx^{\ast}yy^{\ast}xx^{\ast}y^{\ast}xyy^{\ast}x^{\ast}yz,\,\, \lambda_{x}\lambda_{y}(z)=xx^{\ast}yy^{\ast}z.
\end{equation}
If $(S,\cdot,\,\ast)$ is completely regular, orthodox and locally inverse,  by Corollary \ref{zhang2} we have   $$xx^{\ast}yy^{\ast}xx^{\ast}y^{\ast}xyy^{\ast}x^{\ast}yz=xx^{\ast}(y\cdot y^{\ast}xx^{\ast}y^{\ast}x\cdot yy^{\ast}\cdot x^{\ast}yz)
=xx^{\ast}yy^{\ast}(x\cdot x^{\ast}y^{\ast}\cdot xx^{\ast}\cdot yz)$$$$=x\cdot x^{\ast}yy^{\ast}\cdot xx^{\ast}\cdot y^{\ast}  yz
=xx^{\ast}(y\cdot y^{\ast} \cdot y^{\ast}  y\cdot z)=xx^{\ast}yy^{\ast}z.$$
By (\ref{zhang3}), we have $\lambda_{\lambda_{x}(y)}\lambda_{\rho_{y}(x)}(z)=\lambda_{x}\lambda_{y}(z)$.
Secondly,  by Corollary \ref{zhang2},
$$\rho_{\rho_{z}(y)}\rho_{\lambda_{y}(z)}(x)=(z^{\ast}yz)^{\ast}\cdot (yy^{\ast}z)^{\ast}\cdot x\cdot yy^{\ast}z\cdot z^{\ast}yz
=z^{\ast}y^{\ast}zz^{\ast}yy^{\ast}xyy^{\ast}zz^{\ast}yz$$
$$=z^{\ast}(y^{\ast}\cdot zz^{\ast}\cdot yy^{\ast}xyy^{\ast}\cdot zz^{\ast} \cdot y\cdot z) =z^{\ast}y^{\ast}  yy^{\ast}xyy^{\ast} y z
=z^{\ast}y^{\ast} x y z=\rho_{z}\rho_{y}(x).$$
Finally, we have
$\rho_{\lambda_{y}(z)}(x)=\rho_{yy^\ast z}(x)=(yy^\ast z)^\ast x yy^\ast z$ and so
$$\lambda_{\rho_{\lambda_{y}(z)}(x)}\rho_{z}(y)=(yy^{\ast}z)^{\ast}xyy^{\ast}z((yy^{\ast}z)^{\ast}xyy^{\ast}z)^{\ast}\cdot z^{\ast}yz
$$$$=z^{\ast}yy^{\ast}xyy^{\ast}zz^{\ast}yy^{\ast}x^{\ast}yy^{\ast}zz^{\ast}yz
=z^{\ast}\cdot yy^{\ast}xyy^{\ast} \cdot zz^{\ast}\cdot yy^{\ast}x^{\ast}yy^{\ast}\cdot zz^{\ast}\cdot yz $$$$=z^{\ast}yy^{\ast}xyy^{\ast} yy^{\ast}x^{\ast}yy^{\ast} yz=(z^{\ast}\cdot yy^{\ast}\cdot x\cdot yy^\ast\cdot  x^{\ast}\cdot  y)z=z^\ast xx^\ast yz $$ by Corollary \ref{zhang2}.
On the other hand, we have $\lambda_{\rho_{y}(x)}(z)=\lambda_{y^\ast xy}(z)=y^\ast x y (y^\ast xy)^\ast z$  and so
$$\rho_{\lambda_{\rho_{y}(x)}(z)}\lambda_{x}(y)
=(y^{\ast}xy(y^{\ast}xy)^{\ast}z)^{\ast}\cdot xx^{\ast}y
\cdot y^{\ast}xy(y^{\ast}xy)^{\ast}z $$
$$=z^{\ast}y^{\ast}xyy^{\ast}x^{\ast}yxx^{\ast}yy^{\ast}xyy^{\ast}x^{\ast}yz=z^{\ast}(y^{\ast}\cdot x\cdot yy^{\ast}\cdot x^{\ast}yxx^{\ast}\cdot yy^{\ast}\cdot x\cdot yy^{\ast}\cdot x^{\ast}yz)$$$$=z^{\ast}y^{\ast} x   x^{\ast}yxx^{\ast} yz =(z^{\ast}y^{\ast} \cdot x   x^{\ast}\cdot y  x \cdot  x^{\ast})yz=z^{\ast}y^{\ast}  y  x  x^{\ast}yz=(z^{\ast}\cdot y^{\ast}  y \cdot  x  x^{\ast}\cdot y)z= z^{\ast}  x  x^{\ast} y z $$ by Corollary \ref{zhang2}. Hence $\lambda_{\rho_{\lambda_{y}(z)}(x)}\rho_{z}(y)=z^\ast xx^\ast yz=\rho_{\lambda_{\rho_{y}(x)}(z)}\lambda_{x}(y)$.  Thus $r_S$ is a solution.

Conversely, assume that $r_S$ is a solution   and $x,y,z\in S$. Then by (\ref{zhang3}), we have
$xx^{\ast}yy^{\ast}z=xx^{\ast}yy^{\ast}xx^{\ast}y^{\ast}xyy^{\ast}x^{\ast}yz.$ Take  $y=x$ in the above equation, it follows that
$$xx^{\ast}z=xx^{\ast}xx^{\ast}z=xx^{\ast} xx^{\ast} xx^{\ast} x^{\ast}x  xx^{\ast} x^\ast x z=x(x^{\ast} x^{\ast}xxx^{\ast} x^{\ast})x z =xx^{\ast} x^{\ast}x z,$$  and so $x^{\ast}z=x^{\ast}xx^{\ast}z=x^{\ast}xx^{\ast}x^{\ast}xz=x^{\ast}x^{\ast}xz.$ Substituting $x$ by $x^\ast$, we have $xz=xxx^{\ast}z$. By Lemma \ref{zicu9},  $(S,\cdot,\,\ast)$ is completely regular, orthodox and locally inverse.
\end{proof}
\begin{prop}\label{lwwww}
Let $(S,\cdot,\, \ast)$ be a regular $\star$-semigroup. Define the binary operations ``$+$" and ``$\oplus$" on $S$ as follows: For all $a,b\in S$,  $a+b=a^\ast b^\ast,\, a\oplus b=b^\ast a^\ast$. If the map $r_S$ associated to $(S,+,\cdot,\,\ast)$  $(\mbox{respectively, }(S,\oplus, \cdot,\, \ast))$ is a solution, then $(S,\cdot,\,\ast)$ is completely regular, orthodox and locally inverse. If $(S,\cdot,\, \ast)$ satisfies the axiom $x=x^\ast$, then $r_S$ is a solution.
\end{prop}
\begin{proof}
We only prove the case for $(S,+,\cdot,\,\ast)$, the case for $(S,\oplus,\cdot,\,\ast)$ can be proved dually. Let $x,y,z\in S$.
Then $$\lambda_{x}(y)=x(x^{\ast}+y)=xx^{\ast\ast}y^\ast=xxy^\ast,\,\,\,\, \rho_{y}(x)=(x^{\ast}+y)^{\ast}y=(xy^\ast)^{\ast}y=yx^{\ast}y.$$
Let $r_S$ be a solution. Then  for all $x,y,z\in S$,
\begin{equation}\label{lll1}
\begin{array}{cc}
xxzy^\ast y^\ast =xx(yyz^\ast)^\ast=\lambda_{x}\lambda_{y}(z)=\lambda_{\lambda_{x}(y)}\lambda_{\rho_{y}(x)}(z)\\[3mm]
xxy^\ast xxy^\ast (yx^\ast y yx^\ast y z)^\ast=xxy^\ast xxy^\ast z y^\ast xy^\ast y^\ast xy^\ast,
\end{array}
\end{equation}
\begin{equation}\label{lll2}
\begin{array}{cc}
zy^\ast xy^\ast z=\rho_{z}(yx^\ast y)=\rho_{z}\rho_{y}(x)=\rho_{\rho_{z}(y)}\rho_{\lambda_{y}(z)}(x)\\[3mm]
zy^\ast z(yyz^\ast x^\ast yyz^\ast)^\ast zy^\ast z =zy^\ast zzy^\ast y^\ast xz y^\ast y^\ast zy^\ast z.
\end{array}
\end{equation}
Let $e,f,g\in P(S, \cdot)$.  By Lemma \ref{1.2} (1), (4),  we have  $fg, fe, ef, e, f, g\in E(S, \cdot)$ and  $e^\ast=e, f^\ast=f, g^\ast=g$. Take $x=fg, z=ef, y=f$ in (\ref{lll1}). Then $$fgef=fgfgefff=fgfg ef f^\ast f^\ast\overset{(\ref{lll1})}=fg fg f^\ast fg fg f^\ast  ef   f^\ast fg f^\ast  f^\ast fg f^\ast=fgfefgf.$$ In view of the fact $(efg)^\ast=g^\ast f^\ast e^\ast=gfe$ and $e, g, fg\in E(S, \cdot)$, we obtain that $$e(fgef)g=e(fgfefgf)g=efgfefg=efggfeefg=efg(efg)^\ast efg=efg.$$ which gives that $(S,\cdot,\,\ast)$ is orthodox by Lemma \ref{zicu6}. On the other hand, take $x=f, z=feg, y=e$ in (\ref{lll1}). Then
$ff  feg   e^\ast e^\ast =ff e^\ast ff e^\ast   feg   e^\ast f e^\ast e^\ast fe^\ast$ and so $fege=fegefe$ by the fact $fe\in E(S, \cdot)$.  This implies that $efege=efegefe$ and so
$egefe=(efege)^\ast=(efegefe)^\ast=efegefe=efege$.  By Lemma \ref{zicu4}, $(S,\cdot,\,\ast)$ is locally inverse.  Finally, Take $x=y=z$ in (\ref{lll2}).
Then $xx^\ast xx^\ast x=xx^\ast xxx^\ast x^\ast xx x^\ast x^\ast xx^\ast x$, and hence $x=xxx^\ast x^\ast x$. This gives that $x^\ast x=x^\ast xxx^\ast x^\ast x$. By Lemma \ref{zicu4}, $(S,\cdot,\,\ast)$ is  completely regular.

If $(S,\cdot,\, \ast)$  satisfies the axiom $x=x^\ast$, then by Corollary \ref{zhang1}, $(S,\cdot,\,\ast)$ is completely regular, orthodox and locally inverse. In view of the axiom $x=x^\ast$ and Proposition \ref{wwliu3}, we obtain that  $r_S$ is a solution in this case.
\end{proof}
\begin{remark}Let $(S,\cdot,\, \ast)$ be a regular $\star$-semigroup. Define the binary operations ``$+$" and ``$\oplus$" on $S$ as follows: For all $a,b\in S$,  $a+b=a^\ast b^\ast,\, a\oplus b=b^\ast a^\ast$.
The converse  of the two statements in Proposition \ref{lwwww} is not true. On one hand, let $G=\{e,x,y\}$ be a cyclic group, where $e$ is the identity. Define $e^{\ast}=e, x^{\ast}=x^{-1}=y, y^{\ast}=y^{-1}=x.$ Obviously, $(G,\cdot,\, \ast)$ is a completely regular, orthodox, locally inverse and commutative regular ${\star}$-semigroup. Since $(G, \cdot, \ast)$ is commutative, the operation $``+"$ and  $``\oplus"$ coincide. In this case, one can obtain that $\lambda_{e}\lambda_{x}(e)=x\not= e=\lambda_{\lambda_{e}(x)}\lambda_{\rho_{x}(e)}(e)$ by routine calculations. Thus, the map $r_G$ associated to $(G, \cdot, \ast)$ $(\mbox{respectively, }(G,\oplus, \cdot,\, \ast))$ is not a solution.
On the other hand, consider the regular $\star$-semigroup $(S,+,\cdot,\,\ast)$ appeared in Remark \ref{zhuji1}. It is routine to check that the map $r_S$ associated to $(S,+,\cdot,\,\ast)$  $(\mbox{respectively, }(S,\oplus, \cdot,\, \ast))$ is a solution. However, $(S,\cdot,\, \ast)$  does not satisfy the axiom $x=x^\ast$.
\end{remark}
\begin{prop}\label{wwliu4}Let $(S,\cdot,\, \ast)$ be a regular $\star$-semigroup. Define the binary operations ``$+$" and ``$\oplus$" on $S$ as follows: For all $a,b\in S$,  $a+b=a^\ast b,\, a\oplus b=b^\ast a.$   Then the map $r_S$ associated to $(S,+,\cdot,\,\ast)$  $(\mbox{respectively, }(S,\oplus, \cdot,\, \ast))$ is a solution if and only if $(S,\cdot,\,\ast)$ is completely regular, orthodox and locally inverse.
\end{prop}
\begin{proof}We only prove the case for $(S,+,\cdot,\,\ast)$, the case for $(S,\oplus,\cdot,\,\ast)$ can be proved dually. Let $x,y,z\in S$.
Then $$\lambda_{x}(y)=x(x^{\ast}+y)=xxy,\,\,\,\, \rho_{y}(x)=(x^{\ast}+y)^{\ast}y=(xy)^{\ast}y=y^{\ast}x^{\ast}y.$$ Moreover, we have
\begin{equation}\label{ll1}
\lambda_{x}\lambda_{y}(z)=\lambda_{x}(yyz)= xxyyz,\,\,\, \lambda_{\lambda_{x}(y)}\lambda_{\rho_{y}(x)}(z)=xxyxxyy^{\ast}x^{\ast}yy^{\ast}x^{\ast}yz,
\end{equation}
\begin{equation}\label{ll2}
\begin{array}{cc}
\rho_{z}\rho_{y}(x)=\rho_{z}(y^{\ast}x^{\ast}y)=z^{\ast}y^{\ast}xyz,\\[3mm]
 \rho_{\rho_{z}(y)}\rho_{\lambda_{y}(z)}(x)=(z^\ast y^\ast z)^\ast ((yyz)^\ast x^\ast yyz)^\ast z^\ast y^\ast z=z^{\ast}yzz^{\ast}y^{\ast}y^{\ast}xyyzz^{\ast}y^{\ast}z.
\end{array}
\end{equation}

Assume that $(S,\cdot,\,\ast)$ is completely regular, orthodox and locally inverse. Firstly,
$$
xxyxxyy^{\ast}x^{\ast}yy^{\ast}x^{\ast}yz=x((xy)\cdot x\cdot (xy)  (xy)^{\ast}\cdot yy^{\ast}x^{\ast}yz)
$$$$=xxyxyy^{\ast}x^{\ast}yz
=xx(y\cdot x\cdot yy^{\ast}\cdot x^{\ast}yz)=xxyxx^{\ast}yz
=x(x\cdot y\cdot x x^{\ast}\cdot yz)=xxyyz
$$ by Corollary \ref{zhang2}.
By (\ref{ll1}), we have $\lambda_{x}\lambda_{y}(z)=\lambda_{\lambda_{x}(y)}\lambda_{\rho_{y}(x)}(z)$. Secondly,
$$
z^{\ast}yzz^{\ast}y^{\ast}y^{\ast}xyyzz^{\ast}y^{\ast}z=z^{\ast}(yz\cdot z^{\ast}y^{\ast}y^{\ast}xy\cdot yz(yz)^{\ast}\cdot z)
$$$$=z^{\ast}yzz^{\ast}y^{\ast}y^{\ast}xyz
=z^{\ast}\cdot (yz)(yz)^{\ast}\cdot y^{\ast}x\cdot yz=z^{\ast}y^{\ast}xyz
$$by Corollary \ref{zhang2}.
By (\ref{ll2}),  $\rho_{z}\rho_{y}(x)=\rho_{\rho_{z}(y)}\rho_{\lambda_{y}(z)}(x)$. Finally,  ${\rho_{\lambda_{y}(z)}(x)}=(yyz)^\ast x^\ast yyz$ and
$$
 \lambda_{\rho_{\lambda_{y}(z)}(x)}\rho_{z}(y)=\lambda_{(yyz)^\ast x^\ast yyz}(z^\ast y^\ast z)
=((yyz)^\ast x^\ast\cdot  yyz  (yyz)^\ast\cdot  x^\ast \cdot yyz)z^\ast y^\ast z
$$$$=(yyz)^{\ast}x^{\ast}x^{\ast}yyz(z^{\ast}y^{\ast}z)=(yz)^{\ast}\cdot y^{\ast}x^{\ast}x^{\ast}y\cdot (yz)  (yz)^{\ast}\cdot z
=(yz)^{\ast}y^{\ast}x^{\ast}x^{\ast}yz=z^{\ast}y^{\ast}y^{\ast}x^{\ast}x^{\ast}yz
$$ by Corollary \ref{zhang2}.
Moreover, we have $\lambda_{\rho_{y}(x)}(z)=y^\ast x^\ast y y^\ast x^\ast y z$ and by Corollary \ref{zhang2},
\begin{eqnarray*}
&&\rho_{\lambda_{\rho_{y}(x)}(z)}\lambda_{x}(y)=\rho_{y^\ast x^\ast y y^\ast x^\ast y z}(xxy)
=(y^\ast x^\ast y y^\ast x^\ast y z)^\ast (xxy)^\ast (y^\ast x^\ast y y^\ast x^\ast y z)\\
&=&z^{\ast}y^{\ast}xyy^{\ast}xyy^{\ast}x^{\ast}x^{\ast}y^{\ast}x^{\ast}yy^{\ast}x^{\ast}yz
=z^{\ast}(y^{\ast}\cdot x \cdot y y^{\ast}\cdot x\cdot y y^{\ast}\cdot x^{\ast}x^{\ast}y^{\ast}x^{\ast}\cdot yy^{\ast}\cdot x^{\ast}yz) \\
&=&z^{\ast}y^{\ast} x   x x^{\ast}x^{\ast}y^{\ast}x^{\ast}  x^{\ast}yz= (z^{\ast}y^{\ast} x   \cdot x x^{\ast} \cdot x^{\ast}y^{\ast}\cdot x^{\ast})  x^{\ast}yz= (z^{\ast}y^{\ast} x    x^{\ast}y^{\ast}  x^{\ast})  x^{\ast}yz\\
&=&(z^{\ast}y^{\ast} \cdot x    x^{\ast}\cdot y^{\ast} \cdot  x^{\ast})  x^{\ast}yz=z^{\ast}y^{\ast} y^{\ast} x^{\ast}  x^{\ast}yz.
\end{eqnarray*}
This implies that $\lambda_{\rho_{\lambda_{y}(z)}(x)}\rho_{z}(y)=\rho_{\lambda_{\rho_{y}(x)}(z)}\lambda_{x}(y)$. Thus $r_S$ is a solution.

Conversely, let $r_S$ be a solution. Then by (\ref{ll1}) and (\ref{ll2}), the following axioms hold:
\begin{equation}\label{ll3}
 xxyyz=xxyxxyy^{\ast}x^{\ast}yy^{\ast}x^{\ast}yz,\,\,\,z^{\ast}y^{\ast}xyz=z^{\ast}yzz^{\ast}y^{\ast}y^{\ast}xyyzz^{\ast}y^{\ast}z.
\end{equation}
Replacing $x$ and $z$ by $yy^\ast$ and $y^\ast y$ in the second identity in (\ref{ll3}), respectively, by direct calculations we have $y^\ast y=y^\ast y yy^\ast y^\ast y$, and so $(S, \cdot, \ast)$ is completely regular by Lemma \ref{zicu7}.

Let $e,f,g\in P(S, \cdot)\subseteq E(S,\cdot)$ and take $x=e, y=fg$ and $z=g$ in the first identity in (\ref{ll3}). Observe that $(efg)^\ast=g^\ast f^\ast e^\ast=gfe$ and $(fg)^\ast=g^\ast f^\ast=gf, (fg)^2=fg$ by Lemma \ref{1.2} (1), (4),  it follows that
$$efg=eefgfgg\overset{(\ref{ll3})}{=}eefgeefg(fg)^\ast e^\ast fg (fg)^\ast e^\ast fgg=efgefggfefggfefg$$$$=
efg efg gfe efg gfe efg=efg efg (efg)^\ast efg (efg)^\ast efg=efgefg, $$  which implies that $(S, \cdot, \ast)$ is orthodox by Lemma \ref{zicu6}.
Finally, assume that $x, y, z$ are idempotents and $y=z$ in the first identity in (\ref{ll3}). Since  $(S, \cdot, \ast)$ is orthodox,  we have
$xy=xyy^\ast x^\ast y$ and so  $$y^\ast x^\ast y=(xy)^\ast y=(xy)^\ast xy (xy)^\ast y=y^\ast x^\ast (xy y^\ast x^\ast y)=y^\ast x^\ast xy=(xy)^\ast xy.$$
This implies that $y^\ast x y=(y^\ast x^\ast y)^\ast=((xy)^\ast xy)^\ast=(xy)^\ast xy=y^\ast x^\ast y.$  Since  $(S, \cdot, \ast)$ is orthodox and
$e,f,g\in P(S, \cdot)\subseteq E(S, \cdot)$,  we have $feg, e\in E(S, \cdot)$. Take $x=feg$ and $y=e$.  Then $efege=e^\ast feg e=y^\ast x y=y^\ast x^\ast y =e^\ast (feg)^\ast e=egefe$. By Lemma \ref{zicu4},
$(S, \cdot, \ast)$ is  locally inverse.
\end{proof}
\begin{prop}\label{wwliu5}Let $(S,\cdot,\, \ast)$ be a regular $\star$-semigroup. Define the binary operations ``$+$" and ``$\oplus$" on $S$ as follows: For all $a,b\in S$,  $a+b=ab^\ast,\, a\oplus b=ba^\ast.$   Then  $r_S$ associated to $(S,+,\cdot,\,\ast)$  $(\mbox{respectively, }(S,\oplus, \cdot,\,  \ast))$ is a solution  if and only if $(S,\cdot,\,\ast)$ is completely regular, orthodox and locally inverse.
\end{prop}
\begin{proof}
We only prove the case for $(S,+,\cdot,\,\ast)$, the case for $(S,\oplus,\cdot,\,\ast)$ can be proved dually. Let $x,y,z\in S$.
Then $$\lambda_{x}(y)=x(x^{\ast}+y)=xx^\ast y^\ast,\,\,\,\, \rho_{y}(x)=(x^{\ast}+y)^{\ast}y=(x^\ast y^\ast )^{\ast}y=yxy.$$ Moreover, we have
\begin{equation}\label{ll4}
\begin{array}{cc}
\lambda_{x}\lambda_{y}(z)=\lambda_{x}(yy^\ast z^\ast)= xx^\ast z yy^\ast,\\[3mm]
 \lambda_{\lambda_{x}(y)}\lambda_{\rho_{y}(x)}(z)=xx^\ast y^\ast (xx^\ast y^\ast)^\ast(yxy(yxy)^\ast z^\ast)^\ast=xx^\ast y^\ast yx x^\ast z yxy y^\ast x^\ast y^\ast,
\end{array}
\end{equation}
\begin{equation}\label{ll5}
\rho_{z}\rho_{y}(x)=\rho_{z}(yxy)=z y xyz,\,\,\,
\rho_{\rho_{z}(y)}\rho_{\lambda_{y}(z)}(x)=zyz   y y^{\ast}z^\ast  xyy^\ast z^{\ast}zyz,
\end{equation}
\begin{equation} \label{ll6}
\begin{array}{cc}
\lambda_{\rho_{\lambda_{y}(z)}(x)}\rho_{z}(y)=\lambda_{yy^\ast z^\ast xyy^\ast z^\ast}(zyz)=yy^\ast z^\ast xy y^\ast z^\ast zyy^\ast x^\ast zyy^\ast z^\ast y^\ast z^\ast, \\[3mm]
\rho_{\lambda_{\rho_{y}(x)}(z)}\lambda_{x}(y)=\rho_{yxy (yxy)^\ast z^\ast}(xx^\ast y^\ast)=yxyy^\ast x^\ast y^\ast z^\ast xx^\ast y^\ast yxyy^\ast x^\ast y^\ast z^\ast.
\end{array}
\end{equation}
Assume that $(S,\cdot,\,\ast)$ is completely regular, orthodox and locally inverse. Then by using Lemma \ref{zicu9} and some calculations similar to those used in Propositions  \ref{wwliu3} and \ref{wwliu4}, we can obtain that $$\lambda_{x}\lambda_{y}(z)=\lambda_{\lambda_{x}(y)}\lambda_{\rho_{y}(x)}(z),\,\,\, \rho_{z}\rho_{y}(x)=\rho_{\rho_{z}(y)}\rho_{\lambda_{y}(z)}(x),$$
$$\lambda_{\rho_{\lambda_{y}(z)}(x)}\rho_{z}(y)=yy^\ast z^\ast xx^\ast y^\ast z^\ast=\rho_{\lambda_{\rho_{y}(x)}(z)}\lambda_{x}(y)$$ by (\ref{ll4})--(\ref{ll6}). Thus $r_S$ is a solution.

Conversely, let $r_S$ be a solution. Then by (\ref{ll4}) and (\ref{ll5}), the following axioms hold:
\begin{equation}\label{ll7}
xx^\ast y^\ast yx x^\ast z yxy y^\ast x^\ast y^\ast= xx^\ast z yy^\ast,\,\,\, zyz   y y^{\ast}z^\ast  xyy^\ast z^{\ast}zyz=z y xyz.
\end{equation}
Take $y=z=xx^\ast$ in the first axiom in (\ref{ll7}). Then $$xxx^\ast x^\ast=xx^\ast (xx^\ast)^\ast xx^\ast xx^\ast xx^\ast xx^\ast x xx^\ast (xx^\ast)^\ast x^\ast (xx^\ast)^\ast\overset{(\ref{ll7})}{=} xx^\ast xx^\ast xx^\ast (xx^\ast)=xx^\ast,$$ which implies that $x^\ast (xxx^\ast x^\ast) x=x^\ast (xx^\ast) x=x^\ast x$, and so $(S,\cdot,\,\ast)$ is completely regular by Lemma \ref{zicu7}. Let $e,f,g\in P(S, \cdot)$. Take $x=e, y=eg$ and $z=f$ in the first axiom in (\ref{ll7}). By Lemma \ref{1.2} (1), (4),  we have $eg\in E(S, \cdot)$ and
$$egefege=eegeegeefegeeggeege=ee^\ast (eg)^\ast egee^\ast f eg e eg  (eg)^\ast e^\ast (eg)^\ast$$$$\overset{(\ref{ll7})}{=} ee^\ast f eg(eg)^\ast=eefegge=efege,$$
which implies that $egefe=e^\ast g^\ast e^\ast f^\ast e^\ast=(efege)^\ast=(egefege)^\ast=egefege=efege.$ In view of Lemma \ref{zicu4}, $(S,\cdot,\,\ast)$ is locally inverse.
Finally, take $x=e, y=fg=z$ in the second  axiom in (\ref{ll7}). Then by Lemma \ref{1.2} (1), (3), (4), we have $x, y, z, x^\ast, y^\ast, z^\ast \in E(S, \cdot)$ and so
$fggfefg=(fg)(fg)^\ast e fg=(fg)e(fg)=fgefg$. This implies  that $e(fggfefg)=e(fgefg)$. Since $(efg)^\ast=g^\ast f^\ast e^\ast=gfe$, it follows that
$$efgefg=efggfefg=efggfeefg=efg(efg)^\ast efg=efg.$$ Thus $(S,\cdot,\,\ast)$ is orthodox by Lemma \ref{zicu6}.
\end{proof}

\begin{prop}\label{wwliu1}Let $(S,\cdot,\, \ast)$ be a regular ${\star}$-semigroup.  Define the binary operation ``$+$"  on $S$ as follows: For all $a,b\in S$, $a+b=aa^{\ast}b.$   Then the map $r_S$ associated to $(S,\cdot, +,\, \ast)$ is a solution.
\end{prop}
\begin{proof}Let $x,y,z\in S$. Then we have $\lambda_{x}(y)=x(x^{\ast}+y)=xx^{\ast}(x^{\ast})^{\ast}y=xx^{\ast}xy=xy$ and
$$\rho_{y}(x)=(x^{\ast}+y)^{\ast}y=(x^{\ast}(x^{\ast})^{\ast}y)^{\ast}y=(x^{\ast}xy)^{\ast}y=y^{\ast}x^{\ast}xy=(xy)^\ast xy.$$
Observe that
$$\lambda_{x}\lambda_{y}(z)=\lambda_{x}(yz)=xyz=xy(xy)^{\ast}xyz=\lambda_{xy}\lambda_{(xy)^{\ast}xy}(z)=\lambda_{\lambda_{x}(y)}\lambda_{\rho_{y}(x)}(z);$$
$$\rho_{z}\rho_{y}(x)=\rho_{z}((xy)^{\ast}xy)=((xy)^{\ast}xyz)^{\ast}(xy)^{\ast}xyz=(xyz)^{\ast}xy(xy)^{\ast}xyz=(xyz)^{\ast}xyz, $$
$$ \rho_{\rho_{z}(y)}\rho_{\lambda_{y}(z)}(x)=\rho_{(yz)^{\ast}yz}\rho_{yz}(x)=\rho_{(yz)^{\ast}yz}((xyz)^{\ast}xyz)
$$$$=((xyz)^{\ast}xyz (yz)^{\ast}yz)^{\ast}(xyz)^{\ast}xyz (yz)^{\ast}yz=((xyz)^{\ast}xyz)^{\ast}(xyz)^{\ast}(xyz)=(xyz)^{\ast}xyz;$$
$$\lambda_{\rho_{\lambda_{y}(z)}(x)}\rho_{z}(y)=\lambda_{\rho_{yz}(x)}\rho_{z}(y)
 =\lambda_{(xyz)^{\ast}xyz}((yz)^{\ast}yz)=(xyz)^{\ast}xyz(yz)^{\ast}yz=(xyz)^{\ast}xyz,$$
$$\rho_{\lambda_{\rho_{y}(x)}(z)}\lambda_{x}(y)=\rho_{\lambda_{(xy)^{\ast}xy}(z)}\lambda_{x}(y)=\rho_{(xy)^{\ast}xyz}(xy)=(xy (xy)^{\ast}xyz)^{\ast}xy
(xy)^{\ast}xyz=(xyz)^{\ast}xyz.$$
Thus $r_S$   is a solution.
\end{proof}
\begin{prop}\label{wwliu2}Let $(S,\cdot,\, \ast)$ be a regular ${\star}$-semigroup.  Define the binary operation ``$+$"  on $S$ as follows: For all $a,b\in S$, $a+b=ab^{\ast}b.$   Then the map $r_S$ associated to $(S,\cdot, +,\, \ast)$ is a solution  if and only if  $(S,\cdot,\, \ast)$   is completely regular, orthodox and locally inverse.
\end{prop}
\begin{proof}Let $x,y,z\in S$.
Then $$\lambda_{x}(y)=x(x^{\ast}+y)=xx^\ast y^\ast y,\,\,\,\, \rho_{y}(x)=(x^{\ast}+y)^{\ast}y=(x^\ast y^\ast y)^{\ast}y=y^\ast yxy.$$ Moreover, we have
\begin{equation}\label{lll4}
\begin{array}{cc}
\lambda_{x}\lambda_{y}(z)=\lambda_{x}(yy^\ast z^\ast z)= xx^\ast z^\ast z yy^\ast z^\ast z,\\[2mm]
 \lambda_{\lambda_{x}(y)}\lambda_{\rho_{y}(x)}(z)=xx^\ast y^\ast y xx^\ast z^\ast z y^\ast y x yy^\ast x^\ast y^\ast y   z^\ast z,
\end{array}
\end{equation}
\begin{equation}\label{lll5}
\begin{array}{cc}
\rho_{z}\rho_{y}(x)=\rho_{z}(y^{\ast}y x y)=z^\ast z y^\ast y  xyz,\\[2mm]
\rho_{\rho_{z}(y)}\rho_{\lambda_{y}(z)}(x)=z^\ast y^\ast z^\ast z yz yy^\ast z^\ast z xy y^\ast z^\ast z yz,
\end{array}
\end{equation}
\begin{equation} \label{lll6}
\begin{array}{cc}
\lambda_{\rho_{\lambda_{y}(z)}(x)}\rho_{z}(y)=z^\ast z yy^\ast z^\ast z x yy^\ast z^\ast z yy^\ast x^\ast z^\ast z yy^\ast z^\ast   y^\ast z^\ast z yz, \\[2mm]
\rho_{\lambda_{\rho_{y}(x)}(z)}\lambda_{x}(y)=z^\ast z y^\ast y xy y^\ast x^\ast y^\ast yx yy^\ast x^\ast y^\ast y z^\ast z xx^\ast y^\ast y xy y^\ast x^\ast y^\ast y z^\ast z.
\end{array}
\end{equation}
Assume that $(S,\cdot,\,\ast)$ is completely regular, orthodox and locally inverse. Using Corollary  \ref{zhang2} and some calculations similar to those used in Propositions  \ref{wwliu3} and \ref{wwliu4}, we can obtain
$$\lambda_{x}\lambda_{y}(z)=xx^\ast   yy^\ast z^\ast z=xx^\ast   y^\ast  y z^\ast z=\lambda_{\lambda_{x}(y)}\lambda_{\rho_{y}(x)}(z),\,\,\, \rho_{z}\rho_{y}(x)=z^\ast z xyz=\rho_{\rho_{z}(y)}\rho_{\lambda_{y}(z)}(x),$$
$$\lambda_{\rho_{\lambda_{y}(z)}(x)}\rho_{z}(y)=z^\ast zxx^\ast z^\ast y^\ast y z=z^\ast zxx^\ast y^\ast y z^\ast y^\ast yz= z^\ast zxx^\ast y^\ast y z^\ast  z =\rho_{\lambda_{\rho_{y}(x)}(z)}\lambda_{x}(y)$$ by (\ref{lll4})--(\ref{lll6}). Thus $r_S$ is a solution.

Conversely, let $r_S$ be a solution. Then by (\ref{lll4}) and (\ref{lll5}), the following axioms hold:
\begin{equation}\label{lll7}
\begin{array}{cc}
xx^\ast z^\ast z yy^\ast z^\ast z=xx^\ast y^\ast y xx^\ast z^\ast z y^\ast y x yy^\ast x^\ast y^\ast y  z^\ast z,\\[2mm]
z^\ast z y^\ast y  xyz=z^\ast y^\ast z^\ast z yz yy^\ast z^\ast z xy y^\ast z^\ast z yz.
\end{array}
\end{equation}
Take $y=z=xx^\ast$ in the first axiom in (\ref{lll7}). Then
$$xx^\ast (xx^\ast)^\ast xx^\ast xx^\ast (xx^\ast)^\ast (xx^\ast)^\ast xx^\ast
$$$$=xx^\ast (xx^\ast)^\ast (xx^\ast) xx^\ast (xx^\ast)^\ast (xx^\ast) (xx^\ast)^\ast (xx^\ast) x (xx^\ast)(xx^\ast)^\ast x^\ast (xx^\ast)^\ast (xx^\ast)   (xx^\ast)^\ast (xx^\ast).$$
This implies that $xx^\ast=xx^\ast xxx^\ast x^\ast xx^\ast=xx x^\ast x^\ast,$
which gives that $x^\ast (xxx^\ast x^\ast) x=x^\ast (xx^\ast) x=x^\ast x$, and so $(S,\cdot,\,\ast)$ is completely regular by Lemma \ref{zicu7}.

Let $e,f,g\in P(S, \cdot)$ and take $x=e, y=g$ and $z=fe$ in the first axiom in (\ref{lll7}). Then
$$ee^\ast (fe)^\ast fe gg^\ast (fe)^\ast (fe)=ee^\ast g^\ast g ee^\ast (fe)^\ast (fe) g^\ast g e gg^\ast e^\ast g^\ast g  (fe)^\ast (fe),$$ which yields  that
$efegefe=ege\cdot efe \cdot ege \cdot ege \cdot ege \cdot  efe$ by Lemma \ref{1.2} (1).  By Lemma \ref{1.2} (1), (4),  we have $ege, efe \in P(S, \cdot)$ and $ege efe\in E(S, \cdot)$, and so $efegefe=egeefe=egefe$. This implies that $efege=(egefe)^\ast=(efegefe)^\ast=efegefe=egefe.$ In view of Lemma \ref{zicu4}, $(S,\cdot,\,\ast)$ is locally inverse.

Let $e,f,g\in P(S, \cdot)$ and take $x=g, y=f, z=e$ in the second  axiom in (\ref{lll7}). Then
$e^\ast e f^\ast f  gfe=e^\ast f^\ast e^\ast e fe ff^\ast e^\ast e gf f^\ast e^\ast e fe$. Since $e, f, ef, fe\in E(S, \cdot)$ and $e^\ast=e, f^\ast=f$ by Lemma \ref{1.2} (1), (4), we have $efgfe=efegfe$, and so $$efgefe=(efegfe)^\ast =(efgfe)^\ast=efgfe=efegfe,$$
$$efg=efg(efg)^\ast efg=(efg gfe) efg=(efgfe) efg=(efgefe) efg=efg (efeef)g=efgefg.$$
Thus $(S,\cdot,\,\ast)$ is orthodox by Lemma \ref{zicu6}.
\end{proof}
\begin{prop}\label{lwwwww}
Let $(S,\cdot,\, \ast)$ be a regular $\star$-semigroup. Define the binary operations ``$+$" and $``\oplus"$ on $S$ as follows: For all $a,b\in S$,  $a+b=aba^\ast, \, a\oplus b=bab^\ast.$ If the map $r_S$ associated to $(S,+,\cdot,\,\ast)$  $(\mbox{respectively, }(S,\oplus, \cdot,\, \ast))$ is a solution, then $(S,\cdot,\,\ast)$ is completely regular, orthodox and locally inverse. If $(S,\cdot,\, \ast)$ is commutative, then $r_S$ is a solution.
\end{prop}
\begin{proof}
We only prove the case for $(S,+,\cdot,\,\ast)$, the case for $(S,\oplus,\cdot,\,\ast)$ can be proved dually. Let $x,y,z\in S$.
Then $$\lambda_{x}(y)=x(x^{\ast}+y)=xx^\ast y x^{\ast\ast} =xx^\ast y x,\,\,\,\, \rho_{y}(x)=(x^{\ast}+y)^{\ast}y=(x^\ast y x^{\ast\ast})^{\ast}y=x^{\ast}y^\ast xy.$$
Let $r_S$ be a solution. Then  for all $x,y,z\in S$,
\begin{equation}\label{llll1}
\begin{array}{cc}
xx^\ast yy^\ast zyx=\lambda_{x}(yy^\ast zy) =\lambda_{x}\lambda_{y}(z)=\lambda_{\lambda_{x}(y)}\lambda_{\rho_{y}(x)}(z)\\[2mm]
=\lambda_{xx^\ast yx} \lambda_{x^\ast y^\ast xy}(z)=xx^\ast yxx^\ast y^\ast xx^\ast x^\ast y^\ast xy y^\ast x^\ast yxz x^\ast y^\ast xy xx^\ast yx,
\end{array}
\end{equation}
\begin{equation}\label{llll2}
\begin{array}{cc}
y^\ast x^\ast yx z^\ast x^\ast y^\ast xyz=\rho_z(x^\ast y^\ast xy)=\rho_{z}\rho_{y}(x)=\rho_{\rho_{z}(y)}\rho_{\lambda_{y}(z)}(x) \\[2mm]
=\rho_{y^\ast z^\ast yz}\rho_{yy^\ast zy} (x)=y^\ast z^\ast yy^\ast x^\ast yy^\ast zyx z^\ast y^\ast zy   x^\ast y^\ast z^\ast yy^\ast xy y^\ast z y y^\ast z^\ast yz.
\end{array}
\end{equation}
Let $y=z=xx^\ast$ in  (\ref{llll1}). Then
$$xx^\ast (xx^\ast)(xx^\ast)^\ast (xx^\ast) (xx^\ast)x=xx^\ast (xx^\ast)xx^\ast (xx^\ast)^\ast xx^\ast x^\ast$$$$\cdot\,\, (xx^\ast)^\ast x(xx^\ast) (xx^\ast)^\ast x^\ast (xx^\ast)x(xx^\ast) x^\ast (xx^\ast)^\ast x(xx^\ast) xx^\ast (xx^\ast)x.$$ This gives that $x=xx^\ast x^\ast xx$ and so $xx^\ast=xx^\ast x^\ast xxx^\ast$. By Lemma \ref{zicu4}, $(S,\cdot,\,\ast)$ is  completely regular.
Let $e,f,g\in P(S, \cdot)$ and take $x=e, y=f, z=fg$ in (\ref{llll1}). Then
$$ee^\ast ff^\ast (fg) fe=ee^\ast f ee^\ast f^\ast ee^\ast e^\ast f^\ast eff^\ast e^\ast fe (fg) e^\ast f^\ast ef ee^\ast fe$$
and so $efgfe=efgefe$ by the fact that $e=e^\ast, f=f^\ast$ and $e,f,ef,fe\in E(S, \cdot)$ (using Lemma \ref{1.2}). This implies that
\begin{equation}\label{llll3}
efgefe=efgfe=(efgfe)^\ast=(efgefe)^\ast =efe gfe,
\end{equation}
$$efg=efg(efg)^\ast efg=efggfe efg$$$$=efgfefg=(efgfe)fg=(efgefe)fg=efg(efef)g=efgefg.$$
Thus $(S,\cdot,\,\ast)$ is orthodox by Lemma \ref{zicu6}.
Take $x=f, y=e$ and $z=g$ in   (\ref{llll2}). Then
\begin{equation}\label{llll4}e^\ast f^\ast ef g^\ast f^\ast e^\ast feg=e^\ast g^\ast ee^\ast f^\ast ee^\ast gef g^\ast e^\ast ge   f^\ast e^\ast g^\ast ee^\ast f ee^\ast g ee^\ast g^\ast eg.
\end{equation}
Since $(S,\cdot,\,\ast)$ is orthodox and $e,f,g$ are projections, we have $e^\ast=e, f^\ast=f, g^\ast=g$ and $ef, fe, ge, egef, gefe, gef, feg$ are also idempotents. So by (\ref{llll3}), we obtain that
$$e^\ast f^\ast ef g^\ast f^\ast e^\ast feg=efefgfefeg=(efgfe)g=(efegfe)g=e(fegfeg)=efeg.$$
As $ge, egef, gefe, feg, gef$ are idempotents and $efegfe=efgfe$ (by (\ref{llll3})),
$$e^\ast g^\ast ee^\ast f^\ast ee^\ast gef g^\ast e^\ast ge   f^\ast e^\ast g^\ast ee^\ast f ee^\ast g ee^\ast g^\ast eg=
egef egef \cdot gegefegefe\cdot gege\cdot g$$$$=egef \cdot ge fe\cdot geg=e\cdot gefgef\cdot egeg=egefeg=ege(fegfeg)$$$$=eg(efegfe)g=eg(efgfe)g=egefgfeg.$$
In view of (\ref{llll4}), we have $efeg=egefgfeg$,  and so $efege=egefgfege$. This implies that
$$efege=egefgfege=(egefgfege)^\ast=(efege)^\ast=egefe.$$ By Lemma \ref{zicu4}, $(S,\cdot,\,\ast)$ is locally inverse.

If $(S,\cdot,\, \ast)$ is commutative, then by Corollary \ref{zhang1}, $(S,\cdot,\,\ast)$ is completely regular, orthodox and locally inverse. In view of Propositions \ref{wwliu1} and \ref{wwliu2}, we obtain that  $r_S$ is a solution.
\end{proof}

\begin{remark}Let $(S,\cdot,\, \ast)$ be a regular $\star$-semigroup. Define the binary operations ``$+$" and $``\oplus"$ on $S$ as follows: For all $a,b\in S$,  $a+b=aba^\ast, \, a\oplus b=bab^\ast.$ If $(S,\cdot,\,\ast)$ is completely regular, orthodox and locally inverse, then  $r_S$  may not be a solution.  In fact, consider the symmetric group $(S_{3},\cdot, \ast)$, where $S_{3}=\{(1),(12),(13),(23),(123),(132)\}$ and $x^\ast$ is the inverse of $x$ in this group for all $x\in S_3$.  Obviously, $(S_{3},\cdot, \ast)$ is a completely regular, orthodox and locally inverse regular $\star$-semigroup.
Moreover, by (\ref{llll1}) we have  $$\lambda_{(12)}\lambda_{(13)}((12))=(12)(13)(12)=(23)\not= (13)
=\lambda_{\lambda_{(12)}((13))}\lambda_{\rho_{(13)}((12))}((12))$$ in $(S_{3},+, \cdot, \ast)$, and so $r_{S_3}$ associated to  $(S_{3},+, \cdot, \ast)$ is not a solution.  Dually, we can show that  $r_{S_3}$ associated to  $(S_{3},\oplus, \cdot, \ast)$ is not a solution.

On the other hand, assume that  $r_S$ is a solution. Then $(S, \cdot, \ast)$ may not be commutative. For example, consider the regular $\star$-semigroup $(S, +, \cdot, \ast)$ in Remark \ref{zhuji1}. Then  $a+b=aa^\ast$  for all $a,b\in S$, and in $(S, +, \cdot, \ast)$, we have $\lambda_x(y)=x$  and $\rho_y(x)=x^\ast y$. Moreover,  we  obtain
$$\lambda_{x}\lambda_{y}(z)=x=\lambda_{\lambda_{x}(y)}\lambda_{\rho_{y}(x)}(z),\,\, \rho_{z}\rho_{y}(x)=y^\ast z=\rho_{\rho_{z}(y)}\rho_{\lambda_{y}(z)}(x)$$ and $\lambda_{\rho_{\lambda_{y}(z)}(x)}\rho_{z}(y)=x^\ast y=\rho_{\lambda_{\rho_{y}(x)}(z)}\lambda_{x}(y)$ by routine calculations. So  $r_S$ associated to $(S, +, \cdot, \ast)$ is a solution. Dually, $r_S$ associated to $(S, \oplus, \cdot, \ast)$ is also a solution. However, $(S,  \cdot, \ast)$ is not commutative obviously.
\end{remark}

\begin{remark}\label{wwliu6}Let $(S,\cdot,\, \ast)$ be a regular ${\star}$-semigroup.  Define the binary operations ``$+$" and $``\oplus"$ on $S$ as follows: For all $a,b\in S$,  $a+b=a^\ast b a$ and $a\oplus b=b^\ast a b$. If $(S,\cdot,\,  \ast)$ is commutative, then by the statements in the last paragraph in the proof of Proposition \ref{lwwwww}, we obtain that  $r_S$ is a solution. However, up to now, we have not obtained satisfactory sufficient and  necessary (or necessary) conditions under which $r_S$ is a solution.
\end{remark}

\begin{remark} By  Propositions  \ref{lw3}, \ref{2.9}, \ref{2.11}, \ref{wliu2}, \ref{lww3}, \ref{lww2} and \ref{li2.6} in Section \ref{brace} and Propositions  \ref{wwliu3}, \ref{lwwww}, \ref{wwliu4},  \ref{wwliu5},  \ref{wwliu1},  \ref{wwliu2}, \ref{lwwwww} and Remark \ref{wwliu6} in Section \ref{solution}, we have seen that the map associated to left and two-sided  regular $\star$-semibraces involved in Section \ref{brace} are all solutions of the Yang-Baxter equation by Corollary \ref{zhang1}. Observe that the right regular $\star$-semibrace $(S, \oplus, \cdot \ast)$  provided in Proposition \ref{lww3} may not be a solution by Proposition \ref{wwliu2}.
\end{remark}

Let $(S,+,\cdot,\, ^\ast)$ be a left regular $\star$-semibrace. Denote the set of maps from $S$ to itself by ${\mathcal T}(S)$.  It is well known that
$({\mathcal T}(S), \cdot)$ forms a semigroup with the composition of maps (i.e. $(\alpha \beta)(x)=\alpha(\beta(x))$ for all $x\in S$ and $\alpha, \beta \in {\mathcal T}(S)$)
Assume that  $a,b\in S$ and $\lambda_a$ and $\rho_b$ are defined as in (\ref{lw4}).
Then we have the following maps:
$$\lambda: S\longrightarrow {\mathcal T}({S}),\,\,a\longmapsto \lambda_{a},\,\,\,\,  \rho: S\longrightarrow {\mathcal T}({S}),\,\, b\longmapsto \rho_{b}.$$
In the following statements, we shall consider some properties of these maps.

\begin{prop}\label{+zitongtai}
Let $(S,+,\cdot,\,^\ast)$ be a left regular $\star$-semibrace and $a, b, c \in S$.  Then $$\lambda_{a}(b+c)=\lambda_{a}(b)+\lambda_{a}(c).$$
\end{prop}
\begin{proof}
In fact, by the axiom (\ref{lw1}) we have $$\lambda_{a}(b+c)=a(a^\ast+b+c)=a(a^\ast+b)+a(a^\ast+z)=\lambda_{a}(b)+\lambda_{a}(c),$$ as required.
\end{proof}
Let $(S,+,\cdot,\, ^\ast)$ be a left regular $\star$-semibrace and $a\in S$.  By Proposition \ref{+zitongtai}, $\lambda_a$ is an endomorphism of $(S, +)$.
The following example shows that $\lambda_a$ may not be an anti-endomorphism of $(S, +)$ even if  $(S,+,\cdot,\, ^\ast)$ is a skew left brace.
\begin{ex}Consider symmetric group $(S_{3},\cdot)$, where $$S_{3}=\{(1),(12),(13),(23),(123),(132)\}.$$ Define a binary operation $``+"$ on $S_{3}$ as follows:  For all $x,y\in S_{3}$,  $x+y=xy$. Then one can see that $(S_{3},+,\cdot)$ is a skew left brace easily. In this skew left brace, we have $$\lambda_{(12)}((13)+(12))=\lambda_{(12)}((123))=(12)((12)^{-1}+(123))=(12)(12)(123)=(123),$$
$$\lambda_{(12)}((12))=(12)((12)^{-1}+(12))=(12)(12)(12)=(12),$$
$$\lambda_{(12)}((13))=(12)((12)^{-1}+(13))=(12)(12)(13)=(13).$$ This implies that $$\lambda_{(12)}((12))+\lambda_{(12)}((13))=(12)+(13)=(12)(13)=(132)\not=(123)=\lambda_{(12)}((13)+(12)),$$  and so  $\lambda_{(12)}$ is not an anti-endomorphism of $(S_{3},+)$.
\end{ex}
From \cite[Lemma 2.12]{Jespers-Van Antwerpen}, $\lambda$ is a morphism from $(S, \cdot)$ to $({\mathcal T}(S), \cdot)$ for a left semibrace (in particular, left brace, skew left brace)  $(S,+,\cdot)$.
However, the situation is different for left inverse  semibraces and left regular $\star$-semibraces, which can be illustrated by the following example.
\begin{ex}Define two binary operations $``\cdot"$ and $``+"$ on $S=\{0,1\}$  as follows:
$$0\cdot 0=0\cdot 1=1\cdot 0=0,\, 1\cdot 1=1,\,\,\,\, 0+1=0+0=0,\, 1+0=1+1.$$
Then one can easily show that $(S,+,\cdot)$ forms a left inverse semibrace, and $0^{-1}=0, 1^{-1}=1$ in $(S, \cdot)$.  In this left inverse semibrace, we have  $$\lambda_{1\cdot 0}(1)=\lambda_{0}(1)=0\cdot(0^{-1}+1)=0\cdot (0+1)=0\cdot 0=0,$$ $$\lambda_{1}(\lambda_{0}(1))=\lambda_{1}(0)=1\cdot (1^{-1}+0)=1\cdot (1+0)=11=1.$$ This implies that  $\lambda$ is not a morphism from $(S,\cdot)$ to $({\mathcal T}(S), \cdot)$.
\end{ex}
Let $(S, +, \cdot)$ be a left brace and $x\in S$. The following example shows
that $\rho_x$  may not be an endomorphism (respectively, an anti-endomorphism) of $(S, +)$, and $\lambda$ (respectively, $\rho$) may not be an anti-morphism (respectively, a morphism) from $(S, \cdot)$ to $({\mathcal T}(S), \cdot)$.
\begin{ex}Consider the dihedral group $(D_{8},\cdot)$, where $$D_{8}=\langle a,b\mid a^{4}=b^{2}=1, b^{-1}ab=a^{-1}\rangle=\{e,a,a^{2},a^{3},b,ba,ba^{2},ba^{3}\}.$$ Define a binary operation $``+"$ on $D_{8}$ as follows:
$$\begin{tabular}{c|cccccccc}
+: & $e$ & $a$ & $a^{2}$ & $a^{3}$ & $b$ & $ba$ & $ba^{2}$ & $ba^{3}$\\
\hline
    $e$ & $e$ & $a$ & $a^{2}$ & $a^{3}$ & $b$ & $ba$ & $ba^{2}$ & $ba^{3}$ \\
    $a$ & $a$ &$e$ & $ba^{2}$ & $ba^{3}$ & $ba$ & $b$ & $a^{2}$ & $a^{3}$ \\
    $a^{2}$ & $a^{2}$ & $ba^{2}$ & $e$ & $b$ & $a^{3}$ & $ba^{3}$ & $a$ & $ba$ \\
    $a^{3}$ & $a^{3}$ & $ba^{3}$ & $b$ & $e$ & $a^{2}$ & $ba^{2}$ & $ba$ & $a$ \\
    $b$ & $b$ & $ba$ & $a^{3}$ & $a^{2}$ &$e$ & $a$ & $ba^{3}$ & $ba^{2}$ \\
    $ba$ & $ba$ & $b$ & $ba^{3}$ & $ba^{2}$ & $a$ &$e$ & $a^{3}$ & $a^{2}$ \\
    $ba^{2}$ & $ba^{2}$ & $a^{2}$ & $a$ & $ba$ & $ba^{3}$ & $a^{3}$ & $e$& $b$ \\
    $ba^{3}$ & $ba^{3}$ & $a^{3}$ & $ba$ & $a$ & $ba^{2}$ & $a^{2}$ & $b$ & $e$
\end{tabular} \hspace{.5cm}$$
It is easy to see that $(D_{8},+)\cong \mathbb{Z}_{2}\times \mathbb{Z}_{2}\times \mathbb{Z}_{2}$, and one can check that $(D_{8},+,\cdot)$ forms a left brace.
Since $$\rho_{b}(ba)=((ba)^{-1}+b)^{-1}b=(ba+b)^{-1}b=(a)^{-1}b=ba,$$
$$\rho_{b}(b)=(b^{-1}+b)^{-1}b=(b+b)^{-1}b=(e)^{-1}b=b,$$ we have $\rho_{b}(ba)+\rho_{b}(b)=ba+b=a$.
Observe that $$\rho_{b}(ba+b)=\rho_{b}(a)=(a^{-1}+b)^{-1}b=(a^{3}+b)^{-1}b=(a^{2})^{-1}b=ba^{2},$$
it follows that $\rho_{b}$ is not an endomorphism on $(D_{8},+)$. As $(D_{8},+)$ is commutative,   $\rho_{b}$ is not an anti-endomorphism on $(D_{8},+)$.
Moreover, we have $$\lambda_{b(ba)}(b)=\lambda_{a}(b)=a(a^{-1}+b)=a(a^{3}+b)=aa^{2}=a^{3},$$
$$\lambda_{ba}(\lambda_{b}(b))=\lambda_{ba}(b(b^{-1}+b))=\lambda_{ba}(b(b+b))=\lambda_{ba}(be)$$$$=\lambda_{ba}(b)=ba((ba)^{-1}+b)=ba(ba+b)=baa=ba^{2}.$$   This implies that $\lambda$ is  not an anti-morphism from $(D_{8},\cdot)$ to $({\mathcal T}(D_{8}), \cdot)$. Finally, observe that
$$\rho_{b(ba)}(b)=\rho_{a}(b)=(b^{-1}+a)^{-1}a=(b+a)^{-1}a=(ba)^{-1}a=ba^{2},$$
$$\rho_{b}(\rho_{ba}(b))=\rho_{b}((b^{-1}+ba)^{-1}(ba))=\rho_{b}((a+ba)^{-1}(ba))=\rho_{b}(a^{-1}ba)$$$$=\rho_{b}(ba^{2})
=((ba^{2})^{-1}+b)^{-1}b=(ba^{2}+b)^{-1}b=(ba^{3})^{-1}b=a.$$
it follows that $\rho$ is not a morphism from $(D_{8},\cdot)$ to $({\mathcal T}(D_{8}), \cdot)$.
\end{ex}
Let $(S, +, \cdot)$ be a left cancellative semibrace (in particular, skew left brace) and $b\in S$. From \cite[Proposition 6]{c24}, $\rho_b$ is an anti-morphism from $(S, \cdot)$ from $({\mathcal T}(S), \cdot)$. However, the situation is different for general left semibraces, which can be illustrated by the following example.
\begin{ex}Consider the cyclic group $(S,\cdot)$, where $S=\langle a\rangle=\{e,a,a^2,a^3\}$. Define a binary operation $``+"$ on $S$ as follows:
$$\begin{tabular}{c|cccc}
+ & $e$ & $a$ & $a^{2}$ & $a^{3}$\\
\hline
    $e$ & $e$ & $e$ & $a^{2}$ & $a^{2}$\\
    $a$ & $a$ &$a$ & $a^{3}$ & $a^{3}$\\
    $a^{2}$ & $a^{2}$ & $a^{2}$ & $e$ & $e$\\
    $a^{3}$ & $a^{3}$ & $a^{3}$ & $a$ & $a$
\end{tabular} \hspace{.5cm}$$
One can show that $(S,+,\cdot)$ forms a left semibrace. In this left semibarce, we have $$\rho_{aa}(a)=\rho_{a^{2}}(a)=(a^{-1}+a^{2})^{-1}a^{2}=(a^{3}+a^{2})^{-1}a^{2}=a^{-1}a^{2}=a^3 a^2=a,$$
$$\rho_{a}(\rho_{a}(a))=\rho_{a}((a^{-1}+a)^{-1}a)=\rho_{a}((a^{3}+a)^{-1}a)=\rho_{a}((a^{3})^{-1}a)$$$$=\rho_{a}(a^{2})
=((a^{2})^{-1}+a)^{-1}a=(a^{2}+a)^{-1}a=(a^{2})^{-1}a=a^2 a=a^{3}.$$
This implies that $\rho$ is not an anti-morphism from $(S,\cdot)$ to $({\mathcal T}(S), \cdot)$.
\end{ex}

Let $(S,+,\cdot)$ be a left brace and $x\in S$. In the following example, we shall show that $\lambda_x$ (respectively, $\rho_x$) may not be an endomorphism (or anti-endomorphism) of $(S, \cdot)$.
\begin{ex}Consider the additive group $(\mathbb{Z}_{8},+)$ of integers modulo $8$, where $\mathbb{Z}_{8}=\langle1\rangle=\{0,1,2,3,4,5,6,7\}$. Define a binary operation ``$\cdot$" on $\mathbb{Z}_{8}$ as follows:
$$\begin{tabular}{r|rrrrrrrr}
$\cdot$ & 0 & 1 & 2 & 3 & 4 & 5 & 6 & 7\\
\hline
    0 & 0 & 1 & 2 & 3 & 4 & 5 & 6 & 7 \\
    1 & 1 & 0 & 7 & 6 & 5 & 4 & 3 & 2 \\
    2 & 2 & 7 & 4 & 1 & 6 & 3 & 0 & 5 \\
    3 & 3 & 6 & 1 & 4 & 7 & 2 & 5 & 0 \\
    4 & 4 & 5 & 6 & 7 & 0 & 1 & 2 & 3 \\
    5 & 5 & 4 & 3 & 2 & 1 & 0 & 7 & 6 \\
    6 & 6 & 3 & 0 & 5 & 2 & 7 & 4 & 1 \\
    7 & 7 & 2 & 5 & 0 & 3 & 6 & 1 & 4
\end{tabular}$$
One can check that $(\mathbb{Z}_{8},\cdot)\cong (\mathbb{Z}_{2}, +)\times (\mathbb{Z}_{4}, +)$, and $(\mathbb{Z}_{8}, +, \cdot)$ is a left brace with $0^{-1}=0 $ and $1^{-1}=1, 2^{-1}=6$. In this left brace, we have
$$\lambda_{1}(1\cdot1)=\lambda_{1}(0)=1\cdot (1^{-1}+0)=1\cdot(1+0)=1\cdot 1=0.$$
$$\lambda_{1}(1)=1\cdot (1^{-1}+1)=1\cdot (1+1)=1\cdot2=7,\,\,\, \lambda_{1}(1)\cdot \lambda_{1}(1)=7\cdot7=4.$$
This implies that  $\lambda_{1}$ is not an endomorphism of $(Z_{8},\cdot)$. Since $(Z_{8},\cdot)$ is commutative, it follows that  $\lambda_{1}$ is not an anti-endomorphism of $(Z_{8},\cdot)$. On the other hand, observe that
$$\rho_{1}(1\cdot 1)=\rho_{1}(0)=(0^{-1}+1)^{-1}\cdot 1=(0+1)^{-1}\cdot 1=1^{-1}\cdot 1=1\cdot 1=0,$$
$$\rho_{1}(1)=(1^{-1}+1)^{-1}\cdot 1=(1+1)^{-1}\cdot 1=2^{-1}\cdot 1=6\cdot 1=3,\,\,\,\, \rho_{1}(1)\cdot \rho_{1}(1)=3\cdot 3=4.$$ This together with the fact that $(\mathbb{Z}_{8},\cdot)$ is commutative implies that $\rho_{1}$ is neither an endomorphism nor an anti-endomorphism of  $(\mathbb{Z}_{8},\cdot)$
\end{ex}


\section{Set-theoretic solutions of the Yang-Baxter equation associated to weak left $\star$-braces}
In this section, we introduce   weak left  $\star$-braces which generalize
 weak left braces and form  a subclass of left regular $\star$-semibraces. After giving some
necessary  structural properties of weak left  $\star$-braces, we show that the map  associated to a weak left $\star$-brace  is always a solution  of the Yang-Baxter equation. Assume that $(S,+,\, -)$ and $(S,\cdot,\, \ast)$ are regular $\star$-semigroups, where $-: S\rightarrow S,\, x\mapsto -x$ and $\ast: S\rightarrow S,\, x\mapsto x^\ast$. To avoid parentheses,  throughout this section we always  assume that the multiplication has higher precedence than the addition and write $x+(-y)$ as $x-y$ for all $x, y\in S$.

\begin{defn} Let $(S,+,\, -)$ and $(S,\cdot,\, \ast)$ be regular $\star$-semigroups. Then $(S,+,\cdot,\, -,\, \ast)$ is called a {\em weak  left $\star$-brace} if the following axiom holds:
\begin{equation}\label{lw5}x(y+z)=xy-x+xz,\,\,\, -x+x=xx^\ast.
\end{equation}
\end{defn}
\begin{remark}\label{tousheji}Let  $(S,+,\cdot,\, -,\, \ast)$ be a weak  left $\star$-brace. By Lemma \ref{1.2} (1) and the second identity in (\ref{lw5}),  the sets of projections $P(S,+)$ and $P(S, \cdot)$ of $(S, +, -)$ and $(S, \cdot, \ast)$ coincide. That is,
\begin{equation}\label{toushejiyizhi}
\begin{array}{cc}
P(S,+)=\{e\in E(S, +)\mid -e=e\}=\{x-x\mid x\in S\}=\{-x+x\mid x\in S\}\\[2mm]
=\{xx^\ast\mid x\in S\}=\{x^\ast x\mid x\in S\}=\{f\in E(S, \cdot)\mid f^\ast=f \}=P(S, \cdot).
\end{array}
\end{equation}
We  denote it by $P(S)$ and call  it {\em the set of projections of the  weak  left $\star$-brace $(S,+,\cdot,\, -,\, \ast)$}.
\end{remark}
\begin{remark}\label{5.3}
Assume that $(S, +)$ and $(S, \cdot)$ are inverse semigroups. For each $x\in S$,  denote  the inverses of $x$ in  the inverse semigroups $(S, +)$ and  $(S, \cdot)$ by $-x$ and $x^\ast$, respectively. Then $(S, +, -)$ and $(S, \cdot, \ast)$ forms   regular $\star$-semigroups. Conversely, let $(S,+,\, -)$ and $(S,\cdot,\, \ast)$ be regular $\star$-semigroups which are also inverse. Then by Remark \ref{inverse-regular*}, for each $x\in S$, $x^\ast$ and $-x$ must be the unique inverses of $x$ in $(S, \cdot)$ and $(S, +)$, respectively. Thus  weak  left braces (in particular, skew left braces) are necessarily weak  left  $\star$-braces, and a weak  left $\star$-brace $(S,+,\cdot,\, -,\, \ast)$  is a  weak left brace if and only if $(S, +)$ and $(S, \cdot)$ are inverse semigroups.
\end{remark}
The following proposition gives a kind of weak left $\star$-braces which induced by regular $\star$-semigroups.

\begin{prop}\label{lizi}Let $(S,\cdot,\, \ast)$ be a regular $\star$-semigroup. Define two binary operations  $``+"$, $``\oplus"$ and two unary operations $``-"$, $\ominus$ on $S$ as follows: For all $a,b\in S$, $a+b=ab, -a=a^\ast$ and $a\oplus b=ba, \ominus a=a^\ast$. Then $(S, \oplus, \ominus)$ also forms a regular $\star$-semigroup. Moreover, we have the following results:
\begin{itemize}
\item[(1)] $(S,+, \cdot,\,-, \ast)$ is a weak left $\star$-brace if and only if $(S, \cdot,\ast)$ is a Clifford semigroup.
\item[(2)] $(S,\oplus, \cdot,\,\ominus, \ast)$ is a weak left $\star$-brace if and only if $(S,\cdot,\, \ast)$ is completely regular, orthodox and locally inverse.
\end{itemize}
\end{prop}
\begin{proof}Obviously, $(S, \oplus, \ominus)$ is  a regular $\star$-semigroup. Now we prove items (1) and (2). Let $x,y,z\in S$.

(1) If $(S,+, \cdot,\,-, \ast)$ is a weak left $\star$-brace, then by (\ref{lw5}),  we have $x^\ast x=-x+x= xx^\ast$. This implies that $S$ is a Clifford semigroup by Lemma \ref{zicu10}. Conversely, if $(S, \cdot,\ast)$ is a Clifford semigroup, then $-x+x=x^\ast x=xx^\ast$ by Lemma \ref{zicu10} again, and
$x^\ast xy=yx^\ast x$ by the fact that $E(S, \cdot)$ is contained in the center of $(S, \cdot,\ast)$.  This yields that $$xy-x+xz=xy+(-x)+xz=xyx^\ast xz=xx^\ast x yz=xyz=x(y+z).$$ Thus
$(S,+, \cdot,\,-, \ast)$ is a weak left $\star$-brace.

(2) Obviously, we have $\ominus x\oplus x=xx^\ast$. Observe that $x(y\oplus z)=xzy$ and $$xy\ominus x\oplus xz=xy\oplus (\ominus x)\oplus xz=xy\oplus x^\ast \oplus xz=xzx^\ast xy.$$
Thus $(S,\oplus, \cdot,\,\ominus, \ast)$ is a weak left $\star$-brace if and only if $(S,\cdot,\,\ast)$ satisfies the axiom $xzy=xzx^\ast xy$. By Lemma \ref{zicu9}, this is equivalent to the fact that  $(S,\cdot,\, \ast)$ is completely regular, orthodox and locally inverse.
\end{proof}
\begin{remark}Since completely regular, orthodox and locally inverse regular $\star$-semigroups (for example, the regular $\star$-semigroup $(S, \cdot,\, \ast)$ appeared in Remark \ref{zhuji1}) are not necessarily inverse,  there are weak left $\star$-braces which are not weak left braces by Proposition \ref{lizi}.
\end{remark}

The following lemma collect some basic equalities in  weak left $\star$-braces which will be used  frequently   in the sequel.
\begin{lem}\label{lw6} Let $(S,+,\cdot,-,\ast)$ be a  weak left $\star$-brace,  $a,b\in S$ and $e\in P(S)$.
\begin{itemize}
\item[(1)]   $a^{\ast\ast}=a,\, aa^\ast a=a,\, (ab)^\ast=b^\ast a^\ast, a^\ast aa^\ast=a^\ast.$
\item[(2)] $-(-a)=a, a-a+a=a,\, -(a+b)=-b-a,\, -a+a-a=-a.$
\item[(3)]$-aa^\ast=aa^\ast=aa^\ast+aa^\ast,\, -a^\ast a=a^\ast a=a^\ast a+a^\ast  a.$
\item[(4)] $(a-a)^\ast=a-a=(a-a)(a-a),\, (-a+a)^\ast=-a+a=(-a+a)(-a+a)$.
\item[(5)] $a+aa^\ast=a=a-aa^\ast,\,  aa^\ast-a=-a$.
\item[(6)] $(-a)(-a)^\ast=a-a,\, (a-a)(-a)=-a,\,(-a+a)a=a$.
\item[(7)] $(-a+a)(a-a)=a+(-a+a)(-a),\, (a-a)(-a+a)=-a+(a-a)a$.
\item[(8)] $ab-ab+a(b+c)=a(b+c)=a(b+a^\ast ac)$.
\item[(9)] $e^\ast=-e=ee=e+e=e-e=ee^\ast=e^\ast e=-e+e=e$.
\item[(10)]$-a+ab=a(a^\ast+b)$.
\end{itemize}
\end{lem}
\begin{proof}By the definition of regular $\star$-semigroups, (\ref{lw5}) and Remark \ref{tousheji}, items (1)--(6) are obvious. Now, we consider items (7), (8), (9)  and  (10). In fact, by  (\ref{lw5}) we have
$$ (-a+a)(a-a) =(-a+a)a-(-a+a)+(-a+a)(-a) $$$$=aa^\ast a-a+a+(-a+a)(-a) =a-a+a+(-a+a)(-a)=a+(-a+a)(-a).$$
This gives the first identity in (7). Replacing $a$ by $-a$ in the first identity, we can obtain  the second identity.
Again by  (\ref{lw5}) we have
$$ab-ab+a(b+c)=ab-ab+(ab-a+ac)=ab-a+ac=a(b+c)=ab-a+aa^{\ast}ac=a(b+a^\ast ac),$$
which gives item (8). Item (9) follows from items (3), (4) and Remark \ref{tousheji}. Finally, by (\ref{lw5}) and item (5), we have $a(a^\ast+ b)=aa^\ast -a +ab=-a+ab$, which gives (10).
\end{proof}
The following two propositions explore  the relationship between left regular $\star$-semibraces, weak left $\star$-braces and weak left braces, which generalize and enrich  the results in \cite[Proposition 16]{Catino-Mazzotta-Miccoli-Stefanelli} on weak left braces.

\begin{prop}\label{zuozhengze}Let $(S,\cdot, \ast)$ and $(S,+, -)$ be two  regular $\star$-semigroups. Then $(S,+,\cdot, -,\ast)$ is a weak left $\star$-brace if and only if $(S,+,\cdot,\ast)$ is a left regular $\star$-semibrace  and $-x+xy=x(x^{\ast}+y)$ for all $x,y\in S$.
\end{prop}
\begin{proof}Let $(S,+,\cdot, -,\ast)$ be a weak left $\star$-brace.  By Lemma \ref{lw6} (10), we have $-x+xy=x(x^{\ast}+y)$ for all $x,y\in S$.  This together with (\ref{lw5}) gives that
$$x(y+z)=xy-x+xz=xy+x(x^\ast+z)$$ for all $x,y,z\in S$, and so   the axiom (\ref{lw1}) holds. Thus $(S,+,\cdot,\ast)$ is a left regular $\star$-semibrace.

Conversely, let $(S,+,\cdot,\ast)$ be  a left regular $\star$-semibrace which satisfies the axiom  $-x+xy=x(x^{\ast}+y)$.    Then by (\ref{lw1}), $x(y+z)=xy+x(x^\ast+z)=xy-x+xz,$  and so the first identity in (\ref{lw5}) holds. Now we consider the second identity in (\ref{lw5}).  Let $x\in S$. Since $(S,+, -)$ and $(S,\cdot,\ast)$ are regular $\star$-semigroups, we have $-x+x-x=-x,\, x=-(-x)$ and $xx^\ast x=x.$ By the given axiom and the first identity in (\ref{lw5}), we have
$$-x+x=-x+xx^\ast x=x(x^\ast+x^\ast x)=xx^\ast -x+x x^\ast x=xx^\ast -x+x,$$
\begin{equation}\label{lw13}
x=-(-x)=-(-x+x-x)=-(xx^\ast -x+x-x)=-(xx^\ast -x)=x-xx^\ast.
\end{equation}
Replacing $x$ by $xx^\ast$ in (\ref{lw13}), we have  $xx^\ast=xx^\ast-xx^\ast (xx^\ast)^\ast=xx^\ast-xx^\ast$, and so $-xx^\ast=-(xx^\ast-xx^\ast)=xx^\ast-xx^\ast=xx^\ast.$
This implies that $xx^\ast=xx^\ast-xx^\ast=xx^\ast+xx^\ast $ and $ x=x-xx^\ast=x+xx^\ast$ by (\ref{lw13}).  Thus
\begin{eqnarray*}
xx^{\ast}&=&xx^{\ast}xx{^\ast}=xx^{\ast}(xx^{\ast}+xx^{\ast})\,\,\, \,\,\,\,(\mbox{since }xx^\ast=xx^\ast+xx^\ast)\\
&=&x(x^{\ast}(xx^{\ast}+xx^{\ast}))=x(x^{\ast}xx^{\ast}+x^{\ast}(x+xx^{\ast}))\,\,\, \,\,\,\,(\mbox{by (\ref{lw1}) and }x^{\ast\ast}=x)\\
&=&x(x^{\ast}+x^{\ast}x)\,\,\, \,\,\,\,(\mbox{since } x^\ast=x^\ast x x^\ast \mbox{and }x=x+xx^{\ast})\\
&=&-x+xx^{\ast}x=-x+x.\,\,\, \,\,\,\,(\mbox{by the given condition and }xx^\ast x=x)
\end{eqnarray*}
Thus the second identity in (\ref{lw5}) is also true, and so $(S,+,\cdot, -,\ast)$ is a weak left $\star$-brace.
\end{proof}
\begin{prop}\label{zuozhengze2}Let $(S, +, \cdot, \ast)$  be  a left regular $\star$-semibrace  such that $(S, +)$ is an inverse semigroup. Then $(S, +)$ is  a Clifford semigroup  and   $(S, +, \cdot)$ is a weak left brace.
\end{prop}
\begin{proof} Let $e,f\in P(S, \cdot)$. Then $e^\ast=e=ee$ and so $$e=ee=e(e-e+e)=ee+e(e^\ast-e+e)=e+e(e-e+e)=e+ee=e+e\in E(S, +)$$ by (\ref{lw1}).
Thus $P(S, \cdot)\subseteq E(S, +)$. This implies that
\begin{equation}\label{ww1} e-e=e=-e=e+e=ee=e^\ast,\,\,\, e+f=f+e \mbox{ for all } e,f\in P(S, \cdot)
\end{equation} by (\ref{inverse}) and the fact that $(S, +)$ is an inverse semigroup.
Let $x, y\in S$. Then $x^\ast x\in P(S, \cdot)\subseteq E(S, +)$, and so $x^\ast x+x^\ast x=x^\ast x$ by (\ref{ww1}).
By the axiom (\ref{lw1}),
\begin{equation}\label{ww2}
x+x(x^\ast+x^\ast x)=xx^\ast x+x(x^\ast+x^\ast x)=x(x^\ast x+x^\ast x)=xx^\ast x=x,
\end{equation}
$$x(x^\ast+x^\ast x)=x(x^\ast+x^\ast x+x^\ast x)=x(x^\ast+x^\ast x)+x(x^\ast+x^\ast x)\in E(S,+),$$ which implies that $-x(x^\ast+x^\ast x)=x(x^\ast+x^\ast x)$. By (\ref{ww2}) and (\ref{inverse semigroup}),
\begin{equation}\label{ww6}
-x=-(x+x(x^\ast+x^\ast x))=-x(x^\ast+x^\ast x)-x=x(x^\ast+x^\ast x)-x.
\end{equation}
This together with  (\ref{lw1}) yields that
\begin{equation}\label{ww3} xx^\ast-x=xx^\ast+x(x^\ast+x^\ast x)-x=x(x^\ast+x^\ast x)-x=-x.
\end{equation}
Since $xx^\ast\in P(S, \cdot)$, we have  $xx^\ast\in E(S,+)$ and so $-xx^\ast=xx^\ast$ by  (\ref{ww1}). In view of (\ref{ww3}),
\begin{equation}\label{ww4}   x=-(-x)=-(xx^\ast-x)=x-xx^\ast=x+xx^\ast.
\end{equation}
Replacing $x$ by $x^\ast$ in (\ref{ww3}) and (\ref{ww4}), we have
\begin{equation}\label{ww5}
x^\ast x-x^\ast=-x^\ast,\,\,\, x^\ast+x^\ast x=x^\ast.
\end{equation}
Since $y-y, x^\ast x\in E(S, +)$, we have $y-y+x^\ast x=x^\ast x+y-y$ by (\ref{ww1}). This together with  (\ref{ww5}) and (\ref{lw1}) implies that
$$x+x(x^\ast +y)=xx^\ast x+x(x^\ast +y)=x(x^\ast x+y)=x(x^\ast x+y-y+y)$$$$=x(y-y+x^\ast x+y)=x(y-y)+x(x^\ast+x^\ast x+y)=x(y-y)+x(x^\ast+y)=x(y-y+y)=xy$$
and so $x-x+xy=x-x+x+x(x^\ast+y)=x+x(x^\ast+y)=xy$. Thus we have
\begin{equation}\label{ww7}
x+x(x^\ast +y)=xy \mbox{ and }x-x+xy=xy \mbox{ for all } x,y\in S.
\end{equation}
Denote $u=-x^\ast+x^\ast$. Then $u\in E(S, +)$ and $uu^\ast\in P(S, \cdot)\subseteq E(S, +)$. By the fact that $(S, +)$ is inverse,  (\ref{inverse}), (\ref{ww4})  and the first identity in (\ref{ww5})  we have $$uu^\ast+uu^\ast=uu^\ast,\,\,\, uu^\ast+ u=u+uu^\ast=u,\,\,\, x^\ast x+u=x^\ast x-x^\ast +x^\ast=-x^\ast+x^\ast=u, $$$$x^\ast+u=x^\ast-x^\ast+x^\ast=x^\ast,\,\,\,
x^\ast+ uu^\ast=x^\ast +u+ uu^\ast=x^\ast+u=x^\ast.$$ This together with (\ref{lw1}) and (\ref{ww4}) implies that
$$xuu^\ast=x(uu^\ast+uu^\ast)=xuu^\ast+x(x^\ast+uu^\ast)=xuu^\ast+x(x^\ast+u)$$$$=x(uu^\ast+u)=xu=x(x^\ast x+u)=xx^\ast x+ x(x^\ast +u)=x+xx^\ast=x.$$
Thus
\begin{equation}\label{ww8}
x=xu=xuu^\ast=xu^\ast\overset{u=-x^\ast+x^\ast}=x(-x^\ast+x^\ast)^\ast \mbox{ for all } x \in S.
\end{equation}
Substituting $x$ by $x^\ast$ in (\ref{ww8}),  we have $$x^\ast=x^\ast(-x^{\ast\ast}+x^{\ast\ast})^\ast=x^\ast(-x+x)^\ast,\,\,\, x=x^{\ast\ast}=(x^\ast(-x+x)^\ast)^\ast=(-x+x)x.$$ This together with the second identity in (\ref{ww7}) yields that
\begin{equation}\label{ww9}
 \begin{array}{cc}
x=(-x+x)x=(-x+x)-(-x+x)+(-x+x)x=\\[2mm]
-x+x-x+x+ (-x+x)x=-x+x+(-x+x)x=-x+x+x.
\end{array}
\end{equation}
By (\ref{ww3}) and (\ref{ww9}), we have $$xx^\ast+x=xx^\ast-x+x+x=-x+x+x=x.$$
Since $xx^\ast=-xx^\ast$ by the fact that $xx^\ast\in P(S, \cdot)\subseteq E(S, +)$, we  have  $$-x=-(xx^\ast+x)=-x-xx^\ast=-x+xx^\ast.$$
By the second identity in (\ref{ww7}), it follows that $x-x=x-x+xx^\ast=xx^\ast\in P(S, \cdot) \mbox{ for all } x\in S.$
In particular, for every  $e\in E(S, +)$, we have $e=e-e\in  P(S, \cdot)$ as $(S, +)$ is an inverse semigroup.
This yields that $E(S, +)\subseteq P(S, \cdot)$ and so $E(S, +)= P(S, \cdot)$.

Let $e,f\in P(S, \cdot)=E(S, +)$.
By (\ref{lw1}) and (\ref{ww1}), we have
\begin{equation}\label{ww10}
 \begin{array}{cc}
ef=e(f+f)=ef+e(e^\ast +f)=ef+e(e+f)=ef+e(e+e+f)\\[2mm]
=ef+e(e+f+e)=ef+e(e^\ast+f+e)=e(f+f+e)=e(f+e)\\[2mm]
=ef+e(e^\ast +e)=ef+e(e+e)=ef+ee=ef+e.
\end{array}
\end{equation}
By (\ref{ww10})  we obtain that
\begin{equation}\label{ww11}
p(q+p)+p= pq+p=pq \mbox{ for all } p,q\in P(S, \cdot)=E(S, +).
\end{equation}
Since $(S, +)$ is inverse and $P(S, \cdot)=E(S, +)$, we have $$f+e=e+f\in E(S, +)=P(S, \cdot),\,\,\,  (e+f)(e+f)=(e+f)+(e+f)=e+f.$$
Take $p=e+f$ and $q=e$ in (\ref{ww11}). Then $$e+f=(e+f)+(e+f)=(e+f)(e+f)+(e+f)$$$$=(e+f)(e+e+f)+(e+f)\overset{(\ref{ww11})}=(e+f)e.$$ This together with (\ref{ww10}) and the fact that $e+f\in P(S, \cdot)$ implies that
$$e+f=(e+f)^\ast=((e+f)e)^\ast=e^\ast (e+f)^\ast=e(e+f)=e(f+e)=ef.$$  Dually, we have $f+e=fe$ and hence $ef=e+f=f+e=fe$. In view of Lemma \ref{zhengzexingbanqun5}, $(S, \cdot, \ast)$ forms an inverse semigroup.

Finally, substituting $x$ by $-x$ in (\ref{ww9}), we have $$-x=-(-x)+(-x)+(-x)=x-x-x,$$ and so $x=-(-x)=-(x-x-x)=x+x-x.$  Hence
$$-x+x=-x+x+x-x=x-x$$ by (\ref{ww9}). By Lemma \ref{zicu10}, $(S, +)$ is a Clifford semigroup as inverse semigroups are regular $\star$-semigroups.
Moreover, by the first identity in (\ref{ww7}), we have $x+x(x^\ast+y)=xy$, and so $$-x+xy=-x+x+x(x^\ast+y)=x-x+x(x^\ast+y)=x(x^\ast+y)$$ by
the second identity in (\ref{ww7}).
Up to now, we have known that $(S, +)$ is a Clifford semigroup, $(S, \cdot, \ast)$ is an inverse semigroup and $-x+xy=x(x^\ast+y)$ for all $x,y\in S$. By Proposition \ref{zuozhengze} and Remark \ref{5.3}, $(S, +, \cdot)$ forms a weak left brace.
\end{proof}

\begin{remark}\label{zuozhengze1}By Propositions \ref{zuozhengze} and \ref{zuozhengze2},  if $(S, \cdot)$ and $(S, +)$ are two inverse semigroups, then $(S,+,\cdot)$ is a weak left brace if and only if    $(S,+,\cdot)$ is a left inverse semibrace. However, the axiom $-x+xy=x(x^\ast+y)$ is necessary in Propositions \ref{zuozhengze} (see Remark \ref{ww12} below).
\end{remark}

To give more structural properties of weak left  $\star$-braces, we need a series of lemmas.
\begin{lem}\label{6.5} Let $(S,+,\cdot,-,\ast)$ be a weak left $\star$-brace. Then $xe(x^\ast +e)=x(x^\ast +e)$ for all  $x\in S$ and $e\in P(S)$.
\end{lem}
\begin{proof}Let $x\in S$ and $e\in P(S)$. By Lemma \ref{lw6} (9), we have $e+e=e=-e$. By using (\ref{lw5}) several times, we obtain that
$$x(e+x^\ast x+ e)=x(e+x^\ast x)-x+xe=xe-x+xx^\ast x-x+xe$$$$=xe-x+x-x+xe=xe-x+xe=x(e+e)=xe,$$ which together with  (\ref{lw5}) implies that
$$\underline{xe}(x^\ast+e)=\underline{x(e+x^\ast x+ e)} (x^\ast+e)=x(-e-x^\ast +x^\ast +e) (x^\ast+e)$$$$=x (-(x^\ast+e) +x^\ast +e) (x^\ast+e)=x(x^\ast+e) (x^\ast+e)^\ast (x^\ast+e)=x(x^\ast+e).$$
Thus the desired result follows.
\end{proof}

\begin{lem}\label{liwei} Let $(S,+,\cdot,-,\ast)$ be a weak left $\star$-brace, $x, y\in S$ and $e,f\in P(S)$.  Then
$$ -x+xyy^\ast=(xy)(xy)^\ast\in P(S), \,\,\,
e+ef=e+e(f+e), \,\,\,
-x+xe=-xe+x. $$
\end{lem}
\begin{proof}Firstly, we have
\begin{eqnarray*}
&& -x+xyy^\ast=x(x^\ast+yy^\ast)\,\,\,\, \,\,\,\, (\mbox{by Lemma \ref{lw6} (10)})\\
&=&xyy^\ast(x^\ast+yy^\ast)\,\,\,\, \,\,\,\, (\mbox{by Lemma \ref{6.5} and the fact } yy^\ast\in P(S))\\
&=&xyy^\ast x^\ast-xyy^\ast+xyy^\ast yy^\ast=xyy^\ast x^\ast-xyy^\ast+xyy^\ast\,\,\,\, \,\,\,\, (\mbox{by (\ref{lw5})})\\
&=&xyy^\ast x^\ast +xyy^\ast (xyy^\ast)^\ast\,\,\,\, \,\,\,\, (\mbox{by (\ref{lw5})})\\
&=&xyy^\ast x^\ast +xyy^\ast x^\ast=(xy)(xy)^\ast+(xy)(xy)^\ast=(xy)(xy)^\ast.\,\,\,\, \,\,\,\, (\mbox{by Lemma \ref{lw6} (3)})
\end{eqnarray*}
By Remark \ref{toushejiyizhi}, $-x+xyy^\ast=(xy)(xy)^\ast\in P(S)$.   Secondly, by the first identity in the lemma and Lemma \ref{lw6} (9), we have
$e+ef=-e+eff=-e+eff^\ast\in P(S)$, and so $e+ef=-(e+ef)=-ef-e=-ef+e$. This implies that $$e+e(f+e)=\underline{e+ef}-e+ee=\underline{-ef+e}-e+ee=-ef+e=e+ef$$ by (\ref{lw5}) and  Lemma \ref{lw6} (9).   Finally, by the first identity in the lemma and Lemma \ref{lw6} (9), we have $-x+xe=-x+xee=-x+xee^\ast\in P(S)$, and hence $-x+xe=-(-x+xe)=-xe+x$.
\end{proof}

Now, we can give our key lemma in this section.
\begin{lem}\label{jiaohuan}Let $(S,+,\cdot,-,\ast)$ be a weak left $\star$-brace,  $x\in S$ and $e\in P(S)$. Then $ex=x+e$.
In particular, we have $(-x+x)(x-x) =x-x-x+x.$
\end{lem}
\begin{proof}By the first identity in Lemma \ref{liwei},  (\ref{lw5}) and Lemma \ref{lw6} (9),  we have
\begin{equation}\label{linshi1}
\begin{array}{cccc}
 e+exx^\ast =-e+exx^\ast =(ex)(ex)^\ast, \,\,\,\,\,\,
 (e+exx^\ast)ex=(ex)(ex)^\ast ex=ex,\\[3mm]
  e(x+e)=ex-e+ee=ex+e,\\[3mm]
  (x+e)(x+e)^\ast=-(x+e)+(x+e)=-e-x+x+e=e+xx^\ast +e.
\end{array}
\end{equation}
On one hand, replacing $x$ by $x+e$  in the second identity of (\ref{linshi1}), we have
\begin{eqnarray*}
&&e(x+e)=(e+e(x+e)(x+e)^\ast)e(x+e)\\
&=&(e+e(e+xx^\ast+e)) e(x+e)\,\, \,\,\, \,(\mbox{by the fourth identity in (\ref{linshi1})})\\
&=&(e+ee-e+e(xx^\ast+e)) e(x+e)\,\, \,\,\, \,(\mbox{by (\ref{lw5})})\\
&=&(e+e(xx^\ast+e))e(x+e)\,\, \,\,\, \,(\mbox{by  Lemma \ref{lw6} (9)})\\
&=&(e+exx^\ast)e(x+e)\,\, \,\,\, \,(\mbox{by  the second identity in Lemma \ref{liwei} and } e, xx^\ast\in P(S))\\
&=&(ex)(ex)^\ast (ex+e)\,\, \,\,\, \,(\mbox{by the first and third identities in (\ref{linshi1})})\\
&=&ex(ex)^\ast ex-ex(ex)^\ast+ex(ex)^\ast e=ex-ex(ex)^\ast+ exx^\ast e^\ast  e   \,\,\,\, (\mbox{by (\ref{lw5})})\\
&=&ex-ex(ex)^\ast+ exx^\ast e^\ast=ex-ex(ex)^\ast+ex(ex)^\ast=ex.\,\,\,\,  (\mbox{by Lemma \ref{lw6} (9) and (5)})
\end{eqnarray*}
On the other hand,
\begin{eqnarray*}
&&(x+e)^\ast e=(x+e)^\ast(e+e)\,\,\,\,\, (\mbox{since }e+e=e \mbox{ by Lemma \ref{lw6} (9)})\\
&=&(x+e)^\ast e \underline{-(x+e)^\ast+(x+e)^\ast e}\,\,\,\,\,\, (\mbox{by }(\ref{lw5}))\\
&=&(x+e)^\ast(-y)\underline{-(x+e)^\ast e+(x+e)^\ast}\,\,\,\,\,\, (\mbox{by  the third identity in Lemma \ref{liwei}})\\
&=&(x+e)^\ast e-(x+e)^\ast e+(x+e)^\ast\underbrace{(x+e)(x+e)^\ast}\\
&=&(x+e)^\ast e-(x+e)^\ast e+(x+e)^\ast\underbrace{(e+xx^\ast +e)}\,\,\,\,\,\, (\mbox{by }(\ref{linshi1}))\\
&=&(x+e)^\ast \underbrace{(e+xx^\ast +e)}\,\,\,\,\,\,\, (\mbox{by Lemma \ref{lw6} (8)})\\
&=& (x+e)^\ast\underbrace{(x+e)(x+e)^\ast}=(x+e)^\ast\,\,\,\,\,\, (\mbox{by }(\ref{linshi1}))
\end{eqnarray*}
Thus we have $$ex=e(x+e)=e^\ast (x+e)=((x+e)^\ast e)^\ast=((x+e)^\ast)^\ast=x+e$$ by Lemma \ref{lw6} (9).
The final result follows from the fact that $x-x, -x+x\in P(S)$.
\end{proof}
The following theorem gives some necessary and sufficient conditions under which a weak left  $\star$-brace  becomes  a weak left brace.
\begin{theorem}For a weak left $\star$-brace $(S,+,\cdot,-,\ast)$, the following conditions are equivalent:
\begin{itemize}
\item[(1)]$(S,\cdot,\ast)$ is an inverse semigroup.
\item[(2)]$(S,+,-)$ is an inverse semigroup.
\item[(3)]$(S,+,\cdot,\, -,\ast)$ is a weak left brace.
\item[(4)]$(S,+,\cdot,\, -,\ast)$ satisfies the axiom $a(-b)=a-ab+a$.
\item[(5)]$(S,+,\cdot,\, -,\ast)$ satisfies the axiom $ab=a+a(a^{\ast}+b)$.
\item[(6)]$(S,+,\cdot,\, -,\ast)$ satisfies the axiom $a+b=aa^{\ast}(a+b)$.
\item[(7)]$(S,+,\cdot,\, -,\ast)$ satisfies the axiom $a(a^\ast+a^\ast)=-a$.
\item[(8)]$(S,+,\cdot,\, -,\ast)$ satisfies the axiom $b^\ast(a^\ast+b)=b^{\ast}a^{\ast}-b^{\ast}$.
\end{itemize}
\end{theorem}
\begin{proof}
Firstly, let $e,f\in P(S)$. By Lemma \ref{jiaohuan}, we have $ef=f+e$, and so $ef=fe$ if and only if $f+e=e+f$. According to Lemma \ref{zhengzexingbanqun5}, we obtain that (1) is equivalent to (2), and so items (1), (2) and (3) are equivalent mutually.

Secondly, assume that (3) holds. Then item (6)  is true by \cite[Proposition 3]{Catino-Mazzotta-Stefanelli2}.   Moreover, by   \cite[Lemmas 1 and 2, and Proposition 9 (3)]{Catino-Mazzotta-Miccoli-Stefanelli}, we can obtain items (4), (5), (7) and (8).

 In the following statements, we shall prove that (4) implies (1),  (5) implies (1),  (6) implies (1),  (7) implies (1) and  (8) implies (1), respectively.

(4) $\Longrightarrow$ (1). Assume that (4) holds and $e,f\in P(S, \cdot)=P(S)$. Then by Lemma \ref{lw6} (9), we have $-e=e+e=e=e-e$ and $f=-f$. Moreover, Lemma \ref{jiaohuan} gives that $$ef=f+e,\,\, e+f+e=(e+f)+e=e(e+f)=e(fe)=efe.$$  This together with the axiom given in item (4) and Lemma \ref{1.2} (1) implies that
$$ef=e(-f)\overset{(4)}=e-ef+e=e-(f+e)+e=e-e-f+e=e+f+e=efe\in P(S, \cdot),$$  and so $ef=fe$ by  Lemma \ref{1.2} (2). By Lemma \ref{zhengzexingbanqun5}, $(S,\cdot,\ast)$ is an inverse semigroup.

(5) $\Longrightarrow$  (1).   Assume that (5) holds and $e,f\in P(S, \cdot)=P(S)$.  Then by the axiom in item (5), Lemma \ref{lw6} (9) and Lemma \ref{jiaohuan}, we have
$$ef\overset{(5)}{=}e+e(e^{\ast}+f)=e+e(e+f)=e+(e+f)+e=e+f+e=efe\in P(S, \cdot),$$and so $ef=fe$ by  Lemma \ref{1.2} (2). By Lemma \ref{zhengzexingbanqun5}, $(S,\cdot,\ast)$ is an inverse semigroup.

(6) $\Longrightarrow$  (1).   Assume that (6) holds and $e,f\in P(S, \cdot)=P(S)$. Then by Lemma \ref{jiaohuan}, the axiom in item (6) and Lemma \ref{lw6} (9), we have
$$ef=f+e\overset{(6)}{=}ff^{\ast}(f+e)=f(f+e)=(f+e)+f=fef\in P(S, \cdot), $$  and so $ef=fe$ by  Lemma \ref{1.2} (2). By Lemma \ref{zhengzexingbanqun5}, $(S,\cdot,\ast)$ is an inverse semigroup.

(7)$\Longrightarrow$  (1).  Assume that (7) holds and $e,f\in P(S, \cdot)=P(S)$. Then by  Lemma \ref{jiaohuan},  the axiom in item (7), Lemmas \ref{lw6} (9) and   \ref{1.2} (4), we have
$$ef=f+e=-f-e=-(e+f)=-fe\overset{(7)}=fe((fe)^{\ast}+(fe)^{\ast})=fe(e^\ast f^\ast+e^\ast f^\ast)$$$$=fe(ef+ef)=fe((f+e)+(f+e)) =f(((f+e)+(f+e))+e)$$$$=(((f+e)+(f+e))+e)+f=f+e+f=fef\in P(S,\cdot), $$ and so $ef=fe$ by  Lemma \ref{1.2} (2). By Lemma \ref{zhengzexingbanqun5}, $(S,\cdot,\ast)$ is an inverse semigroup.

(8)$\Longrightarrow$  (1).  Assume that (8) holds.  Replacing $b$ by $a^{\ast}$ in the axiom given in (8),  we have
$$a(a^{\ast}+a^{\ast})=a^{\ast\ast}(a^{\ast}+a^{\ast})\overset{(8)}=a^{\ast\ast}a^{\ast}-a^{\ast\ast}=aa^{\ast}-a=-a$$ by  Lemmas \ref{lw6} (5).
This is exactly the axiom in item (7).  By the statements in the previous paragraph,  $(S,\cdot,\ast)$ is an inverse semigroup.
\end{proof}

\begin{lem}\label{lw12}Let $(S,+,\cdot,-,\ast)$ be a weak left  $\star$-brace and $x,y\in S$. Then
$$x-y+y+y=x+y,\,\,\,\, y+y-y+x=y+x.$$
\end{lem}
\begin{proof}Observe that $-y+y\in P(S)$, it follows that
\begin{eqnarray*}
&&x+y=x+y+(-y+y)=(-y+y)(x+y)\,\,\,\,\,\,\, (\mbox{by Lemma }\ref{jiaohuan})\\
&=&(-y+y)x-(-y+y)+(-y+y)y\,\,\,\,\,\, (\mbox{by }(\ref{lw5}))\\
&=&x-y+y-y+y+y-y+y=x-y+y+y.\,\,\,\,\,\,\, (\mbox{by Lemma }\ref{jiaohuan})
\end{eqnarray*}
This implies that $$-y-y+y-x=-(x-y+y+y)=-(x+y)=-y-x.$$ Replacing $x,y$ by $-x, -y$ respectively, we have $y+y-y+x=y+x$.
\end{proof}

The following result gives the structures of the additive semigroup and multiplicative semigroup of a weak left $\star$-brace, which generalizes \cite[Theorem 8]{Catino-Mazzotta-Miccoli-Stefanelli}.
\begin{theorem}\label{zicu9tiaoxingzhi}Let $(S,+,\cdot,-,\ast)$ be a weak left $\star$-brace. Then $(S, +, -)$ is completely regular, orthodox and locally inverse, and $(S, \cdot, \ast)$ is orthodox and locally inverse.
\end{theorem}
\begin{proof}
Lemma \ref{zicu9} (3) and the first identity in Lemma \ref{lw12} together imply that $(S, +, -)$ is completely regular, orthodox and locally inverse.   Let $e,f,g,h\in P(S)=P(S, \cdot)=P(S, +)$.
Then by Lemma \ref{zicu8}, we have $h+g+f+e=h+f+g+e$.  In view of Lemma \ref{jiaohuan}, we obtain that $$efgh=(fgh)+e=(gh)+f+e=h+g+f+e=h+f+g+e=egfh.$$
Thus $(S, \cdot, \ast)$ is orthodox and locally inverse by Lemma \ref{zicu8} and its proof.
\end{proof}
\begin{remark}\label{ww12}
Let $(S,\cdot, \ast)$ and $(S,+, -)$ be two  regular $\star$-semigroups and  $(S,+,\cdot,\, \ast)$ be  a left regular $\star$-semibrace such that $(S, +, -)$ is completely regular, orthodox and locally inverse, and $(S, \cdot, \ast)$ is orthodox and locally inverse.  Then $(S,+,\cdot, -,\ast)$ may not be  a weak left $\star$-brace.  For example, let $S=\{1,2\}\times \{1,2\}$ and denote $e=(1, 1),  g=(1, 2), h=(2, 1), f=(2,2)$.
Then $S=\{e, f, g, h\}$. Define $``+"$ and $``-"$ on $S$ as follows: $$\mbox{For all } (i,j), (k,l)\in S,\,\,\, (i,j)+(k,l)=(i,l), \,\,\,\,  -(i,j)=(j,i).$$
Then $(S, +, -)$ forms a regular $\star$-semigroup and $$-e=e, -f=f, -g=h, -h=g, f+e=h, f+f=f.$$ Moreover, by Lemma \ref{zicu9},  one can easily check that $(S, +, -)$ is completely regular, orthodox and locally inverse. On the other hand, let $(S, \cdot)$ be the   Klein  four-group with the identity $e$ and define $x^\ast=x$ for all $x\in S$. Then $ff=e$.
Obviously, $(S, \cdot, \ast)$ forms an orthodox and locally inverse  regular $\star$-semigroup by Lemma \ref{zicu8}. Moreover, one can check that the axiom (\ref{lw1}) holds routinely. Thus $(S,+,\cdot,\, \ast)$ is  a left regular $\star$-semibrace. In fact, this is a left semibrace.   However, we have $$-f+ff=f+e=h\not=e=ff=f(f+f)=f(f^\ast +f).$$ This implies that  $(S,+,\cdot, -,\ast)$  is not a weak left $\star$-brace by Proposition \ref{zuozhengze}.
\end{remark}

The following corollary generalizes \cite[Proposition 6]{Catino-Mazzotta-Miccoli-Stefanelli}.
\begin{coro}\label{danweiyuanyizhi}Let $(S,+,\cdot,-,\ast)$ be a weak left $\star$-brace. Then $(S,+,-)$ is a monoid if and only if $(S,\cdot,\ast)$ is a monoid.
\end{coro}
\begin{proof} Let $(S,+,-)$ be a monoid and $0$ be the identity, $x\in S$. Then $0, x^\ast x\in P(S)$ by Lemma \ref{danweiyuan} and Remark \ref{tousheji}. According to Lemma \ref{jiaohuan},  we have $$0\cdot x=x+0=x,\,\, x\cdot 0=x (x^\ast x\cdot 0)=x(0+x^\ast x)=xx^\ast x=x.$$   This implies that $0$ is the identity in $(S,\cdot,\ast)$.
Conversely, let $(S,\cdot,\ast)$ be a monoid and $1$ be the identity, $x\in S$. Then Lemma \ref{danweiyuan} and Remark \ref{tousheji} imply that $1\in P(S)$ and $x-x\in P(S)$. By Lemma \ref{jiaohuan}, we have
$$x+1=1\cdot x=x,\,\,1+x=(1+x-x)+x=(x-x)\cdot 1+x=x-x+x=x.$$ This implies that $1$ is the identity in $(S,+,-)$.
\end{proof}
To prove that the map  associated to a weak left $\star$-brace  is   a solution  of the Yang-Baxter equation, we need more lemmas.
\begin{lem}\label{6.8}Let $(S,+,\cdot,-,\ast)$ be a weak left $\star$-brace and $x,y\in S$. Then $x(y-y)y=xy$.
\end{lem}
\begin{proof}
By Lemma \ref{jiaohuan}, the fact $yy^\ast\in P(S)$ and (\ref{lw5}), we have
$$y^\ast x^\ast=y^\ast( yy^\ast x^\ast)=y^\ast(x^\ast + yy^\ast)=y^\ast x^\ast -y^\ast +y^\ast yy^\ast$$$$=y^\ast x^\ast -y^\ast +y^\ast
=y^\ast x^\ast+y^\ast y^{\ast\ast} =y^\ast x^\ast+y^\ast y.$$
Using the second identity in Lemma \ref{lw12},  we have $y=y-y+y=y+y-y-y+y.$
By Proposition  \ref{zuozhengze}, (\ref{lw1}), Lemma \ref{jiaohuan}, the fact $y-y, -y+y\in P(S)$ and (\ref{lw5}) in that order,
$$y^\ast x^\ast+y^\ast y=y^\ast x^\ast+y^\ast (y+y-y-y+y)=y^\ast(x^\ast+y-y-y+y)$$$$=y^\ast((-y+y)(y-y)x^\ast)=y^\ast yy^\ast (y-y)x^\ast=y^\ast (y-y)x^\ast.$$
Thus, we have $y^\ast x^\ast=y^\ast (y-y)x^\ast$. Since $y-y\in P(S)$, we have $(y-y)^\ast=y-y$ by Lemma \ref{lw6} (9).
So $xy=(y^\ast x^\ast)^\ast=(y^\ast (y-y)x^\ast)^\ast=x^{\ast\ast}(y-y)^\ast y^{\ast\ast}=x(y-y)y$.
\end{proof}
\begin{lem}Let $(S,+,\cdot,-,\ast)$ be a weak left $\star$-brace and $x,y\in S, e\in P(S)$. Then
\begin{equation}\label{6.9}
x(-y+y)z=x(y-y)z,
\end{equation}
\begin{equation}\label{6.11}
x+e+z+e=x+z+e,
\end{equation}
\begin{equation}\label{6.12}
(-x+x)e(-x+x)=-x+e+x.
\end{equation}
 \end{lem}
\begin{proof}Firstly,
\begin{eqnarray*}
&&x(-y+y)z =xyy^\ast z=(xy)(z^{\ast}y)^{\ast}\,\,\,\,\,\,\,(\mbox{by }-y+y=yy^\ast \mbox{ in } (\ref{lw5}))\\
&=&x(y-y)y(z^{\ast}(y-y)y)^{\ast}\,\,\,\,\,\,\,\mbox{(by Lemma \ref{6.8})}\\
&=&x(y-y)yy^{\ast}(y-y)z\,\,\,\,\,\,\,(\mbox{since }(y-y)^\ast=y-y \mbox{ by Lemma \ref{lw6} (4)})\\
&=&(x(y-y))(-y+y)(-y)((-y)^{\ast}z)\\
&&  (\mbox{since }yy^\ast=-y+y,\, (y-y)(y-y)=y-y=(-y)(-y)^\ast\mbox{ by Lemma \ref{lw6} (4), (6)})\\
&=&(x(y-y))(-y)((-y)^{\ast}z)\,\,\,\,\,\,\,(\mbox{by Lemma \ref{6.8} and }-y+y=-y-(-y))\\
&=&x(y-y)(y-y)z=x(y-y)z. \,\, (\mbox{since }(y-y)(y-y)=y-y=(-y)(-y)^\ast)
\end{eqnarray*}
Thus (\ref{6.9}) holds. Since $e\in P(S)$, we have $e-e=e$ by Lemma \ref{lw6} (9). In view of  Lemma \ref{jiaohuan} and (\ref{lw5}), we obtain that
$$x+z+e=e(x+z)=ex-e+ez=x+e-e+ez=x+e+z+e,$$
which implies that (\ref{6.11}) holds. Finally, according to the fact $-x+x, e\in P(S)$, Lemma \ref{jiaohuan} and (\ref{6.11}), we have
$$(-x+x)e(-x+x)=(-x+x+e)-x+x=(-x+e+x+e)-x+x=-x+e+(x+e-x)+x.$$
The fact $e\in P(S)$  and Lemma \ref{1.2} (1) imply that $x+e-x\in P(S)$. On the other hand, Theorem \ref{zicu9tiaoxingzhi} gives that $(S, +, -)$ is completely regular, orthodox and locally inverse. By Lemma \ref{zicu8} and (\ref{6.11}) we have
$$-x+(e+(x+e-x))+x=-x+ {((x+e-x)+e)+x} \overset{(\ref{6.11})}=-x+ {(x-x+e)+x}=-x+e+x.$$
Thus $(-x+x)e(-x+x)=-x+e+x.$ That is, (\ref{6.12}) holds.
\end{proof}

\begin{lem}\label{6.13.1}Let $(S,+,\cdot,-,\ast)$ be a weak left $\star$-brace. Then $xyy^\ast(x^\ast+y)=x(x^\ast +y)$ and
$(x +y)^{\ast}z=(x +y)^\ast xx^{\ast}z$ for all $x, y, z\in S$.
\end{lem}

\begin{proof}Let $x,y,z\in S$.  Then by Lemma \ref{jiaohuan}, the fact $yy^\ast\in P(S)$ and (\ref{lw5}), we have $$xyy^\ast(x^\ast+y)=x(x^\ast +y+yy^\ast)=x(x^\ast +y-y+y)=x(x^\ast +y).$$  On the other hand, again by (\ref{lw5}) we obtain
$$(x +y)^{\ast}y=(x +y)^{\ast}(x +y)(x +y)^{\ast}y=(x +y)^{\ast}(-(x +y)+(x +y))y$$$$=(x +y)^{\ast}(-y-x +x
+y)y=(x +y)^{\ast}(-y+xx^{\ast}+y)y.$$ By (\ref{6.12}) and the facts $x x^\ast \in P(S)$ and $-y+y=yy^\ast$, we have
$$(x +y)^{\ast}(-y+xx^{\ast}+y)y=(x +y)^{\ast}(-y+y)xx^{\ast}(-y+y)y=(x +y)^{\ast}yy^\ast xx^{\ast}yy^\ast y.$$
Theorem \ref{zicu9tiaoxingzhi} gives that $(S, \cdot, \ast)$ is orthodox and locally inverse. Since $yy^\ast, x x^\ast\in P(S)$,  we have $(x +y)^{\ast}yy^\ast xx^{\ast}yy^\ast y = (x +y)^{\ast} xx^{\ast} yy^\ast yy^\ast y=(x +y)^{\ast} xx^{\ast} y$ by Lemma \ref{zicu8}, whence
\begin{equation}\label{6.13}
(x +y)^{\ast}y=(x +y)^\ast xx^{\ast}y.
\end{equation}This implies that
\begin{eqnarray*}
&&(x +y)^\ast xx^{\ast}z=(x +y-y+y)^\ast xx^{\ast}z=((x +y)+yy^{\ast})^{\ast}xx^{\ast}z\,\,(\mbox{since }-y+y=yy^\ast)\\
&=&(yy^{\ast}(x +y))^{\ast}xx^{\ast}z=(x +y)^{\ast}yy^{\ast}xx^{\ast}z\,\,\,\,\,\,\,\mbox{(by Lemma \ref{jiaohuan} and }yy^\ast\in P(S))\\
&=&(x +y)^{\ast}xx^{\ast}yy^{\ast}z\,\,\,\,\,\,\,\mbox{(by Lemma \ref{zicu8} and Theorem \ref{zicu9tiaoxingzhi}})\\
&=&(x +y)^{\ast}yy^{\ast}z=(yy^{\ast}(x^{\ast}+y))^{\ast}z\,\,\,\,\,\,\,\mbox{(by (\ref{6.13}))}\\
&=&((x +y)+yy^{\ast})^{\ast}z\,\,\,\,\,\,\,\mbox{(by Lemma \ref{jiaohuan} and }yy^\ast\in P(S))\\
&=&((x +y)-y+y)^{\ast}z=(x +y)^{\ast}z.\,\,\,\,\,\,\,(\mbox{since } -y+y=yy^\ast)
\end{eqnarray*}
Thus the desired result is true.
\end{proof}

Let $(S,+,\cdot,-,\ast)$ be a weak left $\star$-brace and $x,y\in S$. Recall that
$$ \lambda_{x}: S\rightarrow S,\, y\mapsto x(x^\ast +y),\,\, \rho_{y}: S\rightarrow S,\, x\mapsto (x^\ast+y)^\ast y.$$
\begin{theorem}\label{zhongyao}Let $(S,+,\cdot,-,\ast)$ be a weak $\star$-brace and   $x,y,z\in S$.
\begin{itemize}
\item[(1)]$\lambda_x(y+z)=\lambda_{x}(y)+\lambda_{x}(z)$.
\item[(2)]$\lambda_{x}\lambda_{y}(z)=\lambda_{xy}(z)$.
\item[(3)]$\rho_{z}\rho_{x}(y)=\rho_{xz}(y)$.
\end{itemize}
\end{theorem}
\begin{proof} (1) By Proposition \ref{zuozhengze} and (\ref{lw1}), we have
$$\lambda_{x}(y)+\lambda_{x}(z)=x(x^\ast +y)+x(x^\ast+z)=x(x^\ast +y+z)=\lambda_x(y+z).$$

(2) By (\ref{lw5}), the fact $yy^\ast\in P(S)$ and Lemma \ref{jiaohuan},
$$\lambda_{x}\lambda_{y}(z) =x(x^{\ast}+y(y^{\ast}+z))=xx^{\ast}-x+xy(y^{\ast}+z)$$$$ =xx^{\ast}-x+xyy^{\ast}-xy+xyz
=x(x^{\ast}+yy^\ast)-xy+xyz$$$$=x (yy^\ast x^{\ast})-xy+xyz=xy(xy)^\ast -xy+xyz=xy((xy)^\ast +z)=\lambda_{xy}(z).$$

(3) In fact, we have
\begin{eqnarray*}
&&\rho_{z}\rho_{x}(y)
=(x^{\ast}(y^{\ast}+x)+z)^{\ast}z\\
&=&(x^{\ast}(y^{\ast}+x)+z)^{\ast}(x^{\ast}(y^{\ast}+x))(x^{\ast}(y^{\ast}+x))^\ast z\,\,\,\,\,\,\,\mbox{(by Lemma \ref{6.13.1})}\\
&=&(x(x^{\ast}(y^{\ast}+x)+z))^{\ast}(y^{\ast}+x)(y^{\ast}+x)^{\ast}xz\\
&=&(xx^{\ast}(y^{\ast}+x)-x+xz)^{\ast}(y^{\ast}+x)(y^{\ast}+x)^{\ast}xz\,\,\,\,\,\,\,\mbox{(by (\ref{lw5}))}\\
&=&((y^{\ast}+x)+xx^{\ast}-x+xz)^{\ast}(y^{\ast}+x)(y^{\ast}+x)^{\ast}xz\,\,\,\,\,\,\,\mbox{(by Lemma \ref{jiaohuan} and }xx^\ast\in P(S))\\
&=&((y^{\ast}+x)-x+xz)^{\ast}(y^{\ast}+x)(y^{\ast}+x)^{\ast}xz\,\, \,\, (\mbox{by } aa^\ast-a=-a \mbox{ in Lemma  \ref{lw6} (5)} )\\
&=&(\underline{y^{\ast}+x}-x+xz)^{\ast}xz\,\,\,\,\,\,\,\mbox{(by Lemma \ref{6.13.1})}\\
&=&(\underline{y^{\ast}-x+x+x}-x+xz)^{\ast}xz\,\,\,\,\,\,\,\mbox{(by the first axiom in Lemma \ref{lw12})}\\
&=&((y^{\ast}-x)+(x+x-x+xz))^{\ast}xz\\
&=&((y^{\ast}-x)+(x+xz))^{\ast}xz\,\,\,\,\,\,\,\mbox{(by the second axiom in Lemma \ref{lw12})}\\
&=&((y^{\ast}-x+x-x)+(x+xz))^{\ast}xz\,\,\,\,\,\,\, (\mbox{since }-x+x-x=-x)\\
&=&(((-x+x)y^{\ast})-x+x+xz)^{\ast}xz\,\,\,\,\,\,\,(\mbox{by Lemma \ref{jiaohuan} and }-x+x\in P(S))\\
&=&(((-x+x)y^{\ast})-(-x+x)+(-x+x)xz)^{\ast}xz\, \,\, \,\,\,\, (\mbox{by }(-a+a)a=a \mbox{ in Lemma \ref{lw6} (6)}))\\
&=&((-x+x)(y^{\ast}+xz))^{\ast}xz\,\,\,\,\,\,\, \mbox{(by (\ref{lw5}))}\\
&=&((y^{\ast}+xz)-x+x)^{\ast}xz=(y^{\ast}+(-x+x)xz)^{\ast}xz \,(\mbox{by Lemma \ref{jiaohuan} and }-x+x\in P(S))\\
&=&(y^{\ast}+xz)^{\ast}xz=\rho_{xz}(y).\, \,\, \,\,\,\, (\mbox{by  }(-a+a)a=a \mbox{ in Lemma \ref{lw6} (6)}))
\end{eqnarray*}
Thus item (3) holds.
\end{proof}

\begin{lem}\label{tebei}Let $(S,+,\cdot,-,\ast)$ be a weak left $\star$-brace, $x,y\in S$. Then $xy=\lambda_{x}(y)\rho_{y}(x)$.
\end{lem}
\begin{proof}Using the two axioms in (\ref{lw5}) repeatedly, we have
$$\lambda_{x}(y)\rho_{y}(x)=x(x^{\ast}+y)(x^{\ast}+y)^{\ast}y=x(-(x^{\ast}+y)+(x^{\ast}+y))y=x(-y-x^{\ast}+x^{\ast}+y)y$$$$=x(-y+x^{\ast}x+y)y=
(x(-y+x^{\ast}x)-x+xy)y=((x(-y)-x+xx^{\ast}x)-x+xy)y$$$$=(x(-y)-x+x-x+xy)y=(x(-y)-x+xy)y=(x(-y+y))y=xyy^\ast y=xy.$$
Thus, the result follows.
\end{proof}

Let $(S,+,\cdot,-,\ast)$ be a weak left $\star$-brace. Recall that the map $r_S$ associated to $(S,+,\cdot,\, -,\ast)$ is defined as follows: For all $x,y\in S$,
$$r_S(x,y)=(\lambda_{x}(y),\rho_{y}(x))=(x(x^{\ast}+y),(x^{\ast}+y)^{\ast}y).$$
Now we can state the main result of this section.
\begin{theorem}Let $(S,+,\cdot,-,\ast)$ be a weak left $\star$-brace. Then $r_S$ is a solution.
\end{theorem}
\begin{proof}We only need to show that the following three identities hold for all $x,y,z\in S$:
$$({\rm i})\,  \lambda_{x}\lambda_{y}(z)=\lambda_{\lambda_{x}(y)}\lambda_{\rho_{y}(x)}(z), ({\rm ii})\,  \rho_{z}\rho_{y}(x)=\rho_{\rho_{z}(y)}\rho_{\lambda_{y}(z)}(x),  ({\rm iii})\,  \lambda_{\rho_{\lambda_{y}(z)}(x)}\rho_{z}(y)=\rho_{\lambda_{\rho_{y}(x)}(z)}\lambda_{x}(y).$$
In fact, Theorem \ref{zhongyao} (2) and Lemma \ref{tebei} give item (i), and Theorem \ref{zhongyao} (3) and Lemma  \ref{tebei} deduce item (ii). We shall show item (iii) in the sequel. By Lemma  \ref{6.13.1}, for all $a, b\in S$,
$$\lambda_{a}(b)=a(a^{\ast}+b)=abb^{\ast}(a^{\ast}+b),\,\,\,  \rho_{b}(a)=(a^{\ast}+b)^{\ast}b=(a^{\ast}+b)^{\ast}a^{\ast}ab$$
\begin{equation}\label{lmabda-rho2}
\lambda_{a}(b)=ab(\rho_{b}(a))^{\ast},\,\,\, \rho_{b}(a)=(\lambda_{a}(b))^{\ast}ab.
\end{equation}
Let $x,y,z\in S$. Then
\begin{eqnarray*}
&&\lambda_{\rho_{\lambda_{y}(z)}(x)}(\rho_{z}(y))=\underline{\rho_{\lambda_{y}(z)}(x)}\cdot \rho_{z}(y)\cdot [\rho_{\rho_{z}(y)}(\rho_{\lambda_{y}(z)}(x))]^{\ast} \,\,(\mbox{by the first identity in (\ref{lmabda-rho2})})\\
&=&\underline{[\lambda_{x}(\lambda_{y}(z))]^{\ast}\cdot x\cdot \lambda_{y}(z)}\cdot \rho_{z}(y)\cdot [\rho_{\rho_{z}(y)}(\rho_{\lambda_{y}(z)}(x))]^{\ast} \,\,(\mbox{by the second identity in (\ref{lmabda-rho2})})\\
&=&[\lambda_{x}(\lambda_{y}(z))]^{\ast}\cdot x\cdot \underbrace{\lambda_{y}(z)\cdot \rho_{z}(y)}\cdot [\rho_{z}(\rho_{y}(x))]^{\ast}\,\,\, \,\,\,\,(\mbox{by (ii)})\\
&=&[\lambda_{x}(\lambda_{y}(z))]^{\ast}\cdot x\cdot \underbrace{yz}\cdot [\rho_{z}(\rho_{y}(x))]^{\ast}=[\lambda_{x}(\lambda_{y}(z))]^{\ast}\cdot\underline{xy}\cdot z\cdot [\rho_{z}(\rho_{y}(x))]^{\ast} \,(\mbox{by Lemma \ref{tebei}})\\
&=&[\lambda_{x}(\lambda_{y}(z))]^{\ast}\cdot\underline{\lambda_{x}(y)\cdot  \rho_{y}(x)}\cdot z\cdot [\rho_{z}(\rho_{y}(x))]^{\ast} \,\,\,\,\,\, \,(\mbox{by Lemma \ref{tebei}})\\
&=&\overbrace{[\lambda_{x}(\lambda_{y}(z))]^{\ast}}\cdot \lambda_{x}(y)\cdot  \underbrace{\rho_{y}(x) \cdot z\cdot [\rho_{z}(\rho_{y}(x))]^{\ast}}\\
&=&\overbrace{[\lambda_{\lambda_{x}(y)}(\lambda_{\rho_{y}(x)}(z))]^{\ast}}\cdot \lambda_{x}(y)\cdot  \underbrace{\lambda_{\rho_{y}{(x)}}(z)}
\,\,\,\,\,\,(\mbox{by (i) and  the first identity in (\ref{lmabda-rho2})})\\
&=&\rho_{\lambda_{\rho_{y}(x)}(z)}(\lambda_{x}(y)).\,\,(\mbox{by the second identity in (\ref{lmabda-rho2})})
\end{eqnarray*}
This proves that (iii) holds.
\end{proof}

\begin{remark}
In \cite{Catino-Mazzotta-Stefanelli2}, Catino,  Mazzotta and Stefanelli have established the Rota-Baxter operator theory of Clifford semigroups and explored the relationship between this theory and dual weak left braces. Moreover, the construction of dual weak left braces and the solutions of Yang-Baxter equation associated to dual weak left braces are also considered in \cite{Catino-Mazzotta-Stefanelli3,Mazzotta-Rybolowicz-Stefanelli}.
Thus, the following question is natural: How to establish a theory parallel to the theory in \cite{Catino-Mazzotta-Stefanelli2,Catino-Mazzotta-Stefanelli3,Mazzotta-Rybolowicz-Stefanelli} in the class of weak left $\star$-braces? We shall continue to study this problem in a separate paper.
\end{remark}
\noindent {\bf Acknowledgment:}   The authors express  their profound gratitude to Professor Li Guo at Rutgers University for his encouragement and help, and acknowledge  the assistance of Prover9 and Mace4 developed by McCune \cite{McCune} in the course of preparing this article. The paper is supported partially  by the Nature Science Foundations of  China (12271442, 11661082).

\end{document}